\documentclass[12pt,reqno]{amsart}

\usepackage{amssymb}
\usepackage{amscd}
\usepackage{amsfonts}
\usepackage{setspace}
\usepackage{version}
\usepackage{mathrsfs}
\usepackage[colorlinks=true]{hyperref}

\usepackage{graphicx}

\newtheorem{theorem}{Theorem}[section]
\newtheorem{lemma}[theorem]{Lemma}
\newtheorem{proposition}[theorem]{Proposition}
\newtheorem{corollary}[theorem]{Corollary}

\newtheorem*{measure-compare-rpt}{Proposition \ref{measure-compare}}
\newtheorem*{main-geo-num-rpt}{Proposition \ref{main-geo-num}}
\newtheorem*{main-quad-gap-rpt}{Proposition \ref{main-quad-gap}}
\newtheorem*{blue-prop-main-rpt}{Proposition \ref{blue-prop-main}}

\theoremstyle{definition}

\newtheorem{question}{Question}[section]

\renewcommand{\leq}{\leqslant}
\renewcommand{\geq}{\geqslant}

\newcommand\tr{\operatorname{tr}}

\def\F{\mathbf{F}}
\def\R{\mathbf{R}}
\def\C{\mathbf{C}}
\def\Z{\mathbf{Z}}

\def\P{\mathbb{P}}
\def\Q{\mathbf{Q}}
\def\N{\mathbf{N}}
\def\T{\mathbf{T}}
\def\eps{\varepsilon}

\def\b{\mathbf{b}}

\def\init{\operatorname{init}}
\def\dist{\operatorname{dist}}
\def\blue{\operatorname{Blue}}
\def\red{\operatorname{Red}}

\def\vol{\operatorname{vol}}

\def\Sym{\operatorname{Sym}}
\def\Span{\operatorname{Span}}

\def\disc{\operatorname{disc}}
\def\Supp{\operatorname{Supp}}
\def\xx{\mathbf{x}}
\def\LL{\mathbf{L}}
\def\hh{\mathbf{h}}

\newcommand{\md}[1]{\ensuremath{(\operatorname{mod}\, #1)}}

\numberwithin{equation}{section}

\makeatletter
\renewcommand\subsection{\@startsection{subsection}{2}%
	\z@{.5\linespacing\@plus.7\linespacing}{.5\linespacing}%
	{\normalfont\bfseries}}
\makeatother

\makeatletter
\renewcommand\part{\@startsection{part}{2}%
	\z@{0.5\linespacing\@plus2\linespacing}{\linespacing}%
	{\normalfont\large\scshape\bfseries\centering}}
\makeatother

\AtBeginDocument{
  \addtocontents{toc}{\Small}
}

\begin{document}

\title[van der Waerden numbers]{New lower bounds for van der Waerden numbers}


\author{Ben Green}
\address{Mathematical Institute\\
Radcliffe Observatory Quarter\\
Woodstock Road\\
Oxford OX2 6GG\\
England}
\email{ben.green@maths.ox.ac.uk}
\thanks{The author is supported by a Simons Investigator grant and is grateful to the Simons Foundation for their continued support.}

\subjclass[2000]{Primary }

\begin{abstract}
We show that there is a red-blue colouring of $[N]$ with no blue 3-term arithmetic progression and no red arithmetic progression of length $e^{C(\log N)^{3/4}(\log \log N)^{1/4}}$. Consequently, the two-colour van der Waerden number $w(3,k)$ is bounded below by $k^{b(k)}$, where $b(k) = c \big( \frac{\log k}{\log\log k} \big)^{1/3}$. Previously it had been speculated, supported by data, that $w(3,k) = O(k^2)$.
\end{abstract}

\maketitle

\setcounter{tocdepth}{1}
\tableofcontents

%

\part{Introduction}

\section{Statement of results and history}

Let $k \geq 3$ be a positive integer. Write $w(3,k)$ (sometimes written $w(2;3,k)$) for the smallest $N$ such that the following is true: however $[N] = \{1,\dots, N\}$ is coloured blue and red, there is either a blue 3-term arithmetic progression or a red $k$-term arithmetic progression. The celebrated theorem of van der Waerden implies that $w(3,k)$ is finite; the best upper bound currently known is due to Schoen \cite{schoen}, who proved that for large $k$ one has $w(3,k) < e^{k^{1 - c}}$ for some constant $c > 0$. This also follows from the celebrated recent work of Bloom and Sisask \cite{bloom-sisask} on bounds for Roth's theorem.

There is some literature on lower bounds for $w(3,k)$. Brown, Landman and Robertson \cite{blr} showed that $w(3,k) \gg k^{2 - \frac{1}{\log\log k}}$, and this was subsequently improved by Li and Shu \cite{li-shu} to $w(3,k) \gg (k/\log k)^2$, the best bound currently in the literature. Both of these papers use probabilistic arguments based on the Lov\'asz Local Lemma.

Computation or estimation of $w(3,k)$ for small values of $k$ has attracted the interest of computationally-inclined mathematicians. In \cite{blr} one finds, for instance, that $w(3,10) = 97$, whilst in Ahmed, Kullmann and Snevily \cite{aks} one finds the lower bound $w(3,20) \geq 389$ (conjectured to be sharp) as well as $w(3,30) \geq 903$. This data suggests a quadratic rate of growth, and indeed Li and Shu state as an open problem to prove or disprove that $w(3,k) \geq c k^2$, whilst in \cite{aks} it is conjectured that $w(3,k) = O(k^2)$.  Brown, Landman and Robertson are a little more circumspect and merely say that it is ``of particular interest whether or not there is a polynomial bound for $w(3,k)$''. I should also admit that I suggested the plausibility of a quadratic bound myself \cite[Problem 14]{green-open-probs}.

The main result in this paper shows that, in fact, there is no such bound.

\begin{theorem}\label{mainthm}
There is a blue-red colouring of $[N]$ with no blue 3-term progression and no red progression of length $e^{C(\log N)^{3/4}(\log \log N)^{1/4}}$. Consequently we have the bound $w(3,k) \geq k^{b(k)}$, where $b(k) = c \big( \frac{\log k}{\log\log k} \big)^{1/3}$.
\end{theorem}

\emph{Update, June 2022:} Nine months after the arxiv version of this paper was made public, Zachary Hunter \cite{hunter-vdw} was able to simplify parts of the argument, and at the same time improve the lower bound to $w(3,k) \geq k^{b'(k)}$, where $b'(k) = c \frac{\log k}{\log \log k}$.\vspace*{6pt}

\emph{Acknowledgements.} It is a pleasure to thank Jon Keating, Peter Keevash, Jens Marklof and Mark Rudelson and the three anonymous referees for helpful comments on various issues related to the paper.

\section{Overview and structure the paper}\label{outline-sec}

\subsection{Discussion and overview}

I first heard the question of whether or not $w(3,k) = O(k^2)$ from Ron Graham in around 2004. My initial reaction was that surely this must be false, for the following reason: take a large subset of $[N]$ free of 3-term progressions, and colour it blue. Then the complement of this set probably does not have overly long red progressions. However, by considering the known examples of large sets free of 3-term progressions, one swiftly becomes less optimistic about this strategy.\vspace*{8pt}

\emph{Example 1.} If we take the blue points to be the folklore example $\{ \sum_i a_i 3^i : a_i \in \{0,1\}\} \cap [N]$ then the set of red points contains a progression of \emph{linear} length, namely $\{ n \equiv 2 \md{3}\}$.\vspace*{8pt}

\emph{Example 2.} If we take the blue points to be the Salem-Spencer set \cite{salem-spencer} or the Behrend set \cite{behrend} then one runs into similar issues. Both sets consists of points $a_1 + a_2 (2d-1) + \dots + a_n (2d - 1)^n$ with all $a_i \in \{0,1,\dots, d-1\}$, and for such sets the red set will contain progressions such as $\{ n \equiv d \md{2d-1}\}$. If $N := d(2d - 1)^n$, so that the set is contained in $[N]$, the length of such a red progression is $\sim N/d$. In the Behrend case one takes $d \sim \sqrt{\log N}$, so this is very long.\vspace*{8pt}

\emph{Example 3.} One runs into an apparently different kind of obstacle (though in fact it is closely related) when considering the variant of Behrend's construction due to Julia Wolf and myself \cite{green-wolf}. (The \emph{bound} in \cite{green-wolf} was previously obtained by Elkin \cite{elkin}, but the method of construction in \cite{green-wolf} was different.) Roughly speaking, this construction proceeds as follows. Pick a dimension $D$ and consider the torus $\T^D = \R^D/\Z^D$. In this torus, consider a thin annulus, the projection $\pi(A)$ of the set $A = \{ x \in \R^D : \frac{1}{4} - N^{-4/D} \leq \Vert x \Vert_2 \leq \frac{1}{4}\}$ under the natural map. Pick a rotation $\theta \in \T^D$ at random, and define the blue points to be $\{ n \in [N] : \theta n \in \pi(A)\}$. By simple geometry one can show that the only 3-term progressions in $A$ are those with very small common difference $v$, $\Vert v \Vert_2 \ll N^{-4/D}$. This property transfers to $\pi(A)$, essentially because by choosing $\frac{1}{4}$ as the radius of the annulus one eliminates ``wraparound effects'' which could have generated new progressions under the projection $\pi$. Due to the choice of parameters it turns out that with very high probability there are no blue 3-term progressions at all.

Let us now consider red progressions. Suppose that $\theta = (\theta_1,\dots, \theta_D)$, and think of $D$ as fixed, with $N$ large. By Dirichlet's theorem, there is some $d \leq \sqrt{N}$ such that $\Vert \theta_1 d \Vert \leq 1/\sqrt{N}$, where $\Vert \cdot \Vert_{\T}$ denotes the distance to the nearest integer. Since $\theta$ is chosen randomly, the sequence $(\theta n)_{n = 1}^{\infty}$ will be highly equidistributed, and one certainly expects to find an $n_0 = O_D(1)$ such that $\theta n_0 \approx (\frac{1}{2},\dots, \frac{1}{2})$. If one then considers the progression $P = \{ n_0 + nd : n \leq \sqrt{N}/10\}$ (say), one sees that $P \subset [N]$, and that $\Vert \theta_1 (n_0 + nd) - \frac{1}{2} \Vert_{\T} < \frac{1}{4}$ for all $n \leq \sqrt{N}/10$. That is, all points of $P$ avoid the annulus $\pi(A)$ (and in fact the whole ball of radius $\frac{1}{4}$) since their first coordinates are confined to a narrow interval about $\frac{1}{2}$. Therefore $P$ is coloured entirely red.\vspace*{8pt}

This last example does rather better than Examples 1 and 2, and it is the point of departure for the construction in this paper. Note that, in Example 3, the progression $P$ of length $\gg \sqrt{N}$ that we found is not likely to be the only one. One could instead apply Dirichlet's theorem to any of the other coordinates $\theta_2,\dots, \theta_D$. Moreover, one could also apply it to $\theta_1 + \theta_2$, noting that points $x \in \pi(A)$ satisfy $\Vert x_1 + x_2 \Vert_{\T} \leq \frac{\sqrt{2}}{4}$ and so avoid a narrow ``strip'' $\Vert x_1 + x_2 - \frac{1}{2} \Vert_{\T} \leq  \frac{1}{2} - \frac{\sqrt{2}}{4}$, or to $\theta_1 + \theta_2 + \theta_3$, noting that points $x \in \pi(A)$ satisfy $\Vert x_1 + x_2 + x_3 \Vert_{\T} \leq \frac{\sqrt{3}}{4}$ and so avoid a narrow strip $\Vert x_1 + x_2 + x_3 - \frac{1}{2} \Vert_{\T} \leq \frac{1}{2} - \frac{\sqrt{3}}{4}$. However, this no longer applies to $\theta_1 + \theta_2 + \theta_3 + \theta_4$, since the relevant strip has zero width. 

This discussion suggests the following key idea: instead of one annulus $\pi(A)$, we could try taking several, in such a way that every strip like the ones just discussed intersects at least one of these annuli. This then blocks all the ``obvious'' ways of making red progressions of length $\sim \sqrt{N}$. Of course, one then runs the risk of introducing blue 3-term progressions. However, by shrinking the radii of the annuli to some $\rho \lll 1$ a suitable construction may be achieved by picking the annuli to be random translates of a fixed one.

At this point we have an annulus $A = \{ x \in \R^D : \rho - N^{-4/D} \leq \Vert x \Vert_2 \leq \rho\}$ together with a union of translates $S := \bigcup_{i=1}^M (x_i + \pi(A)) \subset \T^D$. Pick $\theta \in \T^D$ at random, and colour those $n \leq N$ for which $\theta n \in S$ blue. As we have stated, it is possible to show that (with a suitable choice of $\rho$, and for random translates $x_1,\dots, x_M$ with $M$ chosen correctly) there are likely to be no blue 3-term progressions. Moreover, the obvious examples of red progressions of length $\sim \sqrt{N}$ coming from Dirichlet's theorem are blocked.

Now one may also apply Dirichlet's theorem to pairs of frequencies, for instance producing $d \leq N^{2/3}$ such that $\Vert \theta_1 d\Vert_{\T}, \Vert \theta_2 d \Vert_{\T} \leq N^{-1/3}$, and thereby potentially creating red progressions of length $\sim N^{1/3}$ unless they too are blocked by the union of annuli. To obstruct these, one needs to consider ``strips'' of codimension $2$, for instance given by conditions such as $x_1, x_2 \approx \frac{1}{2}$. Similarly, to avoid progressions of length $\sim N^{1/4}$ one must ensure that strips of codimension $3$ are blocked, and so on.

Using these ideas one can produce, for arbitrarily large values of $r$, a red--blue colouring of $[N]$ with no blue 3-term progression and no obvious way to make a red progression of length $N^{1/r}$. Of course, this is by no means a proof that there are no such red progressions!

Let us now discuss a further difficulty which arises when one tries to show that there no long red progressions. Consider the most basic progression $P = \{ 1, 2,\dots, X\}$, $X = N^{1/r}$, together with the task of showing it has at least one blue point. Suppose that $x_1$ (the centre of the first annulus in $S$) is equal to $0 \in \T^D$. Then, since $P$ is ``centred'' on $0$ it is natural to try and show that $\{ \theta, 2\theta,\dots, X \theta\}$ intersects the annulus $\pi(A)$ with centre $x_1 = 0$, or in other words to show that there is $n \leq X$ such that 
\begin{equation}\label{gap-task} (\rho - N^{-4/D})^2 < \Vert n \theta_1 \Vert_{\T}^2 + \dots + \Vert n \theta_D \Vert_{\T}^2 < \rho^2.\end{equation}
The general flavour of this problem is to show that a certain ``quadratic form'' takes at least one value in a rather small interval. However, what we have is not a bona fide quadratic form. To make it look like one, we apply standard geometry of numbers techniques to put a multidimensional structure on the \emph{Bohr set} of $n$ such that $\Vert n \theta_1 \Vert_{\T},\dots, \Vert n \theta_D \Vert_{\T} \leq \frac{1}{10}$ (say). This gives, inside the set of such $n$, a multidimensional progression $\{ \ell_1 n_1 + \dots + \ell_{D+1} n_{D+1} : 0 \leq \ell_i < L_i\}$, for certain $n_i$ and certain lengths $L_i$. The size $L_1 \cdots L_{D+1}$ of this progression is comparable to $X$.

In this ``basis'', the task \eqref{gap-task} then becomes to show that there are $\ell_1,\dots, \ell_{D+1}$, $0 \leq \ell_i < L_i$, such that 
\begin{equation}\label{gap-task-2} (\rho - N^{-4/D})^2  < q(\ell_1,\dots, \ell_{D+1}) < \rho^2,\end{equation}
where $q$ is a certain quadratic form depending on $n_1,\dots, n_{D+1}$ and $\theta$. A representative case (but not the only one we need to consider) would be $L_1 \approx \dots \approx L_{D+1} \approx L = X^{1/(D+1)}$, with the coefficients of $q$ having size $\sim L^{-2}$ so that $q$ is bounded in size by $O(1)$. Note that $N^{-4/D} \approx L^{-4r}$. Thus we have a problem of roughly the following type: given a quadratic form $q : \Z^{D+1} \rightarrow \R$ with coefficients of size $\sim L^{-2}$, show that on the box $[L]^{D+1}$ it takes at least one value on some given interval of length $L^{-4r}$.

Without further information, this is unfortunately a hopeless situation because of the possibility that, for instance, the coefficients of $q$ lie in $\frac{1}{Q} \Z$, for some $Q > L^2$. In this case, the values taken by $q$ are $\frac{1}{Q}$-separated and hence, for moderate values of $Q$, not likely to lie in any particular interval of length $L^{-4r}$.

A small amount of hope is offered by the fact that $q$ is not a fixed quadratic form -- it depends on the random choice of $\theta$. However, the way in which random $\theta$ correspond to quadratic forms is not at all easy to analyse and moreover we also need to consider similar problems for $2\theta, 3 \theta,\dots$ corresponding to potential red progressions with common difference $d = 2,3,\dots$.

Our way around this issue, and the second key idea in the paper, is to introduce a large amount of extra randomness elsewhere, in the definition of the annuli. Instead of the standard $\ell^2$-norm $\Vert x \Vert_2$, we consider instead a perturbation $\Vert (I + E) x \Vert_2$, where $E$ is a random $D \times D$ matrix with small entries so that $I+E$ is invertible with operator norm close to $1$. Such ellipsoidal annuli are just as good as spherical ones for the purposes of our construction. The choice of $E$ comes with a massive $D(D+1)/2$ degrees of freedom (not $D^2$, because $E$s which differ by an orthogonal matrix give the same norm). The quadratic form $q$ then becomes a random quadratic form $q_E$. 

With considerable effort, the distribution of $q_E$, $E$ random, can be shown to be somewhat related (for typical $\theta$) to the distribution of a truly random quadratic form $q_a(\ell_1,\dots, \ell_{D+1}) = \sum_{i\leq j} a_{ij} \ell_i \ell_j$, with the coefficients chosen uniformly from $|a_{ij}| \leq L^{-2}$. One is then left with the task of showing that a uniformly random quadratic form $q_a$ takes values in a very short interval of length $L^{-4r}$. Moreover, this is required with a very strong bound on the exceptional probability, suitable for taking a union bound over the $\sim N^{1 - 1/r}$ possible choices of the common difference $d$.

A natural tool for studying gaps in the values of quadratic forms is the Hardy-Littlewood circle method, and indeed it turns out that a suitable application of the Davenport-Heilbronn variant of the method can be applied to give what we require. The application is not direct, and additionally requires a novel amplification argument using lines in the projective plane over a suitable $\F_p$ to get the strong bound on the exceptional probability that we need.

The above sketch omitted at least one significant detail, namely how to handle ``uncentred'' progressions $P$ starting at points other than $0$. For these, one must use the particular choice of the centres $x_i$. One can show that $P$ enters inside at least one of the balls $x_i + \pi(B_{\rho/10}(0))$, and starting from here one can proceed much as centred case.\vspace*{10pt}

\subsection{Structure of the paper} With that sketch of the construction complete, let us briefly describe the structure of the paper. As the above discussion suggests, it is natural to introduce a parameter $r$ and consider the following equivalent form of Theorem \ref{mainthm}.

\begin{theorem}\label{mainthm-1}
Let $r$ be an integer, and suppose that $N > e^{Cr^4 \log r}$. Then there is a red/blue colouring of $[N]$ with no blue 3-term progression and no red progression of length $N^{1/r}$.
\end{theorem}

Taking $r = c \big( \frac{\log N}{\log \log N} \big)^{1/4}$ for suitable $c$, we recover Theorem \ref{mainthm}. While Theorems \ref{mainthm} and \ref{mainthm-1} are equivalent, it is much easier to think about Theorem \ref{mainthm-1} and its proof by imagining that $r$ is fixed and that $N$ is a very large compared to $r$. We will, of course, keep track of just how large $N$ needs to be as we go along, but this is somewhat secondary to understanding the key concepts of the argument. For the rest of the paper, $r$ will denote the parameter appearing in Theorem \ref{mainthm-1}, and we will always assume (as we clearly may) that it is sufficiently large.

In Section \ref{notation-sec} we summarise some key notation and conventions in force for the rest of the paper. In Section \ref{sec4} we turn to the details of our construction, in particular constructing the translates $x_1,\dots, x_M$ of our annuli, and introducing the notion of a random ellipsoidal annulus properly. In Section \ref{colour-def} we describe the red/blue colouring itself and divide the task of showing that there are no blue 3-term progressions or red $N^{1/r}$-term progressions into three parts (the blue progressions, and what we call steps 1 and 2 for the red progressions). In Section \ref{blue-ap-sec} we handle the blue 3-term progressions. Section \ref{diophantine-sec} is then devoted to a technical ``diophantine'' condition on $\theta \in \T^D$ which will be in force for the rest of the paper. In Section \ref{first-step-sec} we handle step 1 of the treatment of red progressions.

At this point we are only one third of the way through the paper. The remaining discussion is devoted to the treatment of step 2 for the red progressions, which involves the geometry of numbers and gaps in random quadratic forms material outlined above. We devote Section \ref{second-step-sec} to a more detailed technical overview of the argument which reduces it to three key propositions. The proofs of these propositions are then handled in Parts IV and V of the paper. Part IV contains, roughly speaking, the relevant geometry of numbers arguments, whilst Part V contains the arguments pertaining to gaps in quadratic forms. These parts may be read independently of one another and of the rest of the paper.

\subsection{Further comments} The discussion around the application of Dirichlet's theorem above suggests that there are certain ``phase changes'' in the problem as one goes from ruling out red progressions of length $\sim N^{1/2}$ to ruling out progressions of length $\sim N^{1/3}$, and so on. Indeed, this is why we formulate our main result in the equivalent form of Theorem \ref{mainthm-1}. I consider it quite plausible that such phase changes are not merely an artefact of our argument but rather of the problem as a whole, and that the apparently strong numerical evidence for quadratic behaviour of $w(3,k)$ reflects the fact that in the regime $k \leq 40$ one is only seeing the first phase in which it is more efficient to take just one large annulus as in the construction of Julia Wolf and myself, at the expense of having to allow strips of codimension 1 which admit red progressions of length $\sim N^{1/2}$. It would be interesting to see whether the ideas of this paper could be used to produce, computationally, an example with $w(3,k) \sim k^3$.

I have worked quite hard to try and optimise the exponent in Theorem \ref{mainthm} and it seems to represent the limit of the method for multiple different reasons, as discussed in a little more detail in Section \ref{notation-sec} below. These limitations seem to be a mix of fundamental ones and artefacts of our analysis. I would expect that the true value of $w(3,k)$ lies somewhere in between the bound of Theorem \ref{mainthm} and something like $k^{c\log k}$, which is what a Behrend construction of the blue points would give if only the complement of such a set ``behaved randomly''. Ron Graham \cite{graham} established a lower bound of this type for a restricted version of the problem in which one only forbids red progressions with common difference $1$.

Finally, we remark that \cite{fox-pohoata} is an earlier example in which random unions of structured objects are used to understand a problem related to arithmetic progressions.

\section{Notation and conventions}\label{notation-sec}

\subsection{Fourier transforms}
We use the standard notation $e(t) := e^{2\pi i t}$ for $t \in \R$.

We will take Fourier transforms of functions on $\R^k, \Z^k, \T^k$ for various integer $k$. We will use the same hat symbol for all of these, and define them as follows:
\begin{itemize}
\item If $f : \R^k \rightarrow \C$, $\hat{f}(\gamma) = \int_{\R^k} f(x) e(-\langle \gamma , x\rangle) dx$ for $\gamma \in \R^k$;
\item If $f : \Z^k \rightarrow \C$, $\hat{f}(\theta) = \sum_{n \in \Z^k} f(n) e(-n \cdot \theta)$ for $\theta \in \T^k$;
\item If $f : \T^k \rightarrow \C$, $\hat{f}(\xi) = \int_{\T^k} f(x) e(-\xi \cdot x) dx$ for $\xi \in \Z^k$.
\end{itemize}
The notation $\langle x, y\rangle$ for $\sum_i x_i y_i$ in $\R^k$, but $x \cdot y$ in $\Z^k$ and $\T^k$, is merely cultural and is supposed to reflect the fact that our arguments in the former space will be somewhat geometric in flavour.

We will only be using the Fourier transform on smooth, rapidly decaying functions where convergence issues are no problem. Note in particular that the normalisation of the Fourier transform on $\R^k$ (with the phase multiplied by $2\pi$) is just one of the standard options, but a convenient one in this paper. With this normalisation, Fourier inversion states that $f(x) = \int_{\R^k} \hat{f}(\gamma) e(\langle \gamma, x\rangle) d\gamma$.

\subsection{Convention on absolute constants}

It would not be hard to write in explicit constants throughout the paper. They would get quite large, but not ridiculously so. However, we believe it makes the presentation neater, and the dependencies between parameters easier to understand, if we leave the larger ones unspecified and adopt the following convention:
\begin{itemize}
\item $C_1$ is a sufficiently large absolute constant;
\item $C_2$ is an even larger absolute constant, how large it needs to be depending on the choice of $C_1$;
\item $C_3$ is a still larger constant, large enough in terms of $C_1, C_2$.
\end{itemize}

To clarify, no matter which $C_1$ we choose (provided it is sufficiently big) there is an appropriate choice of $C_2$, and in fact all sufficiently large $C_2$ work. No matter which $C_2$ we choose, all sufficiently large $C_3$ work. There are many constraints on how large $C_1$ needs to be throughout the paper, and it must be chosen to satisfy \emph{all} of them, and similarly for $C_2, C_3$.

\subsection{Key parameters}

The most important global parameters in the paper are the following:

\begin{itemize}
\item $N$: the interval $[N]$ is the setting for Theorem \ref{mainthm-1}.
\item $r$: a positive integer, always assumed to be sufficiently large. $N^{1/r}$ is the length of red progressions we are trying to forbid.
\item $D$: a positive integer dimension. The torus $\T^D = \R^D/\Z^D$ will play a key role in the paper. 
\end{itemize}

Throughout the paper, we will assume that 
\begin{equation}\label{global-conditions} \mbox{$r$ sufficiently large}, \qquad D = C_3 r^2, \qquad N \geq D^{C_2 D^2}.\end{equation}
Several lemmas and propositions do not require such strong assumptions. However, two quite different results in the paper (Proposition \ref{prop21}, and the application of Proposition \ref{main-quad-gap} during the proof of Proposition \ref{second-step} in Section \ref{second-step-sec}) require a condition of the form $D \gg r^2$. The condition $N > D^{C_2 D^2}$ comes up in the proof of Proposition \ref{second-step}, in fact in no fewer than three different ways in the last displayed equation of Section \ref{second-step-sec}. For these reasons, it seems as if our current mode of argument cannot possibly yield anything stronger than Theorem \ref{mainthm-1}.

A number of other parameters and other nomenclature feature in several sections of the paper:

\begin{itemize}
\item $X$: shorthand for $N^{1/r}$.
\item $\theta$: an element of $\T^D$, chosen uniformly at random, and later in the paper always taken to lie in the set $\Theta$ of diophantine elements (Section \ref{diophantine-sec}).
\item $\rho$: a small radius (of annuli in $\T^D$), from Section \ref{colour-def} onwards fixed to be $D^{-4}$.
\item $\mathbf{e}$: a uniform random element of $[-\frac{1}{D^4} , \frac{1}{D^4} ]^{D(D+1)/2}$ (used to define random ellipsoids). Usually we will see $\sigma(\mathbf{e})$, which is a symmetric matrix formed from $\mathbf{e}$ in an obvious way (see Section \ref{sec4} for the definition).
\item $d$: invariably the common difference of a progression, with $d \leq N/X$.
\end{itemize}

The letter $Q$ is reserved for a ``complexity'' parameter (bounding the allowed size of coefficients, or of matrix entries) in various different contexts.

\subsection{Notation}

$[N]$ always denotes $\{1,\dots, N\}$.

If $x \in \T$  then we write $\Vert x \Vert_{\T}$ for the distance from $x$ to the nearest integer. If $x = (x_1,\dots, x_D) \in \T^D$ then we write $\Vert x \Vert_{\T^D} = \max_i \Vert x_i \Vert_{\T}$.

We identify the dual $\hat{\T}^D$ with $\Z^D$ via the map $\xi \mapsto (x \mapsto e(\xi \cdot x))$, where $\xi \cdot x = \xi_1 x_1 + \cdots + \xi_D x_D$. In this setting we always write $|\xi| := \max_i |\xi_i|$, instead of the more cumbersome $\Vert \xi \Vert_{\infty}$.

Apart from occasional instances where it denotes $3.141592\dots$, $\pi$ is the natural projection homomorphism $\pi : \R^D \rightarrow \T^D$. Clearly $\pi$ is not invertible, but nonetheless we abuse notation by writing $\pi^{-1}(x)$ for the unique element $y \in (-\frac{1}{2}, \frac{1}{2}]^D$ with $\pi(y) = x$.

If $R$ is a $D \times D$ matrix over $\R$ then we write $\Vert R \Vert$ for the $\ell^2$-to-$\ell^2$ operator norm, that is to say $\Vert Rx \Vert_{2} \leq \Vert R \Vert \Vert x \Vert_2$, and $\Vert R \Vert$ is the smallest constant with this property. Equivalently $\Vert R \Vert$ is the largest singular value of $R$.

\part{A red/blue colouring of $[N]$}

\section{Random ellipsoidal annuli}\label{sec4}

In this section, we prepare the ground for describing our red/blue colouring of $[N]$, which we will give in Section \ref{colour-def}. In the next section, we describe our basic construction by specifying the points in $[N]$ to be coloured blue. The torus $\T^D$ (and Euclidean space $\R^D$) play a fundamental role in our construction, where $D = C_3 r^2$. 

\subsection{A well-distributed set of centres}

As outlined in Section \ref{outline-sec}, an important part of our construction is the selection (randomly) of a certain set $x_1,\dots, x_M$ of points in $\T^D$. Later, we will fix $\rho := D^{-4}$, but the following proposition does not make any assumption on $\rho$ beyond that $\rho < D^{-1}$.

\begin{proposition} \label{prop21} Suppose that $D = C_3 r^2$ and that $\rho < \frac{1}{D}$.
There are $x_1,\dots, x_M \in \T^D$ such that the following hold: 
\begin{enumerate}
\item Whenever $i_1, i_2, i_3$ are not all the same, $\Vert x_{i_1} - 2 x_{i_2} + x_{i_3} \Vert_{\T^D} \geq 10 \rho$. 
\item Whenever $V \leq \Q^D$ is a subspace of dimension at most $4r$ and $x \in \T^D$, there is some $j$ such that $\Vert \xi \cdot (x_j - x) \Vert_{\T} \leq \frac{1}{100}$ for all $\xi \in V \cap \Z^D$ with $|\xi| \leq \rho^{-3}$.
\end{enumerate}
\end{proposition}
\emph{Remark.} With reference to the outline in Section \ref{outline-sec}, condition (2) here is saying that  any ``slice'' of codimension at most $4r$ contains one of the $x_i$; this is what obstructs a simple construction of red progressions of length $N^{1/r}$ using Dirichlet's theorem. Item (1) will allow us to guarantee that by taking a union of translates of annuli, rather than just one, we do not introduce new blue 3-term progressions. 

\begin{proof}
Set $M = \lceil \rho^{-D/4}\rceil$ and pick $x_1,\dots, x_M \in \T^D$ independently and uniformly at random. For any triple $(i_1, i_2, i_3)$ with not all the indices the same, $x_{i_1} - 2x_{i_2} + x_{i_3}$ is uniformly distributed on $\T^D$. Therefore $\P ( \Vert x_{i_1} - 2 x_{i_2} + x_{i_3} \Vert_{\T^D} \leq 10 \rho) \leq (20 \rho)^{D}$. Summing over all $< M^3$ choices of indices gives an upper bound of $M^3 (20 \rho)^D < \frac{1}{4}$ on the probability that (1) fails (since $D$ is sufficiently large).

For (2), we may assume that $V$ is spanned (over $\Q$) by vectors $\xi \in \Z^D$ with $|\xi| \leq \rho^{-3}$ (otherwise, pass from $V$ to the subspace of $V$ spanned by such vectors). There are at most $4r (3 \rho^{-3})^{4r D} < \rho^{-14rD}$ such $V$ (choose the dimension $m < 4r$, and then $m$ basis elements with $\xi \in \Z^D$ and $|\xi| \leq \rho^{-3}$).

Fix such a $V$. By Lemma \ref{quant-lattice}, $V \cap \Z^D$ is a free $\Z$-module generated by some $\xi_1,\dots, \xi_m$, $m \leq 4r$, and with every element $\xi \in V \cap \Z^D$ with $|\xi| \leq \rho^{-3}$ being $\xi = n_1 \xi_1+ \cdots + n_m \xi_m$ with $|n_i| \leq m! (2\rho^{-3})^m \leq \frac{1}{400r}\rho^{-14 r}$. Thus, to satisfy our requirement, we need only show that for any $x \in \T^D$ there is $x_j$ such that 
\begin{equation}\label{want}
\Vert \xi_i \cdot (x_j - x) \Vert_{\T} \leq \rho^{14r} \; \mbox{for $i = 1,\dots, m$},\end{equation} since then
\[ \Vert \xi \cdot (x_j - x) \Vert_{\T} \leq  \rho^{14r}\sum_{i = 1}^m |n_i|  < \frac{1}{100}.\]

Divide $\T^m$ into $\rho^{- 14 rm} \leq \rho^{-56r^2}$ boxes of sidelength $\rho^{14 r}$; it is enough to show that each such box $B$ contains at least one point $(\xi_1 \cdot x_j, \cdots \xi_m \cdot x_j)$. The following fact will also be needed later, so we state it as a separate lemma.

\begin{lemma}\label{ind-td-tm}
Let $\xi_1,\dots, \xi_m \in \Z^D$ be linearly independent. Then, as $x$ ranges uniformly over $\T^D$, $(\xi_1 \cdot x, \cdots, \xi_m \cdot x)$ ranges uniformly over $\T^m$.
\end{lemma}
\begin{proof}
Let $f(t) = e(\gamma \cdot  t)$ be a nontrivial character on $\T^{m}$. Then
\[ \int_{\T^D} f(\xi_1 \cdot x, \dots, \xi_m \cdot x) dx = \int_{\T^D} e( (\gamma_1 \xi_1 + \cdots + \gamma_{m} \xi_{m}) \cdot x)dx  = 0 = \int_{\T^m} f,\] since $\gamma_1 \xi_1 + \dots + \gamma_m \xi_m \neq 0$. Since the characters are dense in $L^1(\T^m)$, the result follows.
\end{proof}

Returning to the proof of Proposition \ref{prop21}, Lemma \ref{ind-td-tm} implies that for each fixed $j$, 
\[ \P( (\xi_1 \cdot x_j, \dots , \xi_m \cdot x_j) \notin B ) = 1 - \rho^{14rm}.\]
By independence, 
\begin{align*} \P ( (\xi_1 \cdot x_j, \dots , \xi_m \cdot x_j) \notin & \; B \; \mbox{for $j = 1,\dots, M$} )  \\ &  = (1 - \rho^{14rm})^M  \leq e^{-\rho^{14rm} M}  \leq e^{-\rho^{-D/8}}\end{align*} Here we critically use that $D = C_3 r^2$; $C_3 \geq 448$ is sufficient here, but it will need to be larger than this in later arguments.
Summing over the boxes $B$, we see that the probability of even one empty box is $\leq \rho^{-56 r^2} e^{-\rho^{-D/8}}$. 
This is the probability that \eqref{want} does not hold, for this particular $V$. Summing over the $\leq \rho^{-14 D r}$ choices for $V$, the probability that \eqref{want} fails to hold for \emph{some} $V$ is 
\[ \leq \rho^{-16 D r} e^{-\rho^{-D/8}}.\]
For $D$ large, this will be $< \frac{1}{4}$ as well, uniformly in $\rho$. (To see this, write $X = 1/\rho > D$, then this function is bounded by $X^{D^2} e^{-X^{D/8}}$, which is absolutely tiny on the range $X > D$.)

It follows that, with probability $> \frac{1}{2}$ in the choice of $x_1,\dots, x_M$, both (1) and (2) hold.
\end{proof}

\subsection{Random ellipsoidal annuli} Our construction is based on annuli centred on the points $x_1,\dots, x_M$ just constructed. As outlined in Section \ref{outline-sec}, so as to introduce a source of randomness into the problem we consider, rather than just spherical annuli, random ellipsoidal annuli.

To specify the ellipsoids, here and throughout the paper identify $\R^{D(D+1)/2}$ with the space of all tuples $x = (x_{ij})_{1 \leq i \leq j \leq D}$. Let $\mathbf{e}$ be a random tuple uniformly sampled from $[-\frac{1}{D^4} , \frac{1}{D^4} ]^{D(D+1)/2} \subset \R^{D(D+1)/2}$. 

To any tuple $x \in \R^{D(D+1)/2}$, we associate a symmetric matrix $\sigma(x) \in \Sym_D(\R)$ (the space of $D \times D$ symmetric matrices over $\R$) as follows: $(\sigma(x))_{ii} = x_{ii}$, $(\sigma(x))_{ij} = \frac{1}{2}x_{ij}$ for $i < j$, and $(\sigma(x))_{ij} = \frac{1}{2}x_{ji}$ for $i > j$. 

The ellipsoidal annuli we consider will then be of the form $\pi(A_{\mathbf{e}})$, where $\pi : \R^D \rightarrow \T^D$ is the natural projection and
\begin{equation}\label{annulus-def} A_{\mathbf{e}} :=  \{ x \in \R^D : \rho - N^{-4/D} < \Vert (I + \sigma(\mathbf{e})) x \Vert_2 < \rho\}.\end{equation}
It is convenient to fix, for the rest of the paper,
\begin{equation}\label{rho-def} \rho := D^{-4};\end{equation}
the choice is somewhat arbitrary and any sufficiently large power of $1/D$ would lead to essentially the same bounds in our final result. With this choice, the parameter $M$ in Proposition \ref{prop21} (that is, the number of points $x_1,\dots, x_M$) is $D^D$. 

Now $\Vert \sigma(e) \Vert \leq D \Vert e \Vert_{\infty} \leq \frac{1}{2}$, where $\Vert \cdot \Vert$ denotes the $\ell^2$-to-$\ell^2$ operator norm on matrices, and therefore
\begin{equation}\label{e-operator} \frac{1}{2} \leq \Vert I + \sigma(\mathbf{e}) \Vert \leq \frac{3}{2}.\end{equation}

\emph{Remark.} Taking $\sigma(\mathbf{e})$ to be symmetric is natural in view of the polar decomposition of real matrices. Premultiplying $\sigma(\mathbf{e})$ by an orthogonal matrix makes no difference to $\Vert (I + \sigma(\mathbf{e}) ) x\Vert_2$.

\section{The colouring. Outline proof of the main theorem}\label{colour-def}

We are now in a position to describe our red/blue colouring of $[N]$. Once again let $r, D$ be integers with $r$ sufficiently large and $D = C_3 r^2$. Set $\rho := D^{-4}$, and let $x_1,\dots, x_M \in \T^D$ be points as constructed in Proposition \ref{prop21}, for this value of $\rho$. Pick $\mathbf{e} \in [-\frac{1}{D^4} , \frac{1}{D^4} ]^{D(D+1)/2} \subset \R^{D(D+1)/2}$ uniformly at random, and consider the random ellipsoidal annulus
\[ A_{\mathbf{e}} :=  \{ x \in \R^D : \rho - N^{-4/D} < \Vert (I + \sigma(\mathbf{e})) x \Vert_2 < \rho\}.\]

Pick $\theta \in \T^D$ uniformly at random, let $\pi : \R^D \rightarrow \T^D$ be the natural projection, and define a red/blue colouring by 

\begin{equation}\label{blue-pts-def}  \blue_{\mathbf{e},\theta} := \{ n \in [N] : \theta n \in \bigcup_{j=1}^M (x_j + \pi(A_{\mathbf{e}}) \}, \end{equation}

\begin{equation}\label{red-pts-def} \red_{\mathbf{e},\theta} := [N] \setminus \blue_{\mathbf{e},\theta}.\end{equation}

Suppose henceforth that $N > D^{C_2 D^2}$, this being stronger than needed for some results but necessary in the worst case. We claim that with high probability there is no blue progression of length $3$.

\begin{proposition}\label{blue-prop-main} Suppose that $N > D^{C_2D^2}$. Then 
\[  \P_{\theta, \mathbf{e}} \big( \mbox{$\blue_{\theta, \mathbf{e}}$ has no 3-term progression}   \big) \geq 1 - O(N^{-1}). \]
\end{proposition}

In dealing with the progressions in the red points, we introduce a specific set $\Theta$ of rotations $\theta \in \T^D$ which we wish to consider. We call $\Theta$ the set of \emph{diophantine} $\theta$: the terminology is not standard, but the word diophantine is used in similar ways in other contexts. The precise definition of $\Theta$ is given in Section \ref{diophantine-sec} below. Roughly, $\theta$ is disqualified from $\Theta$ if the orbit $\{ \theta n : n \leq N\}$ exhibits certain pathological behaviours such as being highly concentrated near $0$, or having long subprogressions almost annihilated by a large set of characters on $\T^D$. For our discussion in this section, the important fact about $\Theta$ is that diophantine elements are (highly) generic in the sense that 
\begin{equation}\label{dio-generic} \mu_{\T^D} (\Theta) \geq 1 - O(N^{-1}).\end{equation}
This is proven in Section \ref{diophantine-sec}, specifically Proposition \ref{prop31}, where the definition of $\Theta$ is given. 

Now we claim that, conditioned on the event that $\theta$ is diophantine, with high probability there is no red progression of length $N^{1/r}$.

\begin{proposition}\label{red-prop-main} Suppose that $N > D^{C_2D^2}$. Then
\[  \P_{\mathbf{e}} \big( \mbox{$\red_{\theta, \mathbf{e}}$ has no $N^{1/r}$-term progression} \; | \; \theta \in \Theta  \big) \geq 1 - O(N^{-1}), \]
where $\Theta \subset \T^D$ denotes the set of diophantine elements.\end{proposition}

The proof of Proposition \ref{blue-prop-main} is relatively straightforward and is given in Section \ref{blue-ap-sec}. The proof of Proposition \ref{red-prop-main} is considerably more involved and occupies the rest of the paper.

Let us now show how Theorem \ref{mainthm-1} follows essentially immediately from Propositions \ref{blue-prop-main} and \ref{red-prop-main}. 

\begin{proof}[Proof of Theorem \ref{mainthm-1}] (assuming Propositions \ref{blue-prop-main} and \ref{red-prop-main}) First observe that, with $D = C_3 r^2$, the conditions required in Propositions \ref{blue-prop-main} and \ref{red-prop-main} will be satisfied if $N > e^{C r^4 \log r}$ for a sufficiently large $C$. Also, in proving Theorem \ref{mainthm-1} we may clearly assume that $r$ is sufficiently large.

First note that by Proposition \ref{red-prop-main} and \eqref{dio-generic} we have
\[ \P_{\theta, \mathbf{e}} \big( \mbox{$\red_{\theta, \mathbf{e}}$ has no $N^{1/r}$-term progression}   \big) \geq 1 - O(N^{-1}). \]
This and Proposition \ref{blue-prop-main} imply that there is some choice of $\theta, \mathbf{e}$ (in fact, a random choice works with very high probability) for which simultaneously $\blue_{\theta, \mathbf{e}}$ has no 3-term progression and $\red_{\theta, \mathbf{e}}$ has no $N^{1/r}$-term progression. This completes the proof of Theorem \ref{mainthm-1}.
\end{proof}

Let us consider the task of proving Proposition \ref{red-prop-main} in a little more detail. Let 
\begin{equation} \label{x-def} X := N^{1/r},\end{equation} a notational convention we will retain throughout the paper.

It suffices to show that if $N > D^{C_2D^2}$ and if $\theta \in \Theta$ is diophantine then for each fixed progression $P = \{n_0 +  dn : n \leq X \} \subset [N]$ of length $X$,
\begin{equation}\label{ap-prob-7} \P_{\mathbf{e}} (P \cap \blue_{\theta, \mathbf{e}} = \emptyset) \leq N^{-3}.\end{equation} Indeed, there are fewer than $N^2$ choices of $n_0$ and $d$, the start point and common difference of $P$, and so Proposition \ref{red-prop-main} follows from \eqref{ap-prob-7} by the union bound.

The task, then is to show that (with very high probability) $\theta P \subset \T^D$ intersects one of the annuli $x_j + \pi(A_{\mathbf{e}})$. To achieve this, we proceed in two distinct stages. Denoting by $P_{\init} = \{ n_0 + dn : n \leq X/2\}$ the first half of $P$, we show that $\theta P_{\init}$ at some point enters the interior of some ball $x_j + \pi(B_{\rho/10}(0))$, where $B_{\eps}(0) \subset \R^D$ is the Euclidean ball of radius $\eps$, and as usual $\pi : \R^D \rightarrow \T^D$ is projection. This we call the first step.

\begin{proposition}[First step]\label{first-step}
Suppose that $N > D^{C_2 D^2}$. Let $\theta \in \Theta$ be diophantine. Let $d \leq N/X$, and consider a progression $P_{\init} = \{n_0 + dn : n \leq X/2\}$. Then $\theta P_{\init}$ intersects $x_j + \pi(B_{\rho/10}(0))$ for some $j \in \{1,\dots, M\}$.
\end{proposition}

Once we have a point of $\theta P_{\init}$ in $x_j + \pi(B_{\rho/10}(0))$ we use the remaining half of $P$ to intersect the annulus $x_j + \pi(A_{\mathbf{e}})$. This we call the second step.

\begin{proposition}[Second step]\label{second-step} Suppose that $N \geq D^{C_2D^2}$. Let $d \leq N/X$. Suppose that $\theta \in \Theta$ is diophantine, and consider a progression $\dot{P} = \{ dn : n \leq X/2\}$. Then \[ \P_{\mathbf{e}} (\mbox{there is $y \in \pi(B_{\rho/10}(0))$ such that $(y + \theta \dot{P}) \cap \pi(A_{\mathbf{e}}) = \emptyset$})  \leq N^{-3}.\]
\end{proposition}

Together, Propositions \ref{first-step} and \ref{second-step} imply \eqref{ap-prob-7} and hence, as explained above, Proposition \ref{red-prop-main}. Indeed, Proposition \ref{first-step} implies that there is some $n_0 + n_1 d \in P_{\init}$ (that is, some $n_1 \leq X/2$ such that $\theta (n_0 + n_1 d) \in x_j + \pi(B_{\rho/10}(0))$, for some $j \in \{1,\dots, M\}$). Now apply Proposition \ref{second-step}, taking $y = \theta (n_0 + n_1d) - x_j$. With probability $1 - O(N^{-3})$ in the choice of $\mathbf{e}$, this provides some $n_2 d \in \dot{P}$ (that is, $n_2 \leq X/2$) such that $y + \theta n_2 d \in \pi(A_{\mathbf{e}})$.

If $n_1, n_2$ can both be found (which happens with probability $1 - O(N^{-3})$ in the choice of $\mathbf{e}$) then
\[ \theta(n_0 + (n_1 + n_2) d) - x_j  \in \pi(A_{\mathbf{e}}),\] which means that $n_0 + (n_1 + n_2)d$ is coloured blue. This establishes \eqref{ap-prob-7}.

The remaining tasks in the paper are therefore as follows.

\begin{itemize}
\item Establish Proposition \ref{blue-prop-main} (blue 3-term progressions). This is relatively straightforward and is covered in Section \ref{blue-ap-sec}.
\item Give the full definition of $\Theta$, the set of diophantine $\theta$, and prove \eqref{dio-generic}. This is carried out in Section \ref{diophantine-sec}.
\item Prove Proposition \ref{first-step}, the ``first step'' for the red progressions. This is carried out in Section \ref{first-step-sec}.
\item Prove Proposition \ref{second-step}, the ``second step'' for the red progressions.\end{itemize}

The first three tasks, as well as an outline of the fourth, are carried out in Part III of the paper (Sections \ref{blue-ap-sec}, \ref{diophantine-sec}, \ref{first-step-sec} and \ref{second-step-sec} respectively).

The fourth task (proving Proposition \ref{second-step}) is very involved. We give a technical outline in Section \ref{second-step-sec} which reduces it to the task of proving three further propositions: Propositions \ref{main-geo-num}, \ref{measure-compare} and \ref{main-quad-gap}. These propositions are established in Parts IV and V of the paper. 

\part{Monochromatic progressions}

\section{No blue 3-term progressions}\label{blue-ap-sec}

In this section we establish Proposition \ref{blue-prop-main}. Let us recall the statement.

\begin{blue-prop-main-rpt} Suppose that $N > D^{C_2D^2}$. Then 
\[  \P_{\theta, \mathbf{e}} \big( \mbox{$\blue_{\theta, \mathbf{e}}$ has no 3-term progression}   \big) \geq 1 - O(N^{-1}). \]
\end{blue-prop-main-rpt}
Recall that the definition of $\blue_{\theta, \mathbf{e}}$ is given in \eqref{blue-pts-def}. The following lemma is a quantitative version of Behrend's observation that no three points on a sphere lie in arithmetic progression. The spherical version of this was already used in \cite{green-wolf}. 

\begin{lemma}\label{parallel-law}
Fix $e \in [-\frac{1}{D^4} , \frac{1}{D^4} ]^{D(D+1)/2}$. Suppose that $u , u + v, u + 2v$ all lie in $A_{e}$, where the ellipsoidal annulus $A_{e}$ is defined as in \eqref{annulus-def}, but with $e$ fixed. Then $\Vert v \Vert_2  \leq \frac{1}{2}N^{-2/D}$. 
\end{lemma}
\begin{proof}
We have the parallelogram law
\[ 2 \Vert y \Vert_2^2 = \Vert x + 2y \Vert_2^2 + \Vert x\Vert_2^2 - 2 \Vert x+y \Vert_2^2.\]
Applying this with $x = (1 + \sigma(e))u$, $y = (1 + \sigma(e)) v$ gives
\[ \Vert (1 + \sigma(e)) v \Vert_2^2 = \Vert y \Vert_2^2 \leq \rho^2 -  (\rho - N^{-4/D})^2 < \frac{1}{10} N^{-4/D}.\]
The result now follows from \eqref{e-operator}.
\end{proof}

Condition on the event that $\mathbf{e} = e$ and let $\theta \in \T^D$ be chosen uniformly at random. Suppose that $n, n + d, n + 2d$ are all coloured blue. Then for some $i,j,k \in \{1,\dots, M\}$ we have $\theta n \in x_i + \pi(A_{e})$, $\theta (n+d) \in x_j + \pi(A_{e})$, $\theta (n+2d) \in x_k + \pi(A_{e})$.  Since $\theta (n+2d) - 2\theta (n+d) + \theta n = 0$ we have $x_i - 2x_j + x_k \in \pi(A_{e}) - 2 \pi(A_{e}) + \pi(A_{e})$, and so $\Vert x_i - 2x_j + x_k \Vert_{\T^D} \leq 4 \rho$ since every $x \in \pi(A_{e})$ has $\Vert x \Vert_{\T^D} = \Vert \pi^{-1} x \Vert_{\infty} \leq \Vert \pi^{-1} x \Vert_2 \leq \rho$. By the construction of the points $x_1,\dots, x_M$ (specifically, Proposition \ref{prop21} (1)) it follows that $i = j = k$. 

We apply Lemma \ref{parallel-law} with $u = \pi^{-1}(\theta n - x_i)$, $v = \pi^{-1}(\theta d)$, both of which lie in $B_{2\rho}(0) \subset B_{1/10}(0)$. Since $\pi(u + \lambda v) = \theta (n + \lambda d) - x_i$ for $\lambda \in \{0,1,2\}$, we see that $\pi(u + \lambda v) \in \pi( A_{e})$, and therefore since $u + \lambda v \in B_{1/5}(0)$ and $A_{e} \subset B_{1/5}(0)$ we have $u + \lambda v \in A_{e}$.

It follows from Lemma \ref{parallel-law} that $\Vert v \Vert_2 \leq \frac{1}{2}N^{-2/D}$. Therefore $\Vert \theta d \Vert_{\T^D} = \Vert \pi^{-1}(\theta d) \Vert_{\infty} \leq \Vert \pi^{-1}(\theta d) \Vert_2 \leq \frac{1}{2}N^{-2/D}$.

If $d \neq 0$ then, with $\theta \in \T^D$ chosen randomly, $\theta d$ is uniformly distributed on $\T^D$. Therefore the probability that there is any blue 3-term progression $(n, n+d, n + 2d)$ with common difference $d$ is bounded above by the probability that $\theta d$ lies in the box $\{ x \in \T^D : \Vert x \Vert_{\T^d} \leq \frac{1}{2}N^{-2/D}\}$, a set of volume $N^{-2}$.

Therefore for any fixed $e \in [-\frac{1}{D^4} , \frac{1}{D^4} ]^{D(D+1)/2}$ we have, summing over the at most $N$ possible choices for $d$, that
\[  \P_{\theta} \big( \mbox{$\blue_{\theta, e}$ has no 3AP}) \geq 1 - O(N^{-1}),\] from which Proposition \ref{blue-prop-main} follows immediately by removing the conditioning on $\mathbf{e} = e$.

This completes the proof of Proposition \ref{blue-prop-main}. Note that the randomness of $\theta$ was vital, but the ability to choose $\mathbf{e}$ randomly here was irrelevant, since the analysis works for any fixed $\mathbf{e} = e$.

\section{Diophantine conditions}\label{diophantine-sec}

We turn now to the definition of $\Theta$, the set of ``diophantine'' rotations $\theta \in \T^D$. Here (as usual) $X = N^{1/r}$.

\begin{proposition}\label{prop31}
Suppose that $D = C_3 r^2$ and that $N \geq D^{D^2}$. Define $\Theta \subset \T^D$ to be the set of all $\theta$ satisfying the following two conditions:
\begin{enumerate}
\item 
For all $n \leq N$, \[
\dim \{ \xi \in \Z^D : |\xi| \leq D^{C_2},  \Vert n \xi \cdot \theta\Vert_{\T} \leq D^{C_2 D} X^{-1} \} <  4r.
\]
\item For all $d \leq N/X$, \[ \#\{ n \leq X : \Vert \theta d n \Vert_{\T^D} \leq X^{-1/D} \} \leq X^{9/10}.\]
\end{enumerate}
Then $\mu_{\T^D}(\Theta) \geq 1 - O(N^{-1})$.
\end{proposition}
Before launching into the proof, let us offer some informal explanation of conditions (1) and (2). Condition (2) is fairly self-explanatory, asserting that orbits $\{ \theta d n : n \leq X\}$ are not highly concentrated around $0$. Note for reference that the box $\{x \in \T^D : \Vert x \Vert_{\T^D} \leq X^{-1/D}\}$ has volume $\sim_D X^{-1}$, so on average on expects just $\sim_D 1$ points of the orbit $\{ \theta d n : n \leq X\}$ to lie in this box. Thus that (2) should hold generically is very unsurprising (though seemingly harder to prove than one might imagine).

Item (1) is harder to explain. With reference to the introductory discussion in Section \ref{outline-sec}, if some condition along these lines did not hold then the orbit $\theta n, 2 \theta n,3\theta n,\dots$ would be concentrated near subtori of high codimension, and this makes it easier for such orbits to evade our union of annuli. Thus one should expect this condition to come up in the proof of Proposition \ref{first-step}, and we shall see in the next section that this is indeed the case. (It also comes up later in the proof of Proposition \ref{second-step}.)

\begin{proof}
We can address (1) and (2) separately and then take the intersection of the corresponding sets of $\theta$.

(1) Suppose that $\xi_1,\dots, \xi_{4r}$ are linearly independent over $\Q$. Then, as $x$ varies uniformly over $\T^D$, $(\xi_1\cdot x, \dots, \xi_{4r} \cdot x)$ is equidistributed (the proof is the same as in Proposition \ref{prop21}). It follows that for fixed $n$,
\[ \mu \{ \theta : \Vert n \xi_i \cdot \theta \Vert_{\T} \leq D^{C_2 D} X^{-1} \; \mbox{for $i = 1,\dots, 4r$} \} = (2D^{C_2 D})^{4r} N^{-4}.\] 
By the union bound, the measure of $\theta$ for which there exists $n \leq N$ such that $\Vert n \xi_i \cdot \theta \Vert \leq D^{C_2 D} N^{-1/r}$ for $i = 1,\dots, 4r$ is $\leq (2D^{C_2 D})^{4r} N^{-3}$, which is comfortably less than $N^{-2}$ with our assumptions on $r, D$ and $N$. 

Now if item (1) fails then there must be some choice of $\xi_1,\dots, \xi_{4r} \in \Z^D$, $|\xi_i| \leq D^{C_2}$, for which the preceding statement holds. The number of choices for this tuple is at most  $\big((3D^{C_2})^{D} \big)^{4r} < N$. The result follows by another application of the union bound.

(2) The proof is easier to read if we set $c := \frac{1}{10}$ throughout (also a similar claim with any $c > 0$ would suffice for our purposes). It is enough to show that 
\begin{equation}\label{enough-alpha} \mu_{\T^D} \{ \alpha : \# \{ n \leq X : \Vert \alpha n \Vert_{\T^D} \leq X^{-1/D}\} \geq X^{1-c} \} \leq X^{-c^2 D/2}.\end{equation}
Then one may take a union bound over all $\alpha = \theta d$, $d = 1,\dots, N$, noting that $N X^{-c^2 D} < N^{-1}$ with our assumptions on $D$ and $r$ (recall that $X = N^{1/r}$). To prove \eqref{enough-alpha} we employ an inductive approach based on the following claim.

\emph{Claim.} Suppose that $\alpha' \in \T^{D'}$ and that 
\begin{equation}\label{claim-assumption} \# \{ n \leq X : \Vert \alpha' n \Vert_{\T^{D'}} \leq X^{-(1 - c)/D'} \} \geq X^{1 -c}.\end{equation}Then there is some $\xi \in \Z^{D'}$, $0 < |\xi| \leq 2X^{3c/D'}$, such that $\Vert \xi \cdot \alpha' \Vert_{\T} \leq 2X^{3c - 1}$.

\emph{Proof of Claim.} Inside the proof of the claim, we drop the dashes for clarity (if we did not include them in the statement, it would make the subsequent deduction of Proposition \ref{prop31} (2) confusing). By Lemma \ref{cutoff-for-dio}, there is a smooth cutoff $\chi : \T^D \rightarrow [0,\infty)$ satisfying
\begin{enumerate}
\item $\chi(x) \geq 1$ for $\Vert x \Vert_{\T^D} \leq X^{-(1-c)/D}$;
\item $\int \chi \leq 5^D X^{c-1}$;
\item $\hat{\chi}(\xi) = 0$ for $|\xi| \geq X^{(1-c)/D}$; 
\end{enumerate}
Then if \eqref{claim-assumption} holds (remember, we have dropped the dashes)
\begin{align*} X^{1-c}  \leq \sum_{n \leq X} \chi(\alpha n)  & = \sum_{\xi} \hat{\chi}(\xi) \sum_{n \leq X} e(\xi \cdot \alpha n) \\ & \leq 5^DX^{c-1} \sum_{|\xi| \leq X^{(1-c)/D}}| \sum_{n \leq X} e(\xi \cdot \alpha n)| \\ & < \frac{1}{2}X^{2c-1} \sum_{|\xi| \leq X^{(1-c)/D}}\min(X, \Vert \xi \cdot \alpha \Vert_{\T}^{-1}).\end{align*} (The factor $\frac{1}{2}$ is a minor technical convenience for later.)
Therefore
\[ 2X^{2 - 3c} <  \sum_{|\xi| \leq X^{(1-c)/D}} \min(X, \Vert \xi \cdot \alpha \Vert_{\T}^{-1}).\]
The contribution from those $\xi$ with $\Vert \xi \cdot \alpha\Vert_{\T} \geq X^{3c - 1}$ is bounded above by $(3 X^{(1 - c)/D})^D X^{1 - 3c} < X^{2 - 3c}$. Writing $\Xi \subset [-X^{1/D}, X^{1/D}]^D$ for the set of $\xi$, $|\xi| \leq X^{1/D}$, with $\Vert \xi \cdot \alpha\Vert_{\T} \leq X^{3c - 1}$, it follows that $X^{1 - 3c} < |\Xi|$.
By the pigeonhole principle, dividing $[-X^{1/D}, X^{1/D}]^D$ into $X^{1 - 3c}$ boxes of sidelength $2X^{3c/D}$, we see that there are distinct $\xi_1, \xi_2 \in \Xi$ with $| \xi_1 - \xi_2| \leq 2X^{3c/D}$.
Taking $\xi := \xi_1 - \xi_2$ completes the proof of the claim.

Let us resume the proof of Proposition \ref{prop31} (2).
Set $D_0 := \lceil (1 - c)D\rceil$. We prove, by induction on $j = 0,1,\dots, D - D_0$, the following statement: The measure of all $\alpha^{(j)} \in \T^{D_0 + j}$ such that 
\begin{equation}\label{alpha-j} \# \{ n \leq X : \Vert \alpha^{(j)} n \Vert_{\T^{D_0 + j}} \leq X^{-1/D} \} \geq X^{1-c}\end{equation} is at most $X^{- cj}$. This result is trivial for $j = 0$, and the case $j = D - D_0$ gives the result we are trying to prove, i.e. \eqref{enough-alpha}, which implies Proposition \ref{prop31} (2) as explained above.

To deduce the case $j$ from $j-1$, apply the claim with $D' = D_0 + j$. Since $D' \geq (1 - c)D$,  condition \eqref{alpha-j} implies that 
\[ \# \{ n \leq X : \Vert \alpha^{(j)} n \Vert_{\T^{D'}} \leq X^{-(1-c)/D'} \} \geq X^{1 - c},\] and so by the claim there is some $\xi \in \Z^{D'}$, $0 < |\xi| \leq X^{3c/D_0}$ such that $\Vert \xi \cdot \alpha^{(j)} \Vert_{\T} \leq 2X^{3c - 1}$. Suppose that the last non-zero coordinate of $\xi$ is $\xi_{D'} = \xi_{D_0 + j}$ (the other possibilities can be treated similarly with very minor notational changes). Form $\alpha^{(j-1)} \in \T^{D_0 + j - 1}$ by restricting $\alpha^{(j)}$ to the first $D_0 + j - 1$ coordinates (i.e. by dropping the last coordinate $\alpha^{(j)}_{D_0 + j}$). Then, since $\Vert \alpha^{(j-1)} n \Vert_{\T^{D_0 + j -1}} \leq \Vert \alpha^{(j)} n \Vert_{\T^{D_0 + j}}$, the hypothesis \eqref{alpha-j} is satisfied by $\alpha^{(j-1)}$. By the inductive hypothesis, the measure of possible $\alpha^{(j-1)}$ is at most $X^{-c(j-1)}$. For each such $\alpha^{(j-1)}$, and for each fixed $\xi$, the final coordinate satisfies $\Vert \gamma + \xi_{D_0 + j} \alpha^{(j)}_{D_0 + j} \Vert_{\T}  = \Vert \xi \cdot \alpha^{(j)}\Vert_{\T} \leq 2X^{3c - 1}$, where $\gamma$ depends only on $\alpha^{(j-1)}$ and the first $D_0 + j - 1$ coordinates of $\xi$. As $\alpha^{(j)}_{D_0 + j}$ varies uniformly over $\T$, so does $\xi_{D_0 + j} \alpha^{(j)}_{D_0 + j}$, so the probability of this event is $\leq 4X^{3c - 1}$. 

Summing over all possible choices of $\xi$ (of which there are at most $(3 X^{3c/D_0})^{D}$) it follows that the total measure of $\alpha^{(j)}$ satisfying \eqref{alpha-j} is bounded above by 
\[ 4X^{3c - 1} \cdot (3 X^{3c/D_0})^{D} \cdot X^{-c(j-1)} < X^{-cj}.\] (The key calculation for this last step is that $3c - 1 + \frac{3c}{1 - c} < - c$, which is certainly true for $c = \frac{1}{10}$, and we used the assumption that $N > D^{D^2}$ to comfortably absorb the $3^D$ term into a tiny power of $X = N^{1/r}$.)

This completes the inductive step, and hence \eqref{alpha-j} is true for all $j$. As previously remarked, the case $j = D - D_0 \geq cD/2$ gives Proposition \ref{prop31} (2).
\end{proof}

\emph{Remark.} Whilst the $\frac{9}{10}$ in Proposition \ref{prop31} (2) can easily be improved a little, it cannot be improved very far by the method we have employed here. One feels that, even with $X^{1/10}$ on the right hand side in (2), this should be a highly likely event in $\theta$, but I do not know how to prove anything in this direction. We state this (where, for simplicity, we have set $D = r^2$) as a separate question.

\begin{question}
Define $\Theta' \subset \T^{r^2}$ to be the set of all $\theta$ for which $\# \{ n \leq N^{1/r} : \Vert \theta d n \Vert_{\T^{r^2}} \leq N^{-1/r}\} \leq N^{1/10r}$ for all $d \leq N^{1 - 1/r}$. Is $\mu_{\T^{r^2}}(\Theta') \geq \frac{1}{2}$, for $N$ large enough in terms of $r$?
\end{question}

We also remark that something in the direction of Proposition \ref{prop31} (2) seems essential for obtaining a relatively small exponent of $r$ in Theorem \ref{mainthm-1}, but one can still obtain \emph{some} fixed exponent there using only consequences of Proposition \ref{prop31} (1), though this would require some reorganisation of the paper.

\section{The first step - red progressions enter balls}\label{first-step-sec}

In this section we prove Proposition \ref{first-step}. Thus, let $x_1,\dots, x_M$ be as in Proposition \ref{prop21}, let $\theta \in \Theta$ be diophantine (where $\Theta$ is defined in Proposition \ref{prop31}), and set $X := N^{1/r}$. Recall that $\rho := D^{-4}$, and recall that, in the statement of Proposition \ref{first-step}, we encounter $P_{\init} := \{n_0 + nd : n \leq X/2\}$.

\begin{proof}[Proof of Proposition \ref{first-step}] Set $\alpha := \theta d$. Set $n_1 := n_0 + \lfloor \frac{X}{4} \rfloor d$, so that $P_{\init} \supset \{ n_1 + nd : |n| \leq X/5\}$. Set
\begin{equation}\label{Lam-def} \Lambda := \{ \xi \in \Z^D : |\xi| < \rho^{-3}, \Vert \xi \cdot \alpha  \Vert_{\T} \leq \rho^{-2D} X^{-1}\}.\end{equation}
By the definition of $\Theta$ (specifically, item (1) of Proposition \ref{prop31}, and assuming that $C_2 \geq 12$) we have $\dim_{\Q}\Lambda < 4r$. By Proposition \ref{prop21} there is some $j$ such that 
\begin{equation}\label{xi-bd-8} \Vert \xi \cdot (x_j - \theta n_1) \Vert_{\T} \leq 10^{-2}\end{equation} for all $\xi \in \Lambda$.
We claim that $\theta P_{\init}$ intersects the ball $x_j + \pi(B_{\rho/10}(0))$. To this end, take a function $\chi : \T^D \rightarrow \R$ with the following properties: 
\begin{enumerate}
\item $\chi(x) \leq 0$ outside of $\pi(B_{\rho/10}(0))$;
\item $\hat{\chi}$ is real and non-negative;
\item $\hat{\chi}(\xi)$ is supported on $|\xi| \leq \rho^{-3}$;
\item $\int \chi  = 1$.
\item $\int |\chi| \leq 3$.
\end{enumerate}
Such a function is constructed in Lemma \ref{tera}. Take also a function $w : \Z \rightarrow [0,\infty)$ satisfying
\begin{enumerate}
\item[(a)] $w$ is supported on $[-X/5, X/5]$;
\item[(b)] $\hat{w} : \T \rightarrow \C$ is real and non-negative;
\item[(c)] $\sum_{n \in \Z} w(n) \geq X$;
\item[(d)] $|\hat{w}(\beta)| \leq 2^5 X^{-1} \Vert \beta \Vert_{\T}^{-2}$ for all $\beta \in \T$.
\end{enumerate}
For this, one can take a Fej\'er kernel: see Lemma \ref{fejer} for details.

Then it is enough to show that 
\begin{equation}\label{to-show-8} \sum_{n \in \Z} w(n)\chi(\theta (n_1 + nd) - x_j) > 0 .  \end{equation}
Indeed, if this holds then there must be some $n \in \Supp(w)$ (and hence $|n| \leq X/5$) such that $\chi (\theta (n_1 + nd) - x_j) > 0$, which means that $\theta (n_1 + nd) - x_j \in \pi(B_{\rho/10}(0))$. Thus $n' := n_1 + nd$ lies in $P_{\init}$ and $\theta n' \in x_j + \pi(B_{\rho/10}(0))$, as required.

It remains to establish \eqref{to-show-8}.
By Fourier inversion on $\chi$ (recalling that $\alpha = d \theta$), the LHS of \eqref{to-show-8} is
\begin{align} \nonumber \sum_{\xi \in \Z^D} \hat{\chi}(\xi) e(\xi  \cdot (n_1 \theta - & x_j)) \sum_{n \in \Z} w(n) e(\xi \cdot  \alpha n) \\ & = \sum_{\xi \in \Z^D} \hat{\chi}(\xi) e(\xi \cdot (n_1 \theta - x_j))\hat{w}(-\xi \cdot \alpha) \nonumber \\ & = \sum_{\xi \in \Z^D} \hat{\chi}(\xi) \cos (2\pi\xi \cdot (n_1 \theta - x_j))\hat{w}(-\xi \cdot \alpha) , \label{fourier-inv} \end{align} where the last line follows by taking real parts, noting that the LHS of \eqref{to-show-8} is real, as are $\hat{\chi}$, $\hat{w}$.
Note that by (3) the sum here is supported where $|\xi| \leq \rho^{-3}$. We now divide into the contributions from the following three different classes of $\xi$: (i) $\xi = 0$; (ii) $\xi \in \Lambda$ and (iii) $|\xi| \leq \rho^{-3}$ but $\xi \notin \Lambda$.

(i) The contribution from $\xi = 0$ is $\hat{\chi}(0) \hat{w}(0) = (\int \chi) (\sum_{n \in \Z} w(n)) \geq X$, using the properties of $\chi$ and $w$ listed above.

(ii) If $\xi \in \Lambda$ then by \eqref{xi-bd-8} we have $\cos (2\pi \xi \cdot(n_1 \theta - x_j)) > 0$. Since $\hat{\chi}$, $\hat{w}$ are both real and non-negative, the contribution of these terms to \eqref{fourier-inv} is non-negative.

(iii) By the triangle inequality, the contribution to \eqref{fourier-inv} from these $\xi$ is bounded above by
\begin{equation}\label{third-cont}   \sum_{\xi \notin \Lambda, |\xi| \leq \rho^{-3}} |\hat{\chi}(\xi)| |\hat{w}(-\xi \cdot \alpha) | \leq \big( \sum_{|\xi| \leq \rho^{-3}} |\hat{\chi}(\xi)|\big) \sup_{\Vert \beta \Vert_{\T} > \rho^{-2D} X^{-1}} |\hat{w}(\beta)|,\end{equation} where the second step follows by the definition \eqref{Lam-def} of $\Lambda$.
Now from (5) above we have $|\hat{\chi}(\xi)| \leq \int |\chi| \leq 3$ for all $\xi$, and so 
\begin{equation}\label{xi-ell-1} \sum_{|\xi| \leq \rho^{-3}} |\hat{\chi}(\xi)| \leq 3 (3\rho^{-3})^D < 2^{-6}\rho^{-4D}\end{equation} (here, of course, we have used the fact that $D$ is sufficiently large).
By property (4) of $w$, we have
\begin{equation}\label{w-prop} \sup_{\Vert \beta \Vert_{\T} > \rho^{-2D} X^{-1}} |\hat{w}(\beta)| \leq 2^5 \rho^{4D} X.\end{equation}
Combining \eqref{third-cont}, \eqref{xi-ell-1} and \eqref{w-prop}, we see that the total contribution from $\xi$ in (iii) is at most $X/2$ in magnitude.

Summing the estimates we have obtained under (i), (ii) and (iii), it follows from \eqref{fourier-inv} that the LHS of \eqref{to-show-8} is at least $X + 0 - X/2 = X/2$ and so is indeed positive, which is what we aimed to prove.
\end{proof}

\section{The second step -- red progressions hit annuli} \label{second-step-sec}

In this section, we give the proof of Proposition \ref{second-step}, conditional upon three substantial results which we will prove later in the paper. The proofs of these results may be read independently of one another. 

Throughout this section set $X = N^{1/r}$ (as usual) and let $d \leq N/X$. Recall that $\dot{P}$ denotes the progression $\{ nd : n \leq X/2\}$. Assume through the section that $\theta \in \Theta$ is diophantine in the sense of Proposition \ref{prop31}.

The first substantial result we need is the following proposition, in which we put a multidimensional structure on $\Z$ suitable for analysing the metric behaviour of the orbit $\theta \dot{P} = \{ \theta d n : n \leq X/2\} \subset \T^D$. Proposition \ref{main-geo-num} is the main result we will use in subsequent sections. 

In the statement of this result, recall that $\pi : \R^D \rightarrow \T^D$ is the natural projection and we write $\pi^{-1} : \T^D \rightarrow \R^D$ for the unique partial inverse map taking values in $(-\frac{1}{2}, \frac{1}{2}]^D$. In particular $\Vert x \Vert_{\T^D} = \Vert \pi^{-1} x \Vert_{\infty}$. If $w_1,\dots, w_s \in \R^D$ then we define the \emph{volume} $\vol(w_1,\dots, w_s)$ to be $\sqrt{ \det (\langle w_i, w_j \rangle)_{1 \leq i,j \leq s}}$, that is to say the square root of the determinant of the Gram matrix associated to the $w_i$. For more properties of this notion, see Section \ref{volume-sec}.

\begin{proposition}\label{main-geo-num}
Suppose that $N \geq D^{C_2 D^2}$, let $d \leq N/X$, and suppose that $\theta \in \Theta$ is diophantine. Then there are positive integers $n_1,\dots, n_s$, $s \leq D$, such that if we write $L_i := \Vert \theta d n_i \Vert_{\T^D}^{-1}$ then\begin{enumerate}
\item $\prod_{i = 1}^s L_i \geq X^{1/80}$;
\item $\sum_{i = 1}^s n_i L_i \leq X/2$;
\item if we set $v_i := \pi^{-1}(\theta d n_i) \in \R^D$ and $w_i :=v_i/\Vert v_i \Vert_2$, the unit vector in the direction of $v_i$, then $\vol (w_1,\dots, w_s) \geq D^{-C_1 D}$;
\item $L_i \leq X^{C_1/D}$ for all $i$.
\end{enumerate}
\end{proposition}
Note that the $n_i$, the $L_i$ and $s$ can (and will) depend on $\theta$ and $d$. We will not indicate this dependence later in the section. Various bounds we state will be uniform in $\theta \in \Theta$ and $d \leq N/X$ and so this dependence is ultimately unimportant.

To prove this result  we first develop, in Section \ref{volume-sec}, some basic properties of volume. In Section \ref{slab-avoid-sec} we prove a key technical result stating that the orbit $\theta \dot{P}$ cannot be too concentrated near (the image in $\T^D$ of) the unit ball of a subspace of $\R^D$ of dimension $(1 - \eps) D$: here we use the diophantine assumption on $\theta$. In Section \ref{geom-num} we finally prove Proposition \ref{main-geo-num}, first using standard geometry of numbers techniques (Minkowski's second theorem) and then refining the information those give using the estimates of Sections \ref{volume-sec} and \ref{slab-avoid-sec}.

The multidimensional structure resulting from Proposition \ref{main-geo-num} gives a new ``basis'' relative to which we can try to understand the behaviour of the (random) quadratic form $\Vert (1 + \sigma(\mathbf{e}) )x \Vert_2^2$ (which occurs in the definition of the annuli $A_{\mathbf{e}}$) along the orbit $\theta \dot{P}$. Unfortunately, a uniformly random $\mathbf{e}$ does not transfer to a uniformly random quadratic form with respect to this new basis. However, it turns out that the volume condition (3) above gives some control of the former in terms of the latter. 

The following, the second main ingredient in the proof of Proposition \ref{second-step}, is the technical result we need. Here, we define a map $\sigma : \R^{n(n+1)/2} \rightarrow \Sym_n(\R)$ as in Section \ref{sec4}: to any tuple $\mathbf{x} \in \R^{n(n+1)/2}$, we associate a symmetric matrix $\sigma(\mathbf{x}) \in \Sym_n(\R)$ as follows: $(\sigma(\mathbf{x}))_{ii} = x_{ii}$, $(\sigma(\mathbf{x}))_{ij} = \frac{1}{2}x_{ij}$ for $i < j$, and $(\sigma(\mathbf{x}))_{ij} = \frac{1}{2}x_{ji}$ for $i > j$. We will use this for $n = D$ and for $n = s$. Which we are talking about should be clear from context, and the abuse of notation seems preferable to the clutter of further subscripts.

\begin{proposition}\label{measure-compare}
Let $w_1,\dots, w_{s} \in \R^D$, $s \leq D$, be linearly independent unit vectors. Write $f : \R^{D(D+1)/2} \rightarrow \R^{s(s+1)/2}$ for the map defined by \[ f(\mathbf{x}) = \sigma^{-1} \big( (\langle (I + \sigma(\mathbf{x})) w_i, (I + \sigma(\mathbf{x}))w_j\rangle)_{1 \leq i, j \leq s} \big)\] Then for any open set $U \subset \R^{s(s+1)/2}$ we have \[ \P_{\mathbf{e}} (f(\mathbf{e}) \in U) \leq D^{4D^2} \vol(w_1,\dots, w_{s})^{-D-1} \mu(U),\] where $\mu$ denotes Lebesgue measure. 
\end{proposition}

This is the easiest of our three main ingredients and is a fairly standard argument in calculus/linear algebra. It is given in Section \ref{measure-compare-sec}. 

The third and final main input to the proof of Proposition \ref{second-step} is the following result, which says that with overwhelming probability, random quadratic forms have (very) small gaps in the heart of their range, subject to a few natural conditions.

\begin{proposition}\label{main-quad-gap}
Let $B \geq 1$ be an exponent and let $Q \geq 1$ be a parameter. Suppose that $s \geq C_1 B^2$ is an integer. Let $L \geq (QB)^{C_2B}$.
Let $L_1,\dots, L_s$ be lengths with $L_i \in [L, L^{1+1/48}]$. Choose an $\frac{1}{2}s(s+1)$-tuple $a = (a_{ij})_{1 \leq i \leq j \leq s}$ of coefficients by choosing the $a_{ij}$  independently and uniformly at random from $[-Q, Q]$, except for the diagonal terms $a_{ii}$ which are selected uniformly from $[32, Q]$.  For $b = (b_1,\dots, b_s)$ and $c \in \R$ and write $q_{a,b.c} : \R^s \rightarrow \R$ for the quadratic form defined by $q_{a,b,c}(t) := \sum_{i \leq j} a_{ij} t_i t_j + \sum_i b_i t_i + c$. Let $\Sigma = \Sigma(a)$ be the event that the set 
\begin{equation}\label{quad-set} \{ q_{a,b,c}(\frac{x_1}{L_1},\dots, \frac{x_s}{L_s}) : 0 \leq x_i < L_i\}\end{equation} is $L^{-B}$-dense in $[\frac{1}{2}, \frac{3}{2}]$ \textup{(}that is, intersects that $L^{-B}$-neighbourhood of every point in $[\frac{1}{2}, \frac{3}{2}]$\textup{)} for all $b,c$ satisfying $|b_i| \leq Q$, $|c| \leq \frac{1}{4}$ and $b_i^2 - 4a_{ii}c < 0$ for all $i$. Then $\P_a(\Sigma(a)) = 1 - O(L^{-Bs/16})$.
\end{proposition}

This is the most substantial of our three main ingredients. A rough outline of the proof is as follows. First, we establish a kind of preliminary version of the result with a smaller number $s_0 = O(B)$ of variables, but with a much weaker exceptional probability of $O(L^{-B})$ which is insufficient, by itself, for our application. Indeed, we will be applying Proposition \ref{main-quad-gap}  for each of the $N/X = N^{1 - 1/r}$ values of the common difference $d$ with (roughly) $L \sim X^{1/D}$, $B \asymp r$. Since $X = N^{1/r}$, an exceptional probability of $O(L^{-B})$ is then roughly $O(N^{-C/D})$, nowhere near good enough to apply a union bound over the choices of $d$.

This preliminary variant uses a version of the Hardy-Littlewood circle method (basically the Davenport-Heilbronn variant of that method) but with a random quadratic form, which adds some features not usually seen in the method, although in essence it makes the analysis easier. However, we need some inputs such as a result on the probability that the determinant of a random matrix is small (Section \ref{random-matrix-sec}) for which we do not know references in the literature. 

We then use what appears to be a novel kind of amplification trick to prove Proposition \ref{main-quad-gap} itself. The basic idea is to efficiently pack edge-disjoint copies of the complete graph $K_{s_0}$ into $K_s$, which encodes a way of applying the preliminary variant in a large number of independent instances. To get the best exponents in our main theorems, we need to do this as efficiently as possible, and to achieve this we use a construction based on lines in projective space over $\F_p$, for a suitable prime $p$. 

There is one further complexity, which is that edges of the complete graph $K_s$ only encode \emph{off-diagonal} coefficients $a_{ij}$, $1 \leq i < j \leq s$. Whilst the copies of $K_{s_0}$ in $K_s$ are edge-disjoint, they are certainly not \emph{vertex}-disjoint. To ensure that the applications of the preliminary version really are independent, we need to build in the capability of holding the diagonal entries $a_{ii}$ fixed while only selecting the off-diagonal entries at random.

\begin{proof}[Proof of Proposition \ref{second-step}] (assuming Propositions \ref{main-geo-num},  \ref{measure-compare} and \ref{main-quad-gap}) We start by invoking Proposition \ref{main-geo-num}. The contribution to $\prod_{i = 1}^s L_i$ from those $i$ such that $L_i \leq X^{1/160 D}$ is at most $X^{1/160}$. Removing these $i$ (which may involve reducing $s$) we may assume that $X^{1/160 D} \leq L_i \leq X^{C_1/D}$ for all $i$. Dividing into $O(\log C_1) < C_1/160$ ranges we may further assume that that there is some $L$, 
\begin{equation}\label{L-bds} X^{1/2^8 D} \leq L \leq X^{C_1/D},\end{equation} such that $L_i \in [L, L^{1 + 1/48}]$ for all $i$, at the expense of weakening Proposition \ref{main-geo-num} (1) to
\begin{equation}\label{prod-new-lower} \prod_{i = 1}^s L_i \geq X^{1/C_1}.\end{equation} The upper bound in \eqref{L-bds} then implies that 
\begin{equation}\label{s-lower-9} s \geq D/C_1^2.\end{equation}

Now recall that our task is to show that, with high probability in $\mathbf{e}$,  for all $y \in \pi(B_{\rho/10}(0))$ the orbit $y + \theta \dot{P}$ contains a point in $\pi(A_\mathbf{e})$, where 
\[ A_{\mathbf{e}} := \{ x \in \R^D : \rho - N^{-4/D} < \Vert (1 + \sigma(\mathbf{e})) x \Vert_2 \leq \rho \}.\] 
Let $z := \pi^{-1}(y) \in \R^D$, thus 
\begin{equation}\label{z-norm} \Vert z \Vert_2 \leq \rho/10.\end{equation}
Recall that $v_i := \pi^{-1}(\theta d n_i)$. Thus, by Proposition \ref{main-geo-num}, any element of the form $\pi( \sum_i \ell_i v_i )$, $\ell_i < L_i$, lies in $\theta \dot{P}$ (since $\pi( \sum_i \ell_i v_i ) =  (\sum_i \ell_i n_i) d \theta$).
Thus it suffices to show that with probability $1 - O(N^{-3})$ in the random choice of $\mathbf{e}$, any set of the form
\[ \{ z + \sum_i \ell_i v_i : \ell_i \in \N, 0 \leq \ell_i < L_i \}, \quad \Vert z \Vert_2 < \rho/10, \] has nontrivial intersection with $A_{\mathbf{e}}$, that is to say that there is some choice of the $\ell_i$ such that 
\begin{equation}\label{annulus-cond} \rho - N^{-4/D} < \Vert (I + \sigma(\mathbf{e})) (z + \sum_i \ell_i v_i ) \Vert_2 < \rho.\end{equation}
In other words it is enough that 
\begin{equation}\label{ann-hit} 1 - N^{-4/D} \leq q_{a(\mathbf{e}), b(\mathbf{e}), c(\mathbf{e})}(\frac{\ell_1}{L_1}, \dots, \frac{\ell_{s}}{L_{s}}) \leq 1,\end{equation} where the quadratic form is given by

\begin{equation} q_{a(\mathbf{e}), b(\mathbf{e}), c(\mathbf{e})}(\frac{\ell_1}{L_1}, \dots, \frac{\ell_{s}}{L_{s}})  := \rho^{-2} \Vert (I + \sigma(\mathbf{e}))(z + \sum_{i = 1}^{s} \ell_i v_i ) \Vert_2^2.\end{equation} 

One computes
\[ c(\mathbf{e}) := \rho^{-2} \Vert (I + \sigma(\mathbf{e})) z \Vert_2^2,\qquad b_i(\mathbf{e}) := 2\rho^{-2}  \langle (I + \sigma(\mathbf{e})) z, (I + \sigma(\mathbf{e}))v_i\rangle L_i,\] 
\[ a_{ij}(\mathbf{e}) := 2\rho^{-2} \langle (I + \sigma(\mathbf{e})) v_i, (I + \sigma(\mathbf{e})) v_j \rangle L_i L_j,\]
\[ a_{ii}(\mathbf{e}) := \rho^{-2} \Vert (I + \sigma(\mathbf{e})) v_i \Vert_2^2 L_i^2.   \]

Set $Q := D^{10}$. We claim that the conditions $a_{ii}(\mathbf{e}) \geq 32$, $|a_{ij}(\mathbf{e})|$, $|b_i(\mathbf{e})| \leq Q$, $|c(\mathbf{e})| \leq \frac{1}{4}$ and $b_i(\mathbf{e})^2 - 4a_{ii}(\mathbf{e})c(\mathbf{e}) < 0$ of Proposition \ref{main-quad-gap} are automatically satisfied. 

To prove these, recall \eqref{e-operator} that 
\begin{equation}\label{i-norm-recall} \frac{1}{2} \leq  \Vert I + \sigma(\mathbf{e}) \Vert \leq 2,\end{equation}
and observe the bounds
\begin{equation}\label{crude-ell2} \frac{1}{L_i} = \Vert \theta d n_i \Vert_{\T^D} = \Vert v_i \Vert_{\infty} \leq  \Vert v_i \Vert_2 \leq D^{1/2}\Vert v_i \Vert_{\infty} = \frac{D^{1/2}}{L_i}.\end{equation}

Using \eqref{i-norm-recall}, \eqref{crude-ell2} and recalling that $\rho = D^{-4}$ with $D$ large, we have the following. First, 
\[ a_{ii}(\mathbf{e}) \geq \frac{1}{4} \rho^{-2} \Vert v_i \Vert_2^2 L_i^2 > 32;\]  
Second,  by Cauchy-Schwarz
\[ |a_{ij}(\mathbf{e})| \leq 2\rho^{-2}\Vert I + \sigma(\mathbf{e}) \Vert^2 \Vert v_i \Vert_2 \Vert v_j \Vert_2 L_i L_j \leq \frac{8D}{\rho^2} < Q;\] 
Third, 
\[ |b_i(\mathbf{e})| \leq 2 \rho^{-2} \Vert I + \sigma(\mathbf{e}) \Vert^2 \Vert z \Vert_2 \Vert v_i \Vert_2  L_i  < Q;\] 
Fourth, since $\Vert z \Vert_2 \leq \rho/10$, and by \eqref{i-norm-recall}, we have
\[  c(\mathbf{e})  \leq \rho^{-2}\Vert I + \sigma(\mathbf{e}) \Vert^2 \Vert z \Vert^2_2 < \frac{1}{4}.\]
Finally, the discriminant condition $b_i(\mathbf{e})^2 - 4a_{ii}(\mathbf{e})c(\mathbf{e}) < 0$ is automatic from the positive-definiteness of the quadratic form (and is also immediate using Cauchy-Schwarz from the formulae above).

With all the relevant conditions having been verified, we are in a position to apply Proposition \ref{main-quad-gap} with $Q = D^{10}$, with $L$ as selected above (satisfying \eqref{L-bds}), $s \geq D/C_1^2$ by \eqref{s-lower-9} and the lengths $L_1,\dots, L_s \in [L, L^{1+1/48}]$ as given above, and with $B := 2^{15}C_1^2 r$. We must first check that the application is valid by confirming that $L \geq (QB)^{C_2 B}$ and that $s \geq C_1 B^2$. Using the assumption that $N > D^{D^2}$ and recalling that $D = C_3 r^2$, one may check using \eqref{L-bds} and \eqref{s-lower-9} that this is indeed the case, if $C_3$ is big enough.

Thus all of the conditions are satisfied. Now observe that with the choice of $B$ we have made (and in view of the lower bound in \eqref{L-bds}, and recalling that $X = N^{1/r}$) the fact that $B \geq 2^{10} r$ then implies that
\begin{equation}\label{lbn8} L^{-B} \leq N^{-4/D},\end{equation} the importance of the right-hand side here being that this is the width of our ellipsoidal annuli.

Proposition \ref{main-quad-gap} therefore tells us that indeed $q_{a(\mathbf{e}), b(\mathbf{e}), c(\mathbf{e})}(\frac{\ell_1}{L_1}, \dots, \frac{\ell_{s}}{L_{s}})$  takes values in $[1 -N^{-4/D}, 1]$ (that is, \eqref{ann-hit}, and hence \eqref{annulus-cond} hold) provided that $a(\mathbf{e}) := (a_{ij}(\mathbf{e}))_{1 \leq i \leq j \leq s} \in \Sigma$, where $\Sigma$ is the event appearing in Proposition \ref{main-quad-gap}. To complete the proof of Proposition \ref{second-step}, it therefore suffices to show that 
\begin{equation}\label{e-suffices} \P_{\mathbf{e}} (a(\mathbf{e}) \in \neg \Sigma) \ll N^{-3}.\end{equation}
Now, using the fact that $X = N^{1/r}$ together with \eqref{L-bds} and \eqref{s-lower-9}, we see that if $a$ is chosen randomly as in Proposition \ref{main-quad-gap} (that is, uniformly from $[-Q,Q]^{s(s+1)/2}$ with diagonal terms $\geq 32$) then, from the conclusion of Proposition \ref{main-quad-gap}, we have
\begin{equation}\label{a-except} \P_a(a \in \neg \Sigma) \ll L^{- s B/16} \leq N^{-\frac{2^{-8}}{D} \cdot \frac{1}{16} \cdot \frac{D}{C_1^2} \cdot B \cdot \frac{1}{r}} = N^{-8}. \end{equation}

Now, as $\mathbf{e}$ varies uniformly, $a(\mathbf{e})$ may not be close to a uniformly random element of $[-Q, Q]^{s(s+1)/2}$, so \eqref{e-suffices} and \eqref{a-except} are not trivially comparable. To link the two statements we invoke Proposition \ref{measure-compare}, taking the $w_i$ to be the normalised $v_i$'s, as in the conclusion of Proposition \ref{main-geo-num}. Observe that $a(\mathbf{e}) = \psi(f(\mathbf{e}))$, where $f$ is the map in Proposition \ref{measure-compare} and $\psi : \R^{s(s+1)/2} \rightarrow \R^{s(s+1)/2}$ is the diagonal linear map
\[ (\psi(\mathbf{x}))_{ij} = \rho^{-2} \Vert v_i \Vert_2 \Vert v_j \Vert_2 L_i L_j x_{ij}.  \]  
Recall that $\Vert v_i \Vert_{\infty} = L_i^{-1}$, so the determinant of $\psi$ is at least $1$ (in fact, much bigger).

Therefore we have, by Proposition \ref{measure-compare} and Lemma \ref{main-geo-num} (3),
\begin{align*}
\P_{\mathbf{e}}(a(\mathbf{e}) \in \neg\Sigma) & = \P_{\mathbf{e}} ( f(\mathbf{e}) \in \psi^{-1}(\neg \Sigma)) \\ & \leq D^{4D^2} \vol(w_1,\dots, w_s)^{-D - 1} \mu (\psi^{-1}(\neg\Sigma)) \\ & \leq D^{4D^2}\cdot D^{C_1D(D+1)} \cdot \mu(\psi^{-1}(\neg \Sigma)) \\ & \leq D^{4D^2}\cdot D^{C_1D(D+1)} \cdot (2Q)^{s(s+1)/2} \cdot \P_a(a \in \neg \Sigma),
\end{align*}
In the last step we used the fact that the determinant of $\psi^{-1}$ is at most $1$.
Recalling \eqref{a-except}, the fact that $Q = D^{10}$ and that $s \leq D$, we see that this is $< N^{-3}$ provided that $N \geq D^{C_2 D^2}$, if $C_2$ is chosen sufficiently large.
This concludes the proof of Proposition \ref{second-step}.

\end{proof}

\part{Multidimensional structure and geometry of numbers}

\section{Preliminaries on volume}\label{volume-sec}

In this section we recall some basic concepts related to volume.  Given $w_1,\dots, w_m \in \R^n$ (which in our application will always be unit vectors), define 
\[ \vol (w_1,\dots, w_m) :=  \sqrt{ \det G(w_1,\dots, w_m)},\] 
where the \emph{Gram matrix} $G = G(w_1,\dots, w_m)$ has $(i,j)$-entry $\langle w_i, w_j\rangle$. Note that $\langle Gx, x\rangle = \Vert \sum x_i w_i \Vert_2^2$, so $G$ is positive semi-definite and hence $\det G \geq 0$; therefore the square root is well-defined. $G$ is nonsingular if and only if $G$ is positive definite, if and only if the $w_i$ are linearly independent.

As the notation suggests, $\vol(w_1,\dots, w_m)$ should be interpreted as the $m$-dimensional volume of the parallelepiped spanned by $w_1,\dots, w_m$, and indeed it satisfies the intuitive properties one would expect of such a notion. The one we will need is the following (``volume = base times height''), and we include the proof since this is not completely obvious and hard to find a concise reference for.

\begin{lemma}\label{base-times-height}
Let $w_1,\dots, w_m \in \R^n$. Then 
\[ \vol(w_1,\dots, w_m) = \dist(w_m, \Span_{\R} (w_1,\dots, w_{m-1})) \vol(w_1,\dots, w_{m-1}),\] where the distance is in $\ell^2$.
\end{lemma}
\begin{proof}
If $w_1,\dots, w_{m-1}$ are linearly dependent then this is clear, so suppose they are not. Let the foot of the perpendicular from $w_m$ to $\Span_{\R}(w_1,\dots, w_{m-1})$ be $x_1 w_1 + \dots + x_{m-1} w_{m-1}$. Then the fact that $v = w_m - x_1 w_1 - \dots - x_{m-1} w_{m-1}$ is orthogonal to $w_1,\dots, w_{m-1}$ gives us $m-1$ linear relations
\begin{equation}\label{nminus1} \langle w_1, w_i\rangle x_1 + \cdots + \langle w_{m-1}, w_i\rangle x_{m-1} = \langle w_m, w_i\rangle,\end{equation} $i = 1,\dots, m-1$.
Writing $y := \Vert v \Vert_2^2$ for the length of the perpendicular, Pythagoras's theorem gives
\begin{equation}\label{pythagoras}  \langle w_1, w_m\rangle x_1 + \dots + \langle w_{m-1}, w_m \rangle x_{m-1} + y = \langle w_m, w_m\rangle.\end{equation}
Combining \eqref{nminus1} and \eqref{pythagoras} into an $m \times m$ system of equations in the variables $x_1,\dots, x_{m-1}, y$, it follows from Cramer's rule that $y = \frac{G(w_1,\dots, w_m)}{G(w_1,\dots, w_{m-1})}$, which is equivalent to the stated result.
\end{proof}

Two immediate consequences of this are the following.

\begin{corollary}\label{volume-hereditary}
Suppose that $w_1,\dots, w_m$ are unit vectors in $\R^n$. Then
\[ \vol(w_1,\dots, w_m) \leq \vol(w_1,\dots, w_{m-1}) \leq \dots \leq 1.\]
\end{corollary}

\begin{corollary}\label{lemc4}
Suppose that $w_1,\dots, w_m \in \R^n$ are unit vectors, where $m < n$. Then we may complete this list of vectors list to $w_1,\dots, w_n$ with $\vol(w_1,\dots, w_n) = \vol(w_1,\dots, w_m)$. 
\end{corollary}
\begin{proof} Choose $w_{m+1},\dots, w_n$ to be orthogonal to one another and to $w_1,\dots, w_m$. \end{proof}

The other consequence of Lemma \ref{base-times-height} that we will require is the following dichotomy, for which I do not know a reference.

\begin{lemma}\label{volume-dichotomy}
Let $w_1,\dots w_m \in \R^n$ be unit vectors, and let $k \leq m$. Then, after reordering the $w_i$, at least one of the following statements is true:
\begin{enumerate}
\item $\vol(w_1,\dots, w_{k+1}) \geq \delta$;
\item If $V = \Span(w_1,\dots, w_k)$ then $\dist (w_i, V) \leq \delta^{1/k}$ for all $i \in \{1,\dots, m\}$.
\end{enumerate}
\end{lemma}
\begin{proof}
We perform the following algorithm for as long as possible. Start, at stage 1, with $w_1$. At the $i$th stage, we will have (after reordering) $w_1,\dots, w_i$. If \[ \dist(w_j, \Span(w_1,\dots, w_i)) \leq \delta^{1/k}\]  for all $j \in \{1,\dots,m\}$ then stop. If this happens for some $i \leq k$ then we have (2). Otherwise, we may reorder $w_{i+1},\dots, w_m$ so that $\dist(w_{i+1},\Span(w_1,\dots, w_i)) > \delta^{1/k}$. By Lemma \ref{base-times-height} (and a simple induction) we have $\vol(w_1,\dots, w_{i+1}) > \delta^{i/k}$. If this continues as far as $i = k$ then we have (1).
\end{proof}

\section{Non-concentration on subspaces}\label{slab-avoid-sec}

In this section we establish a key technical result, Lemma \ref{slab-avoid} below, which says that orbits $\{ \theta d n : n \leq X\}$ with $\theta$ diophantine cannot concentrate too much near low-dimensional subspaces (or, more accurately, the image under $\pi : \R^D \rightarrow \T^D$ of small balls in such subspaces). Here, we use $\mathcal{N}_{\delta}(U)$ to denote the $\delta$-neighbourhood (in the Euclidean metric) of a set $U \subset \R^D$, and as usual $B_{\delta}(0)$ denotes the Euclidean ball about $0 \in \R^D$. Write $\theta d [X] = \{ \theta d n : n \leq X\}$ for short.

\begin{lemma}\label{slab-avoid}
Let $r$ be sufficiently large, suppose that $D = C_3 r^2$, and assume that $N \geq D^{C_2 D^2}$. Let $\eps \in (\frac{1}{r},1)$. Let $\theta \in \T^D$ be diophantine, let $d \leq N/X$, and let $V \leq \R^D$ be a subspace of dimension at most $D(1 - \eps)$. Then for all $d \leq N/X$ we have\begin{equation}\label{key-bd} 
\# \{ \theta d [X] \cap \pi(\mathcal{N}_{D^{-C_1}}(V) \cap B_{1/10}(0))\} \leq 2D^{( 1-\eps C_1/2)D} X.
\end{equation}
\end{lemma}
\begin{proof}
We may assume, by adding elements if necessary, that $\dim V = \lfloor (1 - \eps) D\rfloor$. Set $m := \lceil\eps D\rceil$, $\delta := D^{-C_1}$ and $S := \mathcal{N}_{D^{-C_1}}(V) \cap B_{1/10}(0)$. Let $(v_i)_{i = 1}^D$ be an orthonormal basis for $\R^D$ with $V = \Span_{\R}(v_{m+1},\dots, v_D)$. Let $\psi : \R \rightarrow \R$ be a smooth cutoff satisfying the following conditions:
\begin{enumerate}
\item $\psi \geq 0$ everywhere, and $\psi(x) \geq 1$ for $|x| \leq 1$;
\item $\hat{\psi}(y) = 0$ for $|y| \geq 1$;
\item $\int \psi \leq 5$.
\end{enumerate}
For the proof that such a function exists, see Lemma \ref{band-limited-r}. Now define $\chi : \R^D \rightarrow [0,\infty)$ by
\begin{equation}\label{bump-slab}  \chi(x) = \prod_{i = 1}^{D} \psi(\delta^{-1_{i \leq m}}\langle x, v_i\rangle),\end{equation} where the notation means that the scale factor $\delta^{-1}$ is included only for $i \leq m$ (that is, in the directions orthogonal to $V$), and is otherwise set equal to 1. If $x \in \mathcal{N}_{\delta}(V)$ and $i \leq m$ then \[ |\langle x, v_i \rangle| \leq \big(\sum_{i = 1}^m |\langle x, v_i \rangle|^2\big)^{1/2} = \dist(x, V) \leq \delta,\] and if $x \in B_{1/10}(0)$ then $|\langle x, v_i \rangle| < 1$ for all $i$. It follows from these observations and property (1) of $\psi$ that $\chi(x) \geq 1$ for $x \in S$, and therefore
\begin{equation}\label{majorant} \# \big( \theta d[X] \cap \pi(S) \big)  \leq \sum_{n \leq X} \sum_{\lambda \in \Z^D} \chi(\theta d n + \lambda) .\end{equation}
The Poisson summation formula tells us that for any $t \in \T^D$ we have
\[ \sum_{\lambda \in \Z^D} \chi(\lambda + t) = \sum_{\xi \in \Z^D} \hat{\chi}(\xi) e( \xi \cdot t).\] Substituting into \eqref{majorant} therefore implies that
\begin{equation}\label{poisson-sum}
\# \big( \theta d [X] \cap \pi(S) \big)  \leq \sum_{\xi \in \Z^D} \hat{\chi}(\xi) \sum_{n \leq X} e( \xi \cdot \theta d  n).
\end{equation}
To proceed further, we need to understand the Fourier transform $\hat{\chi}$, particularly at points of $\Z^D$. First,  it follows from property (3) of $\psi$ above that 
\begin{equation}\label{chi-bd} \Vert \hat{\chi} \Vert_{\infty} \leq \int \chi = \delta^m (\int \psi)^D \leq 5^D \delta^{m}.  \end{equation} 
Next, expanding in the orthonormal basis $(v_j)_{j = 1}^D$ then changing variables we have \begin{align*} \hat{\chi}(\gamma)  & = \int_{\R^D} \prod_{j = 1}^D \psi (\delta^{-1_{j \leq m}} \langle x , v_j\rangle) e(-\langle x, v_j\rangle \langle \gamma, v_j \rangle) dx 
\\ &   = \int_{\R^D} \prod_{j = 1}^D \psi(\delta^{-1_{j \leq m}} t_j) e(-t_j \langle \gamma ,v_j \rangle) dt \\ & = \delta^{m} \prod_{j = 1}^D \hat{\psi} (\delta^{1_{j \leq m}} \langle \gamma , v_j\rangle).  \end{align*}
Therefore from property (2) of $\psi$ we see that 
\begin{equation}\label{support} \Supp(\hat{\chi}) \subset B := \{ \gamma \in \R^D : | \langle \gamma , v_j \rangle | \leq  \delta^{-1_{j \leq m}} \; \mbox{for all $j$} \}.\end{equation}
First note that this implies (rather crudely) that if $\hat{\chi}(\xi) \neq 0$ for some $\xi \in \Z^D$ then
\begin{equation}\label{xi-bd} |\xi|^2 < \Vert  \xi \Vert_2^2 = \sum_{i} |\langle \xi, v_i \rangle|^2 \leq \frac{D}{\delta^2} < (\frac{D}{\delta})^2.\end{equation}
Second, to analyse \eqref{poisson-sum} we need a bound on $\#(\Supp (\hat{\chi}) \cap  \Z^D)$. To get such a bound, note that if $\gamma \in \Supp(\chi)$ and if $u \in [0,1]^D$ then by \eqref{support} (and since $D$ is big)
\[ |\langle \gamma + u, v_j \rangle| \leq  \delta^{-1_{j \leq m}} + |\langle  u , v_j \rangle | \leq \frac{D}{10} \delta^{-1_{j \leq m}}\] for all $j$. Therefore the disjoint union of cubes $(\Supp(\hat{\chi}) \cap \Z^D) + [0,1)^D$ is contained in the cuboid
\[ \{ x \in \R^D : | \langle x , v_j \rangle | \leq \frac{D}{10} \delta^{-1_{j \leq m}} \; \mbox{for all $j$}\},\] which has volume $(D/5)^D \delta^{-m}$.
It follows that $\# (\Supp(\chi) \cap \Z^D) \leq (D/5)^D \delta^{-m}$, and so by \eqref{chi-bd}
\begin{equation}\label{xi-integer-support} \sum_{\xi \in \Z^D} |\hat{\chi}( \xi)| \leq D^D. \end{equation}

Let us return to the main task of estimating the right-hand side of \eqref{poisson-sum}. Summing the geometric series on the right of \eqref{poisson-sum}, we have
\begin{equation}\label{poisson-follow} \# \big( \theta d [X] \cap \pi(S) \big)  \leq \sum_{\xi \in \Z^D} |\hat{\chi}(\xi)| \min(X, \Vert \xi \cdot \theta d \Vert_{\T}^{-1}).\end{equation}
Now by \eqref{xi-integer-support}, the contribution from $\xi$ with $\Vert \xi \cdot \theta d \Vert_{\T} > D^D \delta^{-m} X^{-1}$ is at most $\delta^{m}X \leq D^{-\eps C_1D}X$. It therefore follows from \eqref{chi-bd}, \eqref{xi-bd} and \eqref{poisson-follow} that 
\begin{equation}\label{poisson-follow-more} \# \big( \theta d [X] \cap \pi(S) \big) \leq 5^D \delta^{m} X \# \Omega+ D^{-\eps C_1D}X.\end{equation}
where
\begin{equation}\label{omega-def} \Omega := \{ \xi \in \Z^D: |\xi| < D/\delta,  \Vert \xi \cdot \theta d \Vert_{\T}  \leq D^D \delta^{-m} X^{-1}\} .\end{equation}
Now since $\theta$ is diophantine, it follows from Proposition \ref{prop31} (1) (the definition of diophantine) that $\dim \Omega < 4r$, assuming that $C_2 \geq C_1 + 1$.

It follows from Lemma \ref{lattice-subspace} that
\begin{equation}\label{bd-num-pts} \# \Omega \leq 20^D (4r)^{D/2}(D/\delta)^{4r}.\end{equation}

To conclude the proof, we combine \eqref{poisson-follow-more} and \eqref{bd-num-pts} and bound the resulting terms crudely. Since $D = C_3 r^2$ and $r$ is large, crude bounds for the terms in \eqref{bd-num-pts} show that
\[ \# \big( \theta d [X] \cap \pi(S)\big) \leq D^D \delta^{m-4r} X + D^{-\eps C_1D}X.\] Using the assumption that $m \geq 8r$, the first term is bounded above by $D^D\delta^{m/2}X \leq D^{1-C_1\eps D/2}X$. The proposition follows.\end{proof}

\section{Geometry of numbers}\label{geom-num}

Set $X := N^{1/r}$ as usual, with $r$ sufficiently large. Let $d \leq N/X$; for the rest of the section we regard $d$ as fixed and do not explicitly indicate the dependence of various objects (lengths $L_i, L'_i$, vectors $v_i, w_i$ and so on) on $d$.
Our aim in this section is to prove Proposition \ref{main-geo-num}, whose statement was as follows.

\begin{main-geo-num-rpt}
Suppose that $N \geq D^{C_2 D^2}$, let $d \leq N/X$, and suppose that $\theta \in \Theta$ is diophantine. Then there are positive integers $n_1,\dots, n_s$, $s \leq D$, such that if we write $L_i := \Vert \theta d n_i \Vert_{\T^D}^{-1}$ then\begin{enumerate}
\item $\prod_{i = 1}^s L_i \geq X^{1/80}$;
\item $\sum_{i = 1}^s n_i L_i \leq X/2$;
\item if we set $v_i := \pi^{-1}(\theta d n_i) \in \R^D$ and $w_i :=v_i/\Vert v_i \Vert_2$, the unit vector in the direction of $v_i$, then $\vol (w_1,\dots, w_s) \geq D^{-C_1 D}$;
\item $L_i \leq X^{C_1/D}$ for all $i$.
\end{enumerate}
\end{main-geo-num-rpt}

Just to reiterate: $s$, the $v_i$, the $n_i$ and the $L_i$ will all depend on $d$, as well as on $\theta$ which should be thought of as fixed.

The proof of Proposition \ref{main-geo-num} is somewhat lengthy. We begin by establishing a preliminary statement, Lemma \ref{lem71}, featuring related (but weaker) statements, but which does not require any diophantine assumption on $\theta$. This statement is essentially the principle, well-known in additive combinatorics, that ``Bohr sets contain large generalised progressions''. Usually in the literature this is given for Bohr sets in $\Z/p\Z$, whereas we require it in $\Z$, which requires a minor tweak to the proof and causes the dimension of the resulting progressions to be greater by one than in the cyclic group case.

\begin{lemma}\label{lem71}
There are positive integers $n_1,\dots, n_{D+1}$ and also lengths $L'_1,\dots, L'_{D+1}$ such that the following hold:\begin{enumerate}
\item The elements $\sum_{i=1}^{D+1} \ell_i n_i$, $\ell_i \in \Z$, $0 \leq \ell_i < L'_i$, are all distinct and at most $D^{-D} X$;
\item $\Vert \theta d n_i \Vert _{\T^D} \leq 1/L'_i$ for all $i$;
\item $\prod_{i=1}^{D+1} L'_i \geq D^{-3D} X$.
\end{enumerate}
\end{lemma}
\emph{Remark.} Write $\alpha := \theta d$ throughout the proof. The $n_i$ appearing in Proposition \ref{main-geo-num} will be a subset of the ones appearing here, but the lengths $L_i$ appearing there will be modified versions of the $L'_i$. That is why we have put dashes on these lengths.
\begin{proof} In $\R \times \R^D$, consider the lattice
\[ \Lambda = \Z (\frac{1}{X}, \alpha_1,\dots, \alpha_D) \oplus (\{0\} \times \Z^D)\] and the centrally-symmetric convex body
\[ K := [-D^{-D}, D^{-D}] \times [-\frac{1}{2}, \frac{1}{2}]^D.\] We have $\vol (K) = 2D^{-D}$ and $\det (\Lambda) = \frac{1}{X}$, so Minkowski's second theorem tells us that the successive minima $\lambda_1,\dots, \lambda_{D+1}$ for $K$ with respect to $\Lambda$ satisfy
\begin{equation}\label{minkowski-ii}  \lambda_1 \cdots \lambda_{D+1} \leq \frac{2^{D+1} \det(\Lambda)}{\vol(K)} = (2D)^D \frac{1}{X}. \end{equation} Consider a directional basis $\b_1,\dots, \b_{D+1}$ for $\Lambda$ with respect to $K$. Thus $\b_i \in \Lambda$, the $\b_i$ are linearly independent and $\b_i \in \lambda_i K$ (but the $\b_i$ need not be an integral basis for $\Lambda$). Write
\[ \b_i = (\frac{n_i}{X}, n_i \alpha - m_{i}),\] where $n_i \in \Z$ and $m_i \in \Z^D$. Replacing $\b_i$ by $-\b_i$ if necessary, we may assume that $n_i \geq 0$. We take these $n_i$s to be the ones in Lemma \ref{lem71}. Set
\[ L'_i := \frac{1}{(D+1) \lambda_i}.\]
We now verify statements (1), (2) and (3) in the lemma. Item (3) is a straightforward consequence of the definition of the $L'_i$ and \eqref{minkowski-ii}:
\[ \prod_{i = 1}^{D+1} L'_i  = (D+1)^{-D - 1} \big(\prod_{i=1}^{D+1} \lambda_i \big)^{-1} \geq (D+1)^{-D-1} (2D)^{-D} X > D^{-3D} X.\]
Item (2) follows from the fact that $\b_i \in \lambda_i K$; looking at the last $D$ coordinates, one sees that this means that $\Vert n_i \alpha \Vert_{\T^D} \leq \lambda_i/2 < 1/L'_i$.
Finally we turn to (1). We have
\begin{equation}\label{ell-bee} \ell_1 \b_1 + \dots + \ell_{D+1} \b_{D+1} = \big(\frac{1}{X} \sum_{i=1}^{D+1} \ell_i n_i , \sum_{i=1}^{D+1} \ell_i (n_i \alpha - m_i)\big).\end{equation}
Since $\b_i \in \lambda_i K$, comparing first coordinates we see that if $0 \leq \ell_i < L'_i$ then
\[ 0 \leq \frac{1}{X} \sum_{i=1}^{D+1} \ell_i n_i \leq D^{-D} \sum_{i=1}^{D+1} \ell_i \lambda_i \leq D^{-D}.\]
This is one part of statement (1). For the statement about distinctness, suppose that $\sum_{i=1}^{D+1} \ell_i n_i = \sum_{i=1}^{D+1} \ell'_i n_i$ with $0 \leq \ell_i, \ell'_i < L'_i$. Then $\sum_{i=1}^{D+1} (\ell_i - \ell'_i) \b_i \in \{0\} \times \Z^D$ (by \eqref{ell-bee} and its analogue for the $\ell'_i$). However,
\[ \Vert \sum_{i=1}^{D+1} (\ell_i - \ell'_i) \b_i \Vert_{\infty} \leq \sum_{i=1}^{D+1} L'_i \lambda_i \cdot \frac{1}{2}  < 1.\] It follows that $\sum_{i=1}^{D+1} (\ell_i - \ell'_i) \b_i = 0$ and hence, since the $\b_i$ are linearly independent, that $\ell_i = \ell'_i$ for all $i$.
\end{proof}

From now on we work with the $n_i$ generated in Lemma \ref{lem71}, and set $L_i := \Vert \theta d n_i \Vert^{-1}_{\T^D}$ (which agrees with the statement of Proposition \ref{main-geo-num}). Let us remind the reader that we are thinking of $d$ as fixed; the $L_i$ of course depend on $d$. By Lemma \ref{lem71} (2) we have
\begin{equation} \label{lili} L'_i \leq L_i .\end{equation}
Before turning to the proof of Proposition \ref{main-geo-num} itself, we use Proposition \ref{prop31} (2) (that is, the second condition in the definition of $\theta$ being diophantine) to get some rough control on the lengths $L_i$. The following lemma, though sufficient for our needs, is rather weak, asserting that it is not possible for almost all of the product $\prod L_i$ to be concentrated on a few values of $i$.

\begin{lemma}\label{weak-i-control}
Suppose that $D = C_3r^2$ and that $N \geq D^{D^2}$. Suppose that $I \subset [D+1]$ and that $|I| \leq 2^{-7} D$. Then $\prod_{i \notin I} L_i \geq X^{1/80}$.
\end{lemma}
\begin{proof}
Suppose that this is not the case for some index set $I$.  Then certainly (by \eqref{lili}) $\prod_{i \notin I} L'_i < X^{1/80}$, and so by Lemma \ref{lem71} (3) we have
\[ \prod_{i \in I} L'_i \geq D^{-3D} X^{1 - 1/80} > X^{1 - 1/40}.\]
Now look at all sums $\sum_{i \in I} \ell_i n_i$ with $0 \leq \ell_i \leq \frac{1}{2D} X^{-1/D} L'_i$.  By Lemma \ref{lem71} (1) and (2), these sums are all distinct, and for each one
\[ \Vert \theta d \sum_{i \in I} \ell_i n_i \Vert_{\T^D} \leq \sum_{i \in I} \ell_i \Vert \theta d n_i \Vert_{\T^D} \leq X^{-1/D}.\]  
However, Proposition \ref{prop31} (2) tells us that 
\[ \# \{ n \leq X : \Vert \theta d n \Vert_{\T^D} \leq X^{-1/D} \} \leq X^{9/10}.\]
It follows that 
\[   \big(  \frac{X^{-1/D}}{2D}   \big)^{2^{-7} D} X^{1 - 1/40} \leq \prod_{i \in I}(\frac{X^{-1/D}}{2D}  L'_i) \leq  X^{9/10},\] which is a contradiction.
\end{proof}

Now we turn to the proof of Proposition \ref{main-geo-num} itself.

\begin{proof}[Proof of Proposition \ref{main-geo-num}] The idea is to take the $n_i$ output by Lemma \ref{lem71}, but discard indices $i$ which cause items (2), (3) and (4) of Proposition \ref{main-geo-num} to be violated, whilst ensuring that the lower bound (1) is maintained.  When we say that (2), (3) or (4) holds for indices in a set $I$, this has the obvious meaning (namely, for (2) we mean that $\sum_{i \in I} n_i L_i \leq X/2$, and for (3) that $\vol ((w_i)_{i \in I}) \geq D^{-C_1 D}$). Note that properties (2), (3) and (4) are all hereditary, that is to say if they hold for indices in $I$ then they also hold for indices in $I'$, for any subset $I' \subset I$. For (2) and (4) this is obvious; for (3), it follows from Corollary \ref{volume-hereditary}.

This, together with Lemma \ref{weak-i-control}, allows us to treat (2), (3) and (4) of Proposition \ref{main-geo-num} essentially separately. We will show, for each $n \in \{2,3,4\}$, that there is a set $I_{(n)} \subset [D+1]$ of indices,
\begin{equation}\label{exceptional-i} |I^c_{(2)} |,|I^c_{(3)}|  \leq 2^{-9} D, \; \; |I^c_{(4)}| \leq 2^{-8} D,\end{equation} such that item $(n)$ of Proposition \ref{main-geo-num} holds on $I_{(n)}$.  If we then set $I := I_{(2)} \cap I_{(3)} \cap I_{(4)}$ then properties (2) , (3) and (4) all hold on $I$, and by Lemma \ref{weak-i-control} we have the lower bound $\prod_{i \in I} L_i \geq X^{1/80}$. Relabelling so that $I = \{1,\dots, s\}$, Proposition \ref{main-geo-num} then follows.

The arguments for properties (2) and (3) share some common features, which we introduce now. Consider the set
\begin{equation}\label{to-slab} B =  \{ \sum_{i=1}^{D+1} \ell_i n_i : 0 \leq \ell_i \leq \frac{L'_i}{20D^2}\}.\end{equation}
All the elements in this presentation of $B$ are distinct by Lemma \ref{lem71} (1), so (rather crudely)
\begin{equation}\label{b-lower} |B| \geq D^{-6D} X\end{equation} by Lemma \ref{lem71} (3) and the fact that $D$ is large. 
Recall that $v_i = \pi^{-1}(\theta d n_i)$ (that is, the unique smallest lift under the projection map $\pi : \R^D \rightarrow \T^D$, so in particular $\Vert v_i \Vert_{\infty} = \Vert \theta d n_i \Vert_{\T^D}$) and that $w_i = v_i/\Vert v_i \Vert_2$ is the associated unit vector. Therefore, since $\pi$ is a homomorphism
\[ \theta d ( \sum_{i=1}^{D+1} \ell_i n_i) = \pi (x),\]
where \begin{equation}\label{xvw} x = \sum_{i=1}^{D+1} \ell_i v_i = \sum_{i=1}^{D+1} \ell_i \Vert v_i \Vert_2 w_i.\end{equation}
By \eqref{lili} (and the fact that $\ell_i < L'_i/20D^2 \leq L_i/20D^2$) we have 
\begin{equation}\label{livi} |\ell_i| \Vert v_i \Vert_2 \leq \frac{L_i}{20D^2} \cdot D^{1/2} \Vert v_i \Vert_{\infty} = \frac{L_i}{20D^2} \cdot D^{1/2} \Vert \alpha n_i \Vert_{\T^D} = \frac{1}{20 D^{3/2}},\end{equation} and so in particular
\begin{equation}\label{x-tenth} \Vert x \Vert_2 \leq (D+1) \frac{1}{20 D^{3/2}} < \frac{1}{10}.\end{equation}
Now we look at properties (2), (3) and (4) of Proposition \ref{main-geo-num} separately.

\emph{Property (2).} Set $I_{(2)} := \{ i : L_i \leq D^{C_1} L'_i.\}$ Then property (2) holds on $I_{(2)}$, since
\[ \sum_{i \in I_{(2)}} n_i L_i \leq D^{C_1} \sum_{i \in I_{(2)}} n_i L'_i \leq D^{C_1-D} X < \frac{X}{2},\] by Lemma \ref{lem71} (1).
It remains to prove \eqref{exceptional-i}, that is to say that $|I_{(2)}^c| \leq 2^{-9} D$. Suppose not, and take $V := \Span_{\R}((w_i)_{i \in I_{(2)}})$. Suppose that \eqref{exceptional-i} fails; then $\dim V \leq (1 - 2^{-9})D$. Consider an element $b = \sum_i \ell_i n_i \in B$, with $B$ defined in \eqref{to-slab}. As explained in \eqref{x-tenth} above, we have $\theta d b = \pi(x)$, with $x \in B_{1/10}(0)$. From \eqref{xvw}, since the $w_i$ are unit vectors we see that 
\begin{equation}\label{dist-V} \dist(x, V) \leq \sum_{i \notin I_{(2)}} |\ell_i |\Vert v_i \Vert_2 .\end{equation}
Now observe that if $i \notin I_{(2)}$ then
\begin{align*} |\ell_i| \Vert v_i \Vert_2 \leq \frac{L'_i \Vert v_i \Vert_2}{20D^2}  \leq \frac{L'_i \Vert v_i \Vert_{\infty} }{20 D^{3/2}}  & = \frac{L'_i \Vert \theta d n_i \Vert_{\T^D}}{20 D^{3/2}}  \\ & = \frac{L'_i}{20 D^{3/2} L_i} < \frac{1}{2}D^{-1-C_1},\end{align*}and so from \eqref{dist-V} we have $\dist(x, V) < D^{-C_1}$. We have shown that $\theta d B  \subset \pi( N_{D^{-C_1}}(V) \cap B_{1/10}(0))$, and so by \eqref{b-lower} and the fact that $B \subset [X]$ we have
\begin{equation}\label{slab-lower} \# \{ \theta d [X]  \cap \pi( N_{D^{-C_1}}(V) \cap B_{1/10}(0))\} \geq D^{-6D} X.\end{equation}
On the other hand,  Lemma \ref{slab-avoid} (with $\eps = 2^{-9}$, since $\dim V \leq (1 - 2^{-9}) D$) tells us that 
\begin{equation}\label{slab-avoid-app}  \# \{ \theta d [X] \cap \pi(\mathcal{N}_{D^{-C_1}}(V) \cap B_{1/10}(0))\} \leq 2D^{( 1-2^{-10} C_1)D} X.
  \end{equation}
 If $C_1$ is big enough, statements \eqref{slab-lower} and \eqref{slab-avoid-app} contradict one another, so we were wrong to assume that $|I_{(2)}^c| > 2^{-9} D$.

\emph{Property (3).} Suppose that there does not exist a set $I_{(3)} \subset [D+1]$, $|I_{(3)}| \geq D +1 - 2^{-9} D$, with $\vol ((w_i)_{i \in I_{(3)}}) \geq D^{-C_1 D}$. Then applying Lemma \ref{volume-dichotomy}, we see that there is a subspace $V \leq \R^D$, $\dim V < D(1 - 2^{-9})$, such that 
\begin{equation}\label{distw-v} \dist (w_i, V) \leq D^{-C_1} \end{equation}for all $i$. We now proceed much as before. Consider an element $b = \sum_i \ell_i n_i \in B$, with $B$ defined in \eqref{to-slab}. As explained above, we have $\theta d b = \pi(x)$, with $x \in B_{1/10}(0)$. Now, using \eqref{xvw}, \eqref{livi} and \eqref{distw-v}, we see that

\begin{equation}\label{dist-first} \dist (x, V ) \leq \sum_i |\ell_i| \Vert v_i \Vert_2 \dist(w_i, V)  \leq D^{-C_1}.\end{equation}
We have shown that $\theta d B  \subset \pi( N_{D^{-C_1}}(V) \cap B_{1/10}(0))$, and so once again we conclude from \eqref{b-lower} and the fact that $B \subset [X]$ that
\[  \# \{ \theta d [X]  \cap \pi( N_{D^{-C_1}}(V) \cap B_{1/10}(0))\} \geq D^{-6D} X.\]
Once again, this contradicts Lemma \ref{slab-avoid} (if $C_1$ is large enough). Thus we were wrong to assert that $I_{(3)}$ does not exist.

\emph{Property (4).} Set $J := \{ i \in I_{(2)} : L_i \geq X^{C_1/D} \}$. Since $J  \subset I_{(2)}$, we have
\[ \prod_{i \in J} L_i \leq D^{C_1 (D+1)} \prod_i L'_i  < D^{C_1 (D+1)} X < X^2,\] by Lemma \ref{lem71} (1) and the assumption on $N$. If $C_1 \geq 2^{11}$ then it follows that $|J| \leq 2D/C_1 < 2^{-10} D$, and so if we set $I_{(4)} := I_{(2)} \setminus J$ then the required bound \eqref{exceptional-i} follows.

Finally -- a very minor point -- we note that property (3) implies that $w_1,\dots, w_s$ are linearly independent and so $s \leq D$. This completes the proof of Proposition \ref{main-geo-num}.
\end{proof}

\emph{Remark.} In Section \ref{geom-num} (which depended heavily on Sections \ref{diophantine-sec} and \ref{slab-avoid-sec}) we obtained what amounts to some weak information about the ``shape'' of a random Bohr set such as \[ \{ n \leq X : \Vert n \theta_i \Vert_{\T} \leq \frac{1}{4} \; \mbox{for $i = 1,\dots, D$}\}.\] We were not interested in just the almost sure behaviour, but rather what can be said with very small exceptional probability on the order of $X^{-r}$, say. Many aspects of this situation remain very mysterious to me and it may be of interest to study the problem further.

\section{Comparison of two distributions on quadratic forms}\label{measure-compare-sec}

In this section, which essentially stands by itself, we prove Proposition \ref{measure-compare}. Recall that we define a map $\sigma : \R^{n(n+1)/2} \rightarrow \Sym_n(\R)$ as in Section \ref{sec4}: to any tuple $\mathbf{x} \in \R^{n(n+1)/2}$, we associate a symmetric matrix $\sigma(\mathbf{x}) \in \Sym_n(\R)$ as follows: $(\sigma(\mathbf{x}))_{ii} = x_{ii}$, $(\sigma(\mathbf{x}))_{ij} = \frac{1}{2}x_{ij}$ for $i < j$, and $(\sigma(\mathbf{x}))_{ij} = \frac{1}{2}x_{ji}$ for $i > j$. We will use this for $n = D$ and for $n = s$. 

The inverse $\sigma^{-1} : \Sym_n(\R) \rightarrow \R^{n(n+1)/2}$ is given by $(\sigma^{-1}(M))_{ii} = M_{ii}$ and $(\sigma^{-1}(M))_{ij} = 2M_{ij}$ for $i < j$. Note, in particular, that if $M$ is symmetric and if $x \in \R^n$ then
\begin{equation}\label{quad-form-rel}  x^T M x = \sum_{i \leq j} \sigma^{-1}(M)_{ij} x_i x_j.\end{equation}

Recall that $\mathbf{e}$ is sampled uniformly from $[-\frac{1}{D^4} , \frac{1}{D^4} ]^{D(D+1)/2}$.

\begin{measure-compare-rpt}
Let $w_1,\dots, w_{s} \in \R^D$, $s \leq D$, be linearly independent unit vectors. Write $f : \R^{D(D+1)/2} \rightarrow \R^{s(s+1)/2}$ for the map defined by \[ f(x) = \sigma^{-1} \big( (\langle (I + \sigma(x)) w_i, (I + \sigma(x))w_j\rangle)_{1 \leq i, j \leq s} \big)\] Then for any measurable set $U \subset \R^{s(s+1)/2}$ we have \[ \P_{\mathbf{e}} (f(\mathbf{e}) \in U) \leq D^{4D^2} \vol(w_1,\dots, w_{s})^{-D-1} \mu(U),\] where $\mu$ denotes Lebesgue measure. 
\end{measure-compare-rpt}
\begin{proof} We first handle the case $s = D$ (which we will prove with the slightly stronger constant $D^{3D^2}$), and then deduce the case $s < D$ from it. Suppose, for the moment, that $s = D$. 

We write $f(x) = f_2(f_1(x))$ as a composition of two maps
\[ f_1 : \R^{D(D+1)/2} \rightarrow \R^{D(D+1)/2}: \quad f_1(x) = \sigma^{-1} \big((I + \sigma(x))^T (I + \sigma(x))    \big)\] (note here that the transpose is superfluous, since $\sigma(\mathbf{x})$ is symmetric) and 
\[ f_2 : \R^{D(D+1)/2} \rightarrow \R^{D(D+1)/2}: \quad f_2(x) = \sigma^{-1} \big( W^T \sigma(x) W),\] where $W_{ij} = (w_j)_i$ (that is, the $i$th coordinate of $w_j$ in the standard basis). Checking that $f$ is indeed the composition of these two maps amounts to checking that $(W^T A^T A W)_{ij} = \langle A w_i, A w_j \rangle$ for any matrix $A$, which is an easy exercise. 

We claim that $f_1$ is injective on the domain $[-\frac{1}{D^4}, \frac{1}{D^4}]^{D(D+1)/2}$. If we have $f_1(x_1) = f_1(x_2)$, then $(I + \sigma(x_1))^2 = (I + \sigma(x_2))^2$. Suppose that $I + \sigma(x_i) = U_i^{-1} \Delta_i U_i$ for $i = 1,2$, with the $U_i$ orthogonal and $\Delta_i$ diagonal with entries non-increasing down the diagonal.  If $v$ is a unit vector and $x$ lies in the domain then $\Vert \sigma(x) v \Vert_2 \leq D \Vert x \Vert_{\infty} \leq D^{-3}$, and so all eigenvalues of $I + \sigma(x_i)$ (that is, entries of $\Delta_i$) are extremely close to $1$ and in particular positive. Therefore $\Delta_1^2, \Delta_2^2$ are also diagonal with entries non-increasing down the diagonal. Moreover $U_1^{-1} \Delta^2_1 U_1 = U_2^{-1} \Delta^2_2 U_2$, so $\Delta^2_1$ and $\Delta^2_2$ are similar matrices and therefore the same. It follows that $\Delta_1 = \Delta_2$, and also $U_1 U_2^{-1}$ commutes with $\Delta_1^2$ ($= \Delta_2^2$). However, a diagonal matrix $\Delta$ with positive entries and its square $\Delta^2$ have the same centraliser (block matrices based on equal diagonal entries) and hence $U_1 U_2^{-1}$ commutes with $\Delta_1$ ($= \Delta_2$), which means that $I + \sigma(x_1) = U_1^{-1} \Delta_1 U_1 = U_2^{-1} \Delta_1 U_2 = U_2^{-1} \Delta_2 U_2 = I + \sigma(x_2)$ and so $x_1 = x_2$. The claim follows.

For the two maps $f_1, f_2$, we additionally claim that:
\begin{enumerate}
\item $f_1$ is differentiable on $[-\frac{1}{D^4} , \frac{1}{D^4} ]^{D(D+1)/2}$, with Jacobian bounded below by $1$;
\item $f_2$ is linear, with $|\det f_2| = |\det W|^{D+1}$.
\end{enumerate}

The proposition then follows by change of variables, in fact with the more precise constant $(2D^4)^{D(D+1)/2}$ for the $D$-dependence, noting that $\vol(w_1,\dots, w_D) = |\det (W^T W)|^{1/2} = |\det W|$.

\emph{Proof of (1).} This can be established by direct, co-ordinatewise, calculation. Indeed, it is easy to check that 
\[ (f_1(\mathbf{x}))_{ij} = 1_{i = j} + 2 x_{ij}  + q_{ij}(\mathbf{x}),\] where $q_{ij}(\mathbf{x})$ is a quadratic form with at most $D$ terms, each with coefficient bounded by $1$. Thus 
\begin{equation}\label{jac-near} \big| \frac{\partial (f_1)_{ij}}{\partial x_{uv}}(\mathbf{e}) - 2 \cdot 1_{(i,j) = (u,v)} \big| \leq 2D \Vert \mathbf{e} \Vert_{\infty}.\end{equation}
To bound the Jacobian of $f_1$ below, we need a lower bound for determinants of perturbations of the identity. Using Fredholm's identity $\det (I + E) = \exp(\sum_{k = 1}^{\infty} \frac{(-1)^{k-1}}{k} \mbox{tr}(E^k))$, one can check that if $E$ is an $n$-by-$n$ matrix with entries bounded in absolute value by $\eps < \frac{1}{2n}$ then $\det(I+E) \geq e^{-2 n \eps}$. (In fact, the stronger bound $1 - n \eps$ for $\eps \leq \frac{1}{n}$ was shown by Ostrowksi \cite[Eq (5.5)]{ostrowski} in 1938; for our purposes, fairly crude bounds would suffice.)

It follows from this, \eqref{jac-near} and the fact that $\vert \mathbf{e} \Vert_{\infty} \leq D^{-4}$ that, if $D$ is large, the Jacobian of $f_1$ is indeed greater than $1$.

\emph{Proof of (2).} If $A$ is a $D \times D$ matrix, define $f_2^A : \R^{D(D+1)/2} \rightarrow \R^{D(D+1)/2}$ by $f_2^A(x) = \sigma^{-1}(A^T \sigma(x) A)$. It is clear that $f_2^A$ is linear in $x$ for fixed $A$, and also that $f_2^{AA'} = f_2^{A'} \circ f_2^A$ for all $A, A'$. Thus, to prove that $|\det (f_2^A)| = |\det A|^{D+1}$ for all $A$ (and in particular for $A = W$) it suffices, by the existence of singular value decomposition, to prove this in the cases (i) $A$ is diagonal and (ii) $A$ is orthogonal. 

When $A$ is diagonal with diagonal entries $\lambda_1,\dots, \lambda_D$ then one sees by inspection that $(f_2^A(x))_{ij} = \lambda_i \lambda_j x_{ij}$, which renders the result clear in this case. When $A$ is orthogonal, 
$f_2^A$ preserves a non-degenerate quadratic form, namely $\psi(x) := \Vert \sigma(x) \Vert_{\operatorname{HS}}$, where $\Vert \cdot \Vert_{\operatorname{HS}}$ is the Hilbert-Schmidt norm. Therefore in this case $|\det f_2^A| = 1$.

This completes the proof in the case $D = s$. Now suppose that $s < D$.  In this case, we first complete $w_1,\dots, w_{s}$ to a set $w_1,\dots, w_D$ of unit vectors with 
\begin{equation}\label{vol-comp} \vol(w_1,\dots, w_D) \geq \vol(w_1,\dots, w_{s}).\end{equation} For the proof that this is possible, see Corollary \ref{lemc4}.

Let $\tilde f : \R^{D(D+1)/2} \rightarrow \R^{D(D+1)/2}$ be \[ \tilde f(x) = \sigma^{-1} \big( (\langle (I + \sigma(x)) w_i, (I + \sigma(x))w_j\rangle)_{1 \leq i, j \leq D} \big);\]
then $f = \pi \circ \tilde f$, where $\pi : \R^{D(D+1)/2} \rightarrow \R^{s(s+1)/2}$ is the natural projection, i.e.
\[ \pi((a_{ij})_{1 \leq i \leq j \leq D}) = (a_{ij})_{1 \leq i \leq j \leq s}.\]
The case $s =D$ just established shows that for any measurable $\tilde U$
\begin{equation}\label{tilde-case} \P_{\mathbf{e}} (\tilde f(\mathbf{e}) \in \tilde U) \leq D^{3D^2} \vol(w_1,\dots, w_D)^{-D-1} \tilde{\mu}(\tilde U),\end{equation} where $\tilde \mu$ is the Lebesgue measure on $\R^{D(D+1)/2}$. 

Now suppose that $U \subset \R^{s(s+1)/2}$. Then, since (by \eqref{e-operator}) $\Vert \tilde f (\mathbf{e}) \Vert_{\infty} \leq 8$ for all $\mathbf{e} \in [-\frac{1}{D^4} , \frac{1}{D^4} ]^{D(D+1)/2}$, we see using \eqref{tilde-case} that 
\begin{align*} \P_{\mathbf{e}} (f(\mathbf{e}) \in U) & = \P_{\mathbf{e}} (\tilde f(\mathbf{e}) \in \pi^{-1}(U)) \\ & = \P_{\mathbf{e}}(\tilde f(\mathbf{e}) \in \pi^{-1}(U) \cap [-8,8]^{D(D+1)/2}) \\ & \leq D^{3D^2} \vol(w_1,\dots, w_D)^{-D-1} \tilde \mu(\pi^{-1}(U) \cap [-8,8]^{D(D+1)/2}) \\ & \leq D^{4D^2} \vol(w_1,\dots, w_D)^{-D-1} \mu(U). \end{align*}
To complete the proof, apply \eqref{vol-comp}.
\end{proof}

\part{Small gaps and quadratic forms}

\section{Determinants of random matrices}\label{random-matrix-sec}

In this section the main result is Lemma \ref{random-matrix} below, which gives an upper bound on the probability that certain random matrices with fixed diagonal entries have very small determinant. The key ingredient is the following lemma about random symmetric matrices with uniform entries, without any restriction on the diagonal.

\begin{lemma}\label{random-sym}
Let $n$ be sufficiently large and let $W$ be a random symmetric $n \times n$ matrix with entries drawn uniformly and independently from $[-\frac{1}{n},\frac{1}{n}]$. Then $\P(|\det W| \leq \delta) \leq n^{2n^2} \delta$.
\end{lemma}
\begin{proof}
The idea is to compare $W$ with a random (symmetric) matrix $Z$ from the Gaussian Orthogonal Ensemble (GOE) and then compute using the joint density function for eigenvalues there, which is explicit.
The density function of $\mbox{GOE}(n)$ is $f(Z) = c_n e^{-\frac{1}{4} \tr (Z^2)}$, where $c_n  = 2^{-n/2} (2\pi)^{-n(n+1)/4}$ (see \cite[equation 2.5.1]{agz}). If $W$ has entries in $[-\frac{1}{n},\frac{1}{n}]$ then $\tr W^2 \leq 1$, so $f(W) \geq c_n e^{-1/4}$. Therefore we have the comparison estimate
\begin{equation}\label{comparison} \P_W (|\det W| \leq \delta) \leq (\frac{n}{2})^{n(n+1)/2} c_n^{-1} e^{1/4} \P_Z (|\det Z| \leq \delta).\end{equation}

It remains to estimate the right-hand side. For this, we use the well-known formula (see \cite[Theorem 2.5.2]{agz}) for the joint eigenvalue density of $\mbox{GOE}(n)$, together with the fact that the determinant is the product of the eigenvalues. This tells us that $\P_Z(|\det Z| \leq \delta)$ is a normalising constant times
\[ I := \int_{\substack{ \lambda_1 \geq \lambda_2 \geq \dots \geq \lambda_n \\ |\lambda_1 \cdots \lambda_n| \leq \delta}} \prod_{i < j} |\lambda_i - \lambda_j| e^{-\frac{1}{4} \sum_{i = 1}^n \lambda_i^2} d\lambda_1 \dots d\lambda_n .\] The normalising constant is (far) less than $1$ and we do not need anything else about it.
By symmetry, $I$ is equal to 
\[ \int_{\substack{ |\lambda_1| \geq |\lambda_2| \geq \dots \geq |\lambda_n| \\ |\lambda_1 \cdots \lambda_n| \leq \delta}} \prod_{i < j} |\lambda_i - \lambda_j| e^{-\frac{1}{4} \sum_{i = 1}^n \lambda_i^2} d\lambda_1 \dots d\lambda_n .\] Now with this ordering we have $|\lambda_i - \lambda_j| \leq 2 |\lambda_i|$ whenever $j > i$, and so 
\begin{equation}\label{i-desym}  I \leq 2^{n(n-1)/2} \int_{|\lambda_1 \cdots \lambda_n| \leq \delta} |\lambda_1|^{n-1} |\lambda_2|^{n-2} \cdots |\lambda_{n-1}| e^{-\frac{1}{4} \sum_{i = 1}^n \lambda_i^2} d\lambda_1 \dots d\lambda_n  .\end{equation}
Now one may easily check the real-variable inequality $|x|^{k-1} e^{-x^2/8} \leq k^k$ for all positive integers $k$ and all $x \in \R$, which implies that $|x|^k e^{-x^2/4} \leq k^k |x| e^{-x^2/8}$. Substituting into \eqref{i-desym} gives
\[ I \leq 2^{n(n-1)/2} (\prod_{k = 1}^{n-1} k^k) \int_{|\lambda_1 \cdots \lambda_n| \leq \delta}  |\lambda_1\lambda_2 \dots \lambda_{n-1}| e^{-\frac{1}{8} \sum_{i = 1}^n \lambda_i^2} d\lambda_n \cdots d\lambda_1.\]
The inner integral over $\lambda_n$ is 
\[ \int_{|\lambda_n| \leq \delta/|\lambda_1 \cdots \lambda_{n-1}|} e^{-\frac{1}{8}\lambda_n^2} d\lambda_n \leq \frac{2\delta}{|\lambda_1 \cdots \lambda_{n-1}|},\] and therefore
\[ I \leq 2^{n(n-1)/2+1} \delta  (\prod_{k=1}^{n-1} k^k) \int_{\R^{n-1}} e^{-\frac{1}{8}\sum_{i = 1}^{n-1} \lambda_i^2} d\lambda_{n-1} \dots d\lambda_1  = C_n \delta,\]
where
\[ C_n := 2^{n(n-1)/2+1} (8 \pi)^{(n-1)/2}  (\prod_{k=1}^{n-1} k^k) .\]
For $n$ large, a crude bound is $C_n \leq n^{n^2}$ and so the result follows from this and \eqref{comparison}.\end{proof}

\emph{Remarks.} For us, the dependence on $\delta$ (which is sharp) is the important thing. So long as we were not ridiculously profligate, the $n$-dependence of the constant in Lemma \ref{random-sym} was of secondary importance.  However, a weaker $\delta$-dependence such as $\delta^{1/n}$ (which follows from a Remez-type inequality, just treating $\det$ as an arbitrary degree $n$ polynomial) would lead to a weaker exponent in our main theorem.

For fixed $n$, asymptotic formulae of the form $\P(|\det Z| \leq \delta) = (\beta_n  + o_{\delta \rightarrow 0}(1))\delta$ can be extracted from \cite{delannay-lecaer} (I thank Jon Keating for bringing this reference to my attention). So far as I am aware nothing of this type is known for the uniform distribution on matrix entries. 

Here is the result we will actually need in the next section. It requires us to be able to fix the diagonal entries of the symmetric matrix of interest.

\begin{lemma}\label{random-matrix}
Let $m$ be sufficiently large, and let $a$ be an $m \times m$ upper-triangular matrix selected at random as follows. Fix diagonal entries $a_{ii}$ with $1 \leq a_{11},\dots, a_{mm} \leq Q$, and select the off-diagonal entries $a_{ij}$, $i < j$, independently and uniformly at random from $[-Q, Q]$. Then $\P(|\det (a + a^T)| \leq \delta) \leq (Qm)^{5m^2} \delta$, uniformly in the fixed choice of the diagonal terms $a_{ii}$.
\end{lemma}

\begin{proof}
The idea is to amplify the problem so that we are considering a full set of symmetric matrices rather than those with fixed diagonal. To this end, consider the map
\[ \Psi : (1,2)^{m} \times (-Q, Q)^{m(m-1)/2} \rightarrow (1,4Q)^m \times (-4Q, 4Q)^{m(m-1)/2}\] defined by $\Psi(D, M) := D M D$, where here $D$ is a diagonal $m \times m$ matrix with entries in $(1,2)$ (with the space of these being identified with $(1,2)^m$), $M$ is a symmetric matrix with $a_{11},\dots, a_{mm}$ on the diagonal (with the space of these being identified with $\R^{m(m-1)/2}$ by restriction to the entries above the diagonal) and $DMD$ is a symmetric matrix with diagonal entries in the interval $(1,4Q)$ (with the space of these being identified with $(1,4Q)^m \times \R^{m(m-1)/2}$ by restriction to the diagonal and the entries above it). Note that $\Psi$ is a diffeomorphism onto an open set $U = \Psi((1,2)^m \times (-Q,Q)^{m(m-1)/2}) \subset (1,4Q)^m \times (-4Q, 4Q)^{m(m-1)/2}$ with smooth inverse $\Psi^{-1} : U \rightarrow (1,2)^m \times (-Q, Q)^{m(m-1)/2}$ given by
\[ \Psi^{-1} (x) =  \big(  (x_{ii}/a_{ii})^{1/2}_{i = 1,\dots, m},  \big( ( a_{ii} a_{jj} /x_{ii} x_{jj} )^{1/2}x_{ij} )_{1 \leq i < j \leq m} \big). \] All of the partial derivatives of this map are bounded in modulus by $4Q^2$ on its domain, and therefore (crudely) the Jacobian of $\Psi^{-1}$ is bounded by $(\frac{1}{2} m (m+1))! (4Q^2)^{m(m+1)/2}$. Therefore (crudely, and using the fact that $m$ is sufficiently large)
\begin{equation}\label{jac-lower} |\mbox{Jac}(\Psi)| \geq \frac{1}{(\frac{1}{2} m (m+1))! (4Q^2)^{m(m+1)/2}} \geq (Qm)^{-2m^2},\end{equation} uniformly on the domain.

Now set $\Omega := \{ M \in [-Q,Q]^{m(m-1)/2} : |\det M| \leq \delta\}$ (where, recall, $[-Q, Q]^{m(m-1)/2}$ is being identified with the space of symmetric matrices with $a_{11},\dots, a_{mm}$ on the diagonal); our task is to give an upper bound for $\mu_{\R^{m(m-1)/2}}(\Omega)$. First observe that
\begin{equation}\label{compare-1} \mu_{\R^{m(m+1)/2}}( (1,2)^m \times \Omega) = \mu_{\R^{m(m-1)/2}}(\Omega).\end{equation}
To estimate the left-hand side, note that if $M \in \Omega$ and $D$ is diagonal with entries in $(1,2)$ then $\det (\Psi(D,M)) = (\det D)^2 \det M \leq 4^m \delta$. 
Since, moreover,
\[ \Psi((1,2)^{m} \times \Omega) \subset [-4Q, 4Q]^{m(m+1)/2},\] we may use Lemma \ref{random-sym} (rescaling the sample space by a factor $4Qm$) to conclude that 
\begin{align*} \mu_{\R^{m(m+1)/2}} (\Psi ((1,2)^m \times \Omega)) & \leq (4Qm)^{m(m+1)/2} m^{2m^2} (Qm)^{-m} \delta \\ & < (Qm)^{3m^2} \delta.  \end{align*}
By \eqref{jac-lower} and change of variables, it follows that 
\begin{equation}\label{compare-3} \mu_{\R^{m(m+1)/2}} ((1,2)^m \times \Omega)) < (Qm)^{5 m^2}\delta.\end{equation}
Comparing with \eqref{compare-1} concludes the proof.
\end{proof}

\section{An application of the circle method}

In this section we establish the key ingredient in Proposition \ref{main-quad-gap}, which is Proposition \ref{cont-disc-compare}.

For any parameter $\delta > 0$, we will make use of a cutoff function $\chi := \chi_{\delta} : \R \rightarrow [0,\infty)$ satisfying the following properties, where the implied constants are absolute and do not depend on $\delta$.

\begin{enumerate}
\item $\chi(x) \geq 1$ for $|x| \leq \delta/2$;
\item $\chi(x) = 0$ for $|x| > \delta$;
\item $\int \chi \ll \delta$;
\item $\Vert \hat{\chi} \Vert_1 \ll 1$;
\item $\int_{|\xi| > \delta^{-2}} |\hat{\chi}(\xi)| \ll \delta^{2}$.
\end{enumerate}

For the proof that such a $\chi$ exists, see Lemma \ref{chi-15-exist}.

\begin{proposition}\label{cont-disc-compare}
Let $B > 1$. Set $m = C_1 B$. Suppose that $L > (Qm)^{m}$. Let $L_1,\dots, L_m \in [L, L^{1 + 1/48}]$ be lengths. Denote by $\mu$ the Lebesgue measure on $\R^m$, and by $\mu_{\disc}$ the uniform probability measure on the points $(x_1/L_1,\dots, x_m/L_m)$, with the $x_i$ integers satisfying $0 \leq x_i < L_i$. 

Fix $1 \leq a_{11},\dots, a_{mm} \leq Q$, and choose $a_{ij}$, $1 \leq i < j \leq s$  independently and uniformly at random from $[-Q, Q]$. For $b = (b_1,\dots, b_m)$ and $c \in \R$ and write $q_{a,b.c} : \R^s \rightarrow \R$ for the quadratic form defined by $q_{a,b,c}(t) := \sum_{i \leq j} a_{ij} t_i t_j + \sum_i b_i t_i + c$. Let $\chi = \chi_{L^{-B}}$ with $\chi$ as above, and let $w : \R \rightarrow \R$ be a smooth function supported on $[0,1]$ with $\Vert w \Vert_{\infty} \leq 1$ and $\Vert w' \Vert_{\infty}, \Vert \hat{w} \Vert_1 \leq L^{1/C_1}$. Then with probability at least $1 - L^{-B}$ in the random choice of $a$ we have
\begin{equation}\label{key-est}  \big| \int w^{\otimes m}(t) \chi(q_{a,b,c}(t)) d\mu(t) - \int w^{\otimes m}(t) \chi(q_{a,b,c}(t)) d\mu_{\disc}(t) \big| \leq  L^{-B - 1/4} ,\end{equation}
for all $b \in \R^m$, $c \in \R$ with $|b_i|, |c| \leq Q$. \end{proposition}
\emph{Remarks.} Here, $w^{\otimes m}(t) = w(t_1) \cdots w(t_m)$. 

The detailed statement is somewhat complicated. What it says, roughly, is that for almost all $a$ the distribution of the quadratic form $q_{a,b,c}(t)$ on discrete points $t = (x_1/L_1,\dots, x_m/L_m)$ is closely approximated by the distribution over all of $[0,1]^m$, even on the level of rather short intervals of length $L^{-B}$. The two smoothings $\chi, w$ of course make the statement look more exotic, but are necessary for technical reasons in the proof.

The need to fix the diagonal terms $a_{11},\dots, a_{mm}$ and only let the off-diagonal terms vary randomly is important for the key application of the proposition in the next section. This also makes the argument somewhat more complicated. 

The proof is rather lengthy. The reader will lose almost nothing should they wish to look through the proof in the case $L_1 = \cdots = L_m$, $Q = 1$ and without worrying about the dependence on $w$; it is then fairly clear that the argument can be modified, provided one makes suitable assumptions on $Q$ and $w$, so that it works under the slightly looser hypotheses.

Let us outline the proof of the proposition. By Fourier inversion on $\chi$ we have, for $\nu = \mu$ or $\nu = \mu_{\disc}$,
\[ \int_{\R^m} w^{\otimes m}(t) \chi(q_{a,b,c}(t))d\nu(t)  = \int_{\R} \hat{\chi}( \xi)  \int_{\R^m} w^{\otimes m}(t)e(\xi q_{a,b,c}(t)) d\nu(t) d\xi.\] 
Write 
\begin{equation}\label{s-sum-def} S_{a,b,c}(\xi) := \int_{\R^m} w^{\otimes m}(t) e(\xi q_{a,b,c}(t)) d\mu_{\disc}(t) \end{equation}
and
\begin{equation}\label{t-sum-def} T_{a.b,c}(\xi) := \int_{\R^m} w^{\otimes m}(t) e(\xi q_{a,b,c}(t)) d\mu(t);\end{equation} the task is then to prove the estimate
\begin{equation}\label{to-prove-quads} \int_{\R} \hat{\chi}(\xi) (T_{a,b,c}(\xi) - S_{a,b,c}(\xi)) d\xi  \ll L^{-B - 1/4} \end{equation} (for all $b,c$, with high probability in $a$).
To prove this, we will analyse various different ranges of $\xi$, proving the following four lemmas. In these lemmas, we assume that the assumptions (on $L$ and $w$) from Proposition \ref{cont-disc-compare} remain in force.

\begin{lemma}\label{quad-1}
For $|\xi| \leq L^{1/8}$, we have $|S_{a,b,c}(\xi) - T_{a,b,c}(\xi)| \ll L^{-1/2}$, uniformly for all $a, b,c$ with $|a_{ij}|, |b_i|, |c| \leq Q$.
\end{lemma}

\begin{lemma}\label{quad-2}
For $|\xi| \geq L^{1/8}$ we have $|T_{a,b,c}(\xi)| \ll L^{-2B}$,  uniformly in $a$ with $|a_{ij}| \leq Q$ and $\det (a + a^T) \geq L^{-2B}$, and for all $b,c$ with $|b_i|, |c| \leq Q$.
\end{lemma}

\begin{lemma}\label{quad-3-med}
Suppose that $L^{1/8} < |\xi| < L^{5/4}$ and that $\det (a + a^T) \geq L^{-2 B}$. Then $\max_{b,c} |S_{a,b,c}(\xi)| \leq L^{-2B}$.
\end{lemma}

\begin{lemma}\label{quad-3-large}
For each fixed $\xi$, $L^{5/4} \leq |\xi| \leq L^{2B}$, we have 
\[ \P_a ( \max_{b,c} |S_{a,b,c}(\xi)| \geq L^{-2B} ) \leq L^{-7B}.\]
\end{lemma}
The most involved part of the argument is the proof of Lemmas \ref{quad-3-med} and \ref{quad-3-large}. Before turning to the proofs of the lemmas, let us see how they assemble to give a proof of Proposition \ref{cont-disc-compare}, via \eqref{to-prove-quads}.  

\begin{proof}[Proof of Proposition \ref{cont-disc-compare}] (assuming Lemmas \ref{quad-1} -- \ref{quad-3-large}) In this argument we write $o(1)$ to denote a quantity tending to zero as $L \rightarrow \infty$. First, note that for each fixed $a,b,c$ with $|a_{ij}|, |b_i| \leq Q$ the sum $S_{a,b,c}(\xi)$ is weakly continuous in $\xi$, in the sense that $|S_{a,b,c}(\xi) - S_{a,b,c}(\xi')| \ll  Qm^2 |\xi - \xi'|$. This follows from the definition \eqref{s-sum-def} and the fact that $|q_{a,b,c}| \ll Qm^2$ on $[0,1]^m$. If $|\xi - \xi'| \leq L^{-3B}$ then, under our assumption on $L$, this comfortably implies that
\begin{equation}\label{weak-cont} |S_{a,b,c}(\xi) - S_{a,b,c}(\xi')| <  L^{-2B}.\end{equation}

Let $\xi_1,\dots, \xi_{L^{5B}}$ be a $L^{-3B}$-dense set of points in $[L^{1/8}, L^{2B}]$.  Suppose that $\det(a + a^T) \geq L^{-2 B}$. Then by Lemmas \ref{quad-3-med}, \ref{quad-3-large} and the union bound we have 
\[ \P_a(\max_i \max_{b,c} |S_{a,b,c}(\xi_i)| \geq L^{-2B}) \leq L^{-2B} = o( L^{-B}),\] and so by \eqref{weak-cont}
\[ \P_a(\max_{L^{1/8} \leq |\xi| \leq L^{2B}} \max_{b,c}  |S_{a,b,c}(\xi)| \geq 2L^{-2B}) = o(L^{-B}).\]
That is, with probability at least $1 -o(L^{-B})$ in $a$, 
\begin{equation}\label{good-est} |S_{a,b,c}(\xi)| \leq 2L^{-2B} \; \mbox{for $L^{1/8} \leq |\xi| \leq L^{2B}$ and all $b,c$ with $|b_i|, |c| \leq Q$.}\end{equation}

Now by Lemma \ref{random-matrix}, the probability that $|\det(a + a^T)| < L^{-2B}$ is at most $(Qm)^{5m^2} L^{-2 B}$, which is certainly $o(L^{-B})$ with the assumption $L > (Qm)^m$. Suppose from now on that $a$ has the property \eqref{good-est} and that $\det(a + a^T) > L^{-2B}$. We have shown that this is true with probability $1 - o(L^{-B})$.

Returning to the main task \eqref{to-prove-quads}, we divide into low-, middle- and high-frequency ranges. For the low-range frequencies $|\xi| \leq L^{1/8}$ we use Lemma \ref{quad-1} and the trivial bound $|\hat{\chi}(\xi)| \leq L^{-B}$ (which follows from the fact that $\int \chi \ll L^{-B}$, which is property (3) of $\chi$), obtaining
\begin{align} \nonumber 
\int_{|\xi| \leq L^{1/8}} \hat{\chi}( \xi)  (T_{a,b,c}(\xi) - S_{a,b,c}(\xi)) d\xi & \ll L^{-1/2} \int_{|\xi| \leq L^{1/8}} |\hat{\chi}(\xi)| \\ & = o(L^{-B - 1/4}) .\label{small-range}\end{align}

For the middle range frequencies $L^{1/8} < |\xi| < L^{2B}$ we have, by Lemma \ref{quad-2},
\begin{equation}\label{last-1} \int_{L^{1/8} < |\xi| < L^{2B}} |\hat{\chi}( \xi) T_{a,b,c}(\xi)| d\xi  \ll L^{-2B} \int_{\R} |\hat{\chi}( \xi)| d\xi \ll L^{-2B},\end{equation} where the last inequality follows from the fact that $\Vert \hat{\chi} \Vert_1 \ll 1$, which is item (4) of the list of properties satisfied by $\chi$. By \eqref{good-est}, we similarly have
\begin{equation}\label{last-2} \int_{L^{1/8} < |\xi| < L^{2B}} |\hat{\chi}( \xi) S_{a,b,c}(\xi)| d\xi  \ll L^{-2B} \int_{\R} |\hat{\chi}(\xi)| d\xi \ll L^{-2B}.\end{equation}  By the triangle inequality, \eqref{last-1} and \eqref{last-2} together give
\begin{equation}\label{middle-range} \int_{L^{1/8} < |\xi| < L^{2B}} \hat{\chi}(\xi) (T_{a,b,c}(\xi) - S_{a,b,c}(\xi)) d\xi  \ll L^{-2B}.\end{equation}
Finally, for the high-frequencies $|\xi| \geq L^{2B}$ we use the trivial bounds $|S_{a,b,c}(\xi)|, |T_{a,b,c}(\xi)| \leq 1$ (both of which follow immediately from the definitions \eqref{s-sum-def}, \eqref{t-sum-def}, remembering that $w$ is supported on $[0,1]$ and has $\Vert w \Vert_{\infty} \leq 1$) and the estimate $\int_{|\xi| \geq L^{2B}} |\hat{\chi}(\xi)| \ll L^{-2B}$ (item (5) on the list of properties satisfied by $\chi$) to get \begin{equation}\label{high-range} \int_{|\xi| > L^{2B}} \hat{\chi}(\xi) (T_{a,b,c}(\xi) - S_{a,b,c}(\xi)) d\xi  \ll  \int_{|\xi| > L^{2B}} |\hat{\chi}(\xi)|  \ll L^{-2B}. \end{equation}
Putting \eqref{small-range}, \eqref{middle-range} and \eqref{high-range} together completes the proof of \eqref{to-prove-quads}, this having been shown to be true with probability $1 - o(L^{-B})$ in $a$ (and for all $b, c$ with $|b_{i}|, |c| \leq Q$). This completes the proof of Proposition \ref{cont-disc-compare}, subject of course to proving Lemmas \ref{quad-1}, \ref{quad-2}, \ref{quad-3-med} and \ref{quad-3-large}. \end{proof}

We now begin the task of proving those four lemmas.
\begin{proof}[Proof of Lemma \ref{quad-1}] For any function $f$ supported on $[0,1]^m$ we have
\[ \int_{\R^m} f(t) d\mu(t) = \sum_{0 \leq x_i < L_i} \int_{\prod_{i=1}^m [0,\frac{1}{L_i}]^m} f(\frac{x}{L_i} + t) d\mu(t).\]
However for $t \in [0,\frac{1}{L_i}]^m$ the mean value theorem gives
\[ |f(\frac{x}{L_i} + t) - f(\frac{x}{L_i})| \leq \frac{m}{L} \max_j  \Vert\partial_j f\Vert_{\infty}, \] whilst
\begin{align*}
\sum_{0 \leq x_i < L_i} \int_{\prod_{i = 1}^m [0, \frac{1}{L_i}]^m} f(\frac{x}{L_i}) d\mu(t) & = \frac{\lceil L_1 \rceil \cdots \lceil L_m \rceil}{L_1 \cdots L_m} \int f d\mu_{\disc} \\ & = \int f d\mu_{\disc} + O(\frac{m}{L} \Vert f \Vert_{\infty}).
\end{align*}
Therefore
\begin{equation}\label{152a} \big| \int_{\R^m} f(t) d\mu(t) - \int_{\R^m} f(t) d\mu_{\disc}(t) \big| \ll \frac{m}{L} \max_j  \Vert\partial_j f\Vert_{\infty} + \frac{m}{L} \Vert f \Vert_{\infty}.\end{equation}
Taking $f(t) = w^{\otimes m}(t) e(\xi q_{a,b,c}(t))$, we see that for $t \in [0,1]^m$ 
\begin{align}\nonumber \partial_j f(t) & = w^{\otimes w}(t) \partial_j e(\xi q_{a,b,c}(t)) + \partial_j (w^{\otimes m}(t)) e(\xi q_{a,b,c}(t)) \\ &\ll m |\xi| Q + \Vert w' \Vert_{\infty}.\label{152b}\end{align}
Now $|\xi| \leq L^{1/8}$, and the assumptions of Proposition \ref{cont-disc-compare} guarantee that the terms $Qm, \Vert w' \Vert_{\infty}$ are much smaller than $L^{1/8}$. The result follows by combining \eqref{152a}, \eqref{152b}. 
\end{proof}

\begin{proof}[Proof of Lemma \ref{quad-2}.] Recall the definition \eqref{t-sum-def} of $T_{a,b,c}$, that is to say 
\begin{equation}\label{t-sum-rpt} T_{a,b,c}(\xi) = \int_{\R^m} w^{\otimes m}(t) e(\xi q_{a,b,c}(t)) d\mu(t),\end{equation} where 
\[ q_{a,b,c}(t) = \sum_{i \leq j} a_{ij} t_i t_j + \sum_i b_i t_i + c = \frac{1}{2}t^T (a + a^T) t + b^T t + c.\]
Let $\lambda_1,\dots, \lambda_m$ be the eigenvalues of $a + a^T$. Thus by assumption we have
\begin{equation}\label{det-bd}\lambda_1 \cdots \lambda_m = \det(a + a^T) \geq L^{-2B}.\end{equation} Let $\Psi$ be an orthogonal matrix so that $\Psi^T (a + a^T) \Psi = D$, where $D$ is the diagonal matrix with entries $\lambda_1,\dots, \lambda_m$. Making the change of variables $t = \Psi u$ in \eqref{t-sum-rpt} gives
\begin{equation}\label{t-fourier} T_{a,b,c}(\xi) = \int_{\R^m}  w^{\otimes m}(\Psi u) \prod_{j=1}^m e(\frac{1}{2}\xi \lambda_j u_j^2 + \alpha_j u_j + \beta) du_1\cdots du_m.\end{equation}  Here, $\alpha_1,\dots, \alpha_m,\beta$ are real numbers depending on $\xi, b, c, \Psi$, but their precise identity is unimportant. Note that if $\Psi u \in [0,1]^m$ then $\Vert u \Vert_{\infty} \leq \Vert u \Vert_2 = \Vert \Psi u \Vert_2  \leq \sqrt{m}$. Applying Fourier inversion to the cutoff $w^{\otimes m}$ we then obtain
\begin{align*} w^{\otimes m}(\Psi u) & = 1_{\Vert u \Vert_{\infty} \leq \sqrt{m}}  w^{\otimes m}(\Psi u) \\ & =  \prod_{j=1}^m \int_{\R^m} \hat{w}( \gamma_j) 1_{[-\sqrt{m}, \sqrt{m}]}(u_j) e( \gamma_j (\Psi u)_j) d\gamma_j.\end{align*} This, \eqref{t-fourier} and the triangle inequality imply that
\begin{equation}\label{to-use-t} |T_{a,b,c}(\xi)| \leq \Vert \hat{w} \Vert_1^m \prod_{j=1}^m \sup_{\alpha'_j} \big| \int^{\sqrt{m}}_{-\sqrt{m}} e(\frac{1}{2}\xi \lambda_j u_j^2 + \alpha'_j u_j) du_j \big|.\end{equation}
Now we use the fact that
\begin{equation}\label{gaussian} \int^{Y_2}_{Y_1} e(x^2) dx = O(1),\end{equation} uniformly in $Y_1, Y_2$. To see this, divide into positive and negative ranges, and make the substitution $x^2 = w$; it then suffices to show that $\int^Y_0 e(w) w^{-1/2} dw = O(1)$, uniformly in $Y$. Bounding the portion of the integral on $[0,1]$ trivially, it is sufficient to show that $\int^Y_1 e(w) w^{-1/2} dw = O(1)$, uniformly in $Y \geq 1$. This can be done by integration by parts, since $e(w) w^{-1/2}$ is bounded uniformly and moreover \[ |\int^Y _1 e(w) w^{-3/2} dw| < \int^{\infty}_1 w^{-3/2} dw = O(1).\]

Completing the square and making a substitution in \eqref{gaussian}, we have
\[ \int^{\sqrt{m}}_{-\sqrt{m}} e(\frac{1}{2}\xi \lambda_j u_j^2 + \alpha'_j u_j) du_j \ll  |\xi \lambda_j|^{-1/2},\] with the implied constant absolute. Substituting into \eqref{to-use-t}, it follows that \begin{align*} T_{a,b,c}(\xi)  \leq O(1)^m \Vert \hat{w} \Vert_1^m & \prod_{j=1}^m |\xi \lambda_j|^{-1/2} \leq O(1)^m \Vert \hat{w} \Vert_1^m  |\xi|^{-m/2} L^{B}  \\ & \leq O(1)^m \Vert \hat{w} \Vert_1^m L^{-m/16} L^{B} \leq L^{-2B}  ,\end{align*}where this last line of inequalities follows from the assumption that $|\xi| \geq L^{1/8}$, \eqref{det-bd},  $m = C_1 B$ (with $C_1 \geq 50$) and the assumptions that $\Vert \hat{w} \Vert_1 \leq L^{1/C_1}$ and $L > m^m$.\end{proof}

\emph{Remark.} Note how important it was in this proof that we had a smooth cutoff $w^{\otimes m}$ to $1_{[0,1]^m}$; dealing with a rough cutoff under the orthogonal transformation $\Psi$ is problematic. For a very similar use of this device, see Heath-Brown and Pierce \cite[Section 3]{hbp}.

We turn now to the proofs of Lemmas \ref{quad-3-med} and \ref{quad-3-large}. We first begin with some initial arguments common to the proof of both lemmas.

First, we remove the weight $w^{\otimes m}$ by Fourier expansion. By the inversion formula we have, for $\xi \neq 0$, 
\[ S_{a,b,c}(\xi) = \int_{\gamma \in \R^m} (\prod_{j = 1}^m \hat{w}(\gamma_i) ) S'_{a, b +  \frac{\gamma}{\xi},c}(\xi) d\gamma,\] where 
\[ S'_{a,b,c}(\xi) = \int_{\R^m} e(\xi q_{a,b,c}(t)) d\mu_{\disc}(t)\]
is the unsmoothed exponential sum. Therefore for any $\xi$
\begin{equation}\label{unsmooth} \max_{b,c} |S_{a,b,c}(\xi)| \leq \Vert \hat{w} \Vert_1^m \max_{b,c} |S'_{a,b,c}(\xi)|.\end{equation} Note that the maxima here are over \emph{all} $b,c$; in these lemmas we are not assuming any bounds on their size. They will, in any case, shortly disappear from view.

Fix some $\xi$ in the relevant range $L^{1/8} \leq |\xi| \leq L^{2B}$. The initial steps follow the standard proof of Weyl's inequality for quadratic exponential sums. Squaring and making a substitution, we have
\begin{equation} \label{s-lower}|S'_{a,b,c}(\xi)|^2 \leq (L_1 \dots L_m)^{-2} \sum_{h_i \in [-L_i, L_i]} \sum_{x \in B_h} e (\xi (q_{a,b,c}(\frac{\xx+\hh}{\LL})  - q_{a,b,c}(\frac{\xx}{\LL}))),\end{equation} where, for $h \in \Z^m$, $B_h := \prod_{i=1}^m [L_i] \cap (\prod_{i=1}^m [L_i] - h)$.
Here, and in what follows, we use the shorthands
\[ \frac{\xx}{\LL} =(\frac{x_1}{L_1},\dots, \frac{x_m}{L_m}),\quad  \frac{\xx+\hh}{\LL} = (\frac{x_1+h_1}{L_1},\dots, \frac{x_m + h_m}{L_m}) \] (and similar); note this is just notation and we are not dividing vectors, thus we still write $x = (x_1,\dots, x_m)$, $h = (h_1,\dots, h_m)$ without bold font. We have
\[ q_{a,b,c} (\frac{\xx+\hh}{\LL}) - q_{a,b,c}(\frac{\xx}{\LL}) = (\frac{\xx}{\LL})^T (a + a^T) \frac{\hh}{\LL} + q_{a,b,0}(\frac{\hh}{\LL}), \] where here and in what follows we abuse notation slightly and identify $a = (a_{ij})_{1 \leq i \leq j \leq m}$ with an upper-triangular matrix in the obvious way.

Equation \eqref{s-lower} then implies that 
\begin{equation}\label{s-lower-mod} |S'_{a,b,c}(\xi)|^2 \leq (L_1 \cdots L_m)^{-2} \sum_{h_i \in [-L_i, L_i]} \big| \sum_{x \in B_h} e(\xi  (\frac{\xx}{\LL})^T (a + a^T) \frac{\hh}{\LL}) \big|.  \end{equation} Note that $b,c$ no longer appear here, so if $\max_{b,c} |S_{a,b,c}(\xi)| \geq L^{-7B}$ then (using also \eqref{unsmooth} and the assumption $\Vert \hat{w} \Vert^m_1 \leq L^B$) we have
\[  \sum_{h_i \in [-L_i, L_i]} \big| \sum_{x \in B_h} e(\xi (\frac{\xx}{\LL})^T (a + a^T) \frac{\hh}{\LL}) \big| \geq (L_1 \cdots L_m)^2 L^{-16B}.\] Since the inner sum is trivially bounded by $L_1 \cdots L_m$, this means that there is a set $H \subset \prod_{i=1}^m [-L_i, L_i]^m$, $|H| \geq \frac{1}{2}(L_1 \cdots L_m) L^{-16B} \geq (L_1 \cdots L_m) L^{-17B}$, such that 
\begin{align}\nonumber \big| \sum_{x \in B_h} e(\xi (\frac{\xx}{\LL})^T (a + a^T) \frac{\hh}{\LL}) \big| & \geq 2^{-m-1}(L_1 \cdots L_m) L^{-16B} \\ & \geq (L_1 \cdots L_m) L^{- 17B}\label{h-large}\end{align} for $h \in H$. 

\begin{equation}\label{h-large}\big| \sum_{x \in B_h} e(\xi (\frac{\xx}{\LL})^T (a + a^T) \frac{\hh}{\LL}) \big| \geq (L_1 \cdots L_m) L^{- 17B}\end{equation} for $h \in H$. 
For $t \in \R^m$, write 
\[ \ell_{i, a}(t) := ((a + a^T) t)_i, \quad \mbox{$i = 1,\dots, m$}.\] Thus \eqref{h-large} becomes
\[\big| \sum_{x \in B_h} e(\xi \sum_{i=1}^m \frac{x_i}{L_i} \ell_{i,a}(\frac{\hh}{\LL})  )\big| \geq (L_1 \cdots L_m) L^{- 17B}\] for $h \in H$. Noting that, for each $h$, $B_h$ is a sub-box of $\prod_{i=1}^m [L_i]$, we may evaluate the sum as a geometric series and conclude that for $h \in H$ we have
\begin{equation}\label{h-large-geo}  \prod_{i = 1}^m \min(L_i, \Vert \frac{\xi}{L_i} \ell_{i,a}(\frac{\hh}{\LL}) \Vert_{\T}^{-1}) \geq (L_1 \cdots L_m) L^{ - 17B}.  \end{equation}
It follows that for each $h \in H$ there is a set $I(h) \subseteq \{1,\dots, m\}$, $|I(h)| \geq m/2$, such that 
\[ \Vert \frac{\xi}{L_i} \ell_{i,a}(\frac{\hh}{\LL}) \Vert_{\T} \leq L^{-1 + 34B/m} \leq \frac{1}{2}L^{-15/16}\] for all $i \in I(h)$ (recalling that $B = C_1 m$, and if $C_1$ is big enough). Pigeonholing in $h$, we may find a set $I \subseteq \{1,\dots, m\}$, $|I| \geq m/2$, and a set $H'$, \begin{equation}\label{h-dash} |H'| \geq 2^{-m} |H| \geq (L_1 \cdots L_m) L^{-18B},\end{equation} such that 
\[  \Vert \frac{\xi}{L_i} \ell_{i,a}(\frac{\hh}{\LL}) \Vert_{\T} \leq \frac{1}{2}L^{-15/16} \; \mbox{for $h \in H'$ and $i \in I$}.\] Hence, since the $\ell_{i,a}$ are linear, we have
\begin{equation}\label{bigell}  \Vert \frac{\xi}{L_i} \ell_{i,a}(\frac{\hh}{\LL}) \Vert_{\T} \leq L^{-15/16} \; \mbox{for $h \in H' - H'$ and $i \in I$}.\end{equation}
This fact will be the key to the proofs of both Lemma \ref{quad-3-med} and \ref{quad-3-large}, but the subsequent treatment of those two lemmas differs.

\begin{proof}[Proof of Lemma \ref{quad-3-med}] We handle the ranges $L^{1/8} < |\xi| < L^{3/4}$ and $L^{3/4} \leq |\xi| < L^{5/4}$ separately, starting with the latter, which is slightly harder.

Suppose then that $L^{3/4} < |\xi| < L^{5/4}$ and that $\det(a + a^T) \geq L^{-2B}$, as in the statement of the lemma. Note that if $h \in H' - H'$, $i \in I$ and
\begin{equation}\label{first-h-bd} |\ell_{i,a}(\frac{\hh}{\LL})| < L^{-1/3}\end{equation}
then
\[  \frac{|\xi|}{L_i} |\ell_{i,a}(\frac{\hh}{\LL})| < \frac{L^{5/4 - 1/3}}{L_i} <  \frac{1}{2}\] and so \[ \Vert \frac{\xi}{L_i} \ell_{i,a}(\frac{\hh}{\LL}) \Vert _{\T} = \frac{|\xi|}{L_i} |\ell_{i,a}(\frac{\hh}{\LL})|.\] Therefore, if \eqref{bigell}, \eqref{first-h-bd} hold then (using $L_i \in [L, L^{1 + 1/48}]$),
\[ |\ell_{i,a}(\frac{\hh}{\LL})| \leq L^{-15/16} \frac{L_i}{|\xi|} < L^{-15/16 }\frac{L^{1+1/48}}{L^{3/4}} = L^{-2/3}.\] 
Thus we have shown that if \eqref{bigell} holds then 
\begin{equation}\label{do-not} \mbox{for $h \in H' - H'$, $i \in I$ we do not have} \;  L^{-2/3} < |\ell_{i,a}(\frac{\hh}{\LL})| < L^{-1/3}.\end{equation}
To analyse \eqref{do-not} we need the following lemma. 
\begin{lemma}\label{diff-lem}
Let $S \subset [-2mQ,2mQ] \subset \R$ be a set, and suppose that $S - S$ contains no element in $[2L^{-2/3}, \frac{1}{2}L^{-1/3}]$. Then $\mu_{\R}(S) \ll mQ L^{-1/3}$.
\end{lemma}
\begin{proof}
Cover $[-2mQ,2mQ]$ with $O( QmL^{1/3})$ disjoint intervals of length $\frac{1}{2}L^{-1/3}$. The intersection of $S$ with any such interval has measure at most $2L^{-2/3}$ (it is either empty, or if it contains some $x$ then it only contains points in an interval of diameter $\leq 2L^{-2/3}$ about $x$). 
\end{proof}

Returning to the analysis of \eqref{do-not}, consider the set $X \subset [-1,1]^m$ defined by
\[ X := \{ \frac{\hh}{\LL} : h \in H'\}  + \prod_{i=1}^m [0,\frac{1}{L_i}].\]
Thus, by \eqref{h-dash},
\begin{equation}\label{x-lower} \mu_{\R^m}(X) = (L_1 \cdots L_m)^{-1} |H'| \geq L^{-18B}.\end{equation}
Also, if $x,x' \in X$ then for some $h, h' \in H'$
\[ \ell_{i,a}(x - x') = \ell_{i,a}(\frac{\hh - \hh'}{\LL}) + O(\frac{mQ }{L}).\] Thus,  by \eqref{do-not} (and the assumption that $mQ < L^{\frac{1}{4}}$) we see that for $i \in I$, $\ell_{i,a}(X - X) = \ell_{i,a}(X) - \ell_{i,a}(X)$ contains no element in the interval $[2 L^{-2/3}, \frac{1}{2} L^{-1/3}]$. By Lemma \ref{diff-lem} (and noting that $\ell_{i,a}(X) \subset [-2mQ,2mQ]$), 
\[ \mu_{\R} (\ell_{i,a}(X)) = O(mQ L^{-1/3}) \leq L^{-1/12}.\]
Thus, the image of $X$ under the linear map $\psi : \R^m \rightarrow \R^m$ defined by \[ \psi(x_1,\dots, x_m) := (\ell_{1,a}(x),\dots, \ell_{m,a}(x))\] has measure at most $(2mQ)^m L^{-|I|/12} \leq L^{-m/48}$ (here we have used the fact that $\ell_{j,a}(x)$ takes values in $[-2Qm, 2Qm]$ for all $j$, even if $j \notin I$). But $\det\psi = \det (a + a^T) \geq L^{-2 B}$, and so 
\[ \mu_{\R^m}(X) \ll  (\det \psi)^{-1} L^{-m/48} \ll  L^{2B - m/48}.\] 
Recalling that $m = C_1 B$, this contradicts \eqref{x-lower} if $C_1$ is big enough. This completes the proof of Lemma \ref{quad-3-med} in the range $L^{3/4} \leq |\xi| < L^{5/4}$.

Now consider the remaining range $L^{1/8} < |\xi| < L^{3/4}$. Once again, we refer to \eqref{bigell}, which tells us that 
\begin{equation}\label{bigell-rpt-new}  \Vert \frac{\xi}{L_i} \ell_{i,a}(\frac{\hh}{\LL}) \Vert_{\T} \leq L^{-15/16} \; \mbox{for $h \in H' - H'$ and $i \in I$}.\end{equation}
Now
\[ | \frac{\xi}{L_i} \ell_{i,a}(\frac{\hh}{\LL}) | \leq \frac{L^{3/4}}{L_i} \cdot 2Qm < \frac{1}{2}.\] Therefore \eqref{bigell-rpt-new} implies that 
\[ |\ell_{i,a}(\frac{\hh}{\LL})| \ll L^{-15/16} \frac{L_i}{|\xi|} \leq L^{-15/16} \frac{L^{1 + 1/48}}{L^{1/8}} = L^{-1/24}.\]
That is, 
\begin{equation}\label{do-not-new} \mbox{for $h \in H' - H'$, $i \in I$ we have} \;   |\ell_{i,a}(\frac{\hh}{\LL})| \leq L^{-1/24}.\end{equation} This should be compared with \eqref{do-not}, but the subsequent analysis is easier and does not require Lemma \ref{diff-lem} since we immediately have
\[ \mu_{\R} (\ell_{i,a}(X)) \leq 2L^{-1/24}.\] One may now obtain a contradiction essentially as before, with minor numerical modifications.
\end{proof}

\begin{proof}[Proof of Lemma \ref{quad-3-large}] Suppose that $|\xi| \geq L^{5/4}$, as in the statement of the lemma. Recall the key statement \eqref{bigell} established above, that is to say \begin{equation}\label{bigell-rpt}  \Vert \frac{\xi}{L_i} \ell_{i,a}(\frac{\hh}{\LL}) \Vert_{\T} \ll L^{-15/16} \; \mbox{for $h \in H' - H'$ and $i \in I$}.\end{equation} Recall also (see \eqref{h-dash}) that $|H'| \geq (L_1 \cdots L_m)L^{ - 18B}$.
By an application of the pigeonhole principle (dividing $\prod_{i=1}^m [-L_i, L_i]$ into cubes of sidelength $10 L^{18B/m} \leq L^{19 B/m}$) there is some $h^* \in H' - H'$ with $0 < |h^*| \leq L^{19B/m}$. Then \eqref{bigell-rpt} implies that 
\begin{equation}\label{bigell-smallh}  \Vert \frac{\xi}{L_i} \ell_{i,a}(\frac{\hh^*}{\LL}) \Vert_{\T} \leq 2 L^{-15/16} \; \mbox{for $i \in I$}.\end{equation}
Let us summarise the situation so far: under the assumption that $\sup_{b,c} |S_{a,b,c}(\xi)| \geq L^{-7B}$ and the analysis leading to \eqref{bigell}, we have shown that \eqref{bigell-smallh} holds for some $I \subset \{1,\dots,m\}$ with $|I| \geq m/2$, and for some $h^* \in \Z^m$, $0 < |h^*| \ll L^{19B/m}$. Thus, denoting by $E_{I, h^*}(a)$ be the event that \eqref{bigell-smallh} holds, we have the inclusion of events
\begin{equation}\label{inclu-events} \{ a : \sup_{b,c} |S_{a,b,c}(\xi)| \geq L^{-7B} \} \subset \bigcup_{\substack{I \subseteq [m] \\ |I| \geq m/2}} \bigcup_{0 < |h^*| \leq L^{19B/m}} E_{I, h^*}(a).\end{equation}
We will now bound $\P_a (E_{I, h^*}(a))$ for $I, h^*$ fixed, before applying the union bound to \eqref{inclu-events}. Pick some index $j \in [m]$ such that $h^*_j \neq 0$. We are now going to condition on all except some fairly small subset of the random entries $a_{ij}$. Let $I_- := \{ i \in I : i < j\}$ and $I_+ := \{ i \in I : i > j\}$. 

If  $|I_-| \geq m/6$ then we condition on all except the $a_{ij}$, $i \in I_-$.  With all the other variables fixed, the conditions \eqref{bigell-smallh} for $i \in I_-$ become
\begin{equation}\label{i-plus} \Vert \frac{\xi}{L_i L_j} a_{ij} h^*_j + c_i \Vert_{\T} \leq 2L^{-15/16},\end{equation} where the $c_i$ depend on the fixed variables but not on the random variables $(a_{ij})_{i \in I_-}$. ($c_i$ will depend on the entries $(a + a^T)_{ik}$, $k \neq j$, that is to say on $a_{ik}$ for $k > i$ and on $a_{ki}$ for $k < i$; none of these variables are one of the $a_{i_-,j}$, $i_- \in I_-$.)

If $|I_+| \geq m/6$ then we proceed in very similar fashion, but now we condition on all except the $a_{ji}$, $i \in I_+$. With all the other variables fixed, the conditions \eqref{bigell-smallh} for $i \in I_+$ become
\begin{equation}\label{i-minus} \Vert \frac{\xi}{L_i L_j} a_{ji} h^*_j + c_i \Vert_{\T} \leq 2L^{-15/16},\end{equation} with the $c_i$ as before; note that none of the variables in $c_i$ depends on any $a_{j i_+}$ with $i_+ \in I_+$.)
The treatment of the two cases $|I_-| \geq m/6$ and $|I_+| \geq m/6$ is now essentially identical, so we detail only the former. 

Consider a single value of $i \in I_-$ and a fixed $c_i$. As $a_{ij}$ ranges uniformly in $[-Q,Q]$, $\frac{\xi}{L_i L_j} a_{ij} h^*_j + c_i$ ranges uniformly over a subinterval of $\R$ of length at least $|\xi| |h^*_j|/L_i L_j  \geq L^{-3/4 - 1/24}$ (here we use the hypothesis that $|\xi| \geq L^{5/4}$, as well as the assumption that $L_i, L_j \leq L^{1 + 1/48}$). Thus the probability (in $a$) that \eqref{i-minus} holds is bounded above by $O(L^{-15/16 + 3/4 + 1/24}) <  L^{-1/8}$.

As $i$ ranges over $I_-$, these events are independent. Therefore we see, averaging over all choices of the fixed variables, that 
\[ \P_a E_{I, h^*}(a) \leq \P_a \big(  \Vert \frac{\xi}{L_i L_j} a_{ij} h^*_j + c_i \Vert_{\T} \leq 2L^{-15/16} \; \mbox{for $i \in I_-$} \big) \leq L^{-|I_-|/8} .\] Since $|I_-| \geq m/6$, this is at most $L^{-m/48}$. The same bound holds in the case that $|I_+| \geq m/6$. 

Finally, applying the union bound to \eqref{inclu-events}, noting that the number of events in the union is $\leq 4^m L^{19B}$ we see that indeed
\[ \P_a (\sup_{b,c} |S_{a,b,c}(\xi)| \geq L^{-7B}) \leq 4^m L^{19B - m/48} \leq L^{-2B},\] assuming that $C_1$ is sufficiently large. 
\end{proof}

This concludes the proof of all the lemmas, and hence the proof of Proposition \ref{cont-disc-compare}.

\section{An amplification argument}

In this section, we finally prove our key proposition about the density of values taken by random quadratic forms, Proposition \ref{main-quad-gap}.
\begin{main-quad-gap-rpt}
Let $B \geq 1$ be an exponent and let $Q \geq 1$ be a parameter. Suppose that $s \geq C_1 B^2$ is an integer. Let $L \geq (QB)^{C_2B}$.
Let $L_1,\dots, L_s$ be lengths with $L_i \in [L, L^{1+1/48}]$. Choose an $\frac{1}{2}s(s+1)$-tuple $a = (a_{ij})_{1 \leq i \leq j \leq s}$ of coefficients by choosing the $a_{ij}$  independently and uniformly at random from $[-Q, Q]$, except for the diagonal terms $a_{ii}$ which are selected uniformly from $[32, Q]$.  For $b = (b_1,\dots, b_m)$ and $c \in \R$ and write $q_{a,b.c} : \R^s \rightarrow \R$ for the quadratic form defined by $q_{a,b,c}(t) := \sum_{i \leq j} a_{ij} t_i t_j + \sum_i b_i t_i + c$. Let $\Sigma = \Sigma(a)$ be the event that the set 
\[ \{ q_{a,b,c}(\frac{x_1}{L_1},\dots, \frac{x_s}{L_s}) : 0 \leq x_i < L_i\}\] is $L^{-B}$-dense in $[\frac{1}{2}, \frac{3}{2}]$ for all $b,c$ satisfying $|b_i| \leq Q$, $|c| \leq \frac{1}{4}$ and $b_i^2 - 4a_{ii}c < 0$ for all $i$. Then $\P_a(\Sigma(a)) = 1 - O(L^{-Bs/16})$.
\end{main-quad-gap-rpt}
We follow the strategy outlined in Section \ref{second-step-sec}. First, we establish the following, a kind of preliminary version of the result with a smaller number $m = O(B)$ of variables, but with a much weaker exceptional probability of $O(L^{-B})$.

\begin{proposition}\label{quad-gap-fixed}
Let $B \geq 1$, and let $Q \geq 1$ be a parameter. Let $m = C_1 B$. Let $L_1,\dots, L_m$ be lengths with $L_i \in [L, L^{1+1/48}]$, where $L \geq (QB)^{C_2B}$. Fix diagonal terms $a_{ii}$ with $32 \leq a_{11},\dots, a_{mm} \leq Q$, and select $a_{ij}$, $1 \leq i < j \leq m$, uniformly and independently at random from $[-Q, Q]$. Let $a$ be the (random) upper triangular matrix thus formed. For $b = (b_1,\dots, b_m)$ and $c \in \R$ write $q_{a,b,c}(t) := t^T a t + b^T t + c$. Let $\Sigma$ be the event that the set 
\begin{equation}\label{quad-set-rp} \{ q_{a,b,c}(\frac{x_1}{L_1},\dots, \frac{x_m}{L_m}) : 0 \leq x_i < L_i\}\end{equation} is $L^{-B}$-dense in $[\frac{1}{2}, \frac{3}{2}]$ for all $b, c$ satisfying $|b_i| \leq Q$, $|c| \leq \frac{1}{4}$ and $b_i^2 - 4a_{ii}c < 0$. Then $\P_a(\Sigma(a)) \geq 1 - L^{-B}$.
\end{proposition}
\begin{proof}
We apply Proposition \ref{cont-disc-compare}. Let $\chi$ be a minorant to the interval $I$ of length $L^{-B}$ about the origin constructed in Lemma \ref{chi-15-exist}, as in the statement of Proposition \ref{cont-disc-compare}. Set $\eta :=  (Qm)^{-2}$ and let $w$ be a smooth bump function as constructed in Lemma \ref{w-tent}, thus $w$ is supported on $[0,1]$, with $w = 1$ on $[\eta, 1 - \eta]$, and with $\Vert w' \Vert_{\infty} ,\Vert \hat{w} \Vert_1 = O(\eta^{-1}) \leq L^{1/C_1}$, by the condition on $L$ in the statement of the proposition. This means that the conditions involving $w$ in Proposition \ref{cont-disc-compare} are satisfied.

Suppose that $a$ satisfies the conclusion \eqref{key-est} of Proposition \ref{cont-disc-compare} (which happens with probability $\geq 1 - L^{-B}$).

Let $u \in [\frac{1}{2}, \frac{3}{2}]$. We wish to show that the set \eqref{quad-set-rp} meets $u + I$. If it does not then, in the notation of Proposition \ref{cont-disc-compare}, 
\[ \int w^{\otimes m}(t) \chi(q_{a,b,c-u}(t)) d\mu_{\disc}(t) = 0.\]
By the conclusion of Proposition \ref{cont-disc-compare} it follows that
\begin{equation}\label{to-contradict} \int w^{\otimes m}(t) \chi(q_{a,b,c-u}(t)) d\mu(t) \leq L^{-B - 1/4}.\end{equation}
By contrast we claim that the LHS of \eqref{to-contradict} is in fact $\gg \eta^mL^{-B}$. To prove the claim, let $\eta_2,\dots, \eta_m \in [\eta, 2\eta]$ be arbitrary and consider the function
\[ F(t) = F_{\eta_2,\dots, \eta_m}(t) := q_{a,b,c-u}(t, \eta_2, \dots, \eta_m).\]
Note that 
\begin{align*} F(\eta) = q_{a,b,c-u}(\eta, \eta_2,\dots, \eta_m)  & = q_{a,b,c-u}(0) + O(Qm \eta) \\ & = c - u+  O(Qm\eta) < 0\end{align*} (by the choice of $\eta$, and also since $c \leq \frac{1}{4} < \frac{1}{2} \leq u$) whilst for all $t$
\begin{align*} F(t) & = q_{a,b,c-u}(t, \eta_2,\dots, \eta_m) \\ & = q_{a,b,c}(t,0,0,\dots) - u -O(mQ \eta) \\ & \geq a_{11}(t + \frac{b_1}{2a_{11}})^2 - u - O(mQ \eta).\end{align*} 
In this last step we used that $b_1^2 - 4a_{11} c < 0$.
Since
\[ \max_{1/4 \leq t \leq 3/4} a_{11}(t + \frac{b_1}{2a_{11}})^2 \geq \frac{1}{16} a_{11} \geq 2\] and $u \leq \frac{3}{2}$ it follows that 
\[ \max_{1/4 \leq t \leq 3/4} F(t) \geq 2 - u - O(m Q \eta) > 0.\]
Thus, we may apply the intermediate value theorem to see that there is some $x = x(\eta_2,\dots, \eta_m) \in [\eta, \frac{3}{4}] \subset [\eta, 1 - \eta]$ such that $F(x) = 0$, that is to say
\[ q_{a,b,c} (x(\eta_2,\dots, \eta_m), \eta_2,\dots, \eta_m) = u.\] Thus if $t$ lies in the set
\[ S := [\eta , 1 - \eta]^m \cap \bigcup_{\eta \leq \eta_2,\dots, \eta_m \leq 2 \eta} (x(\eta_2,\dots, \eta_m) + [-\eta L^{-B}, \eta L^{-B}], \eta_2,\dots, \eta_m)\] then $|q_{a,b,c}(t) - u| \leq \frac{1}{2}L^{-B}$. We have $\mu(S) \gg \eta^m L^{-B}$. 

Now if $t \in S$ then (by construction of $\chi$ as in Lemma \ref{chi-15-exist}) we have $\chi(q_{a,b,c-u}(t)) = 1$. Also, since $S \subset [\eta, 1 - \eta]^m$, we have $w^{\otimes}(t) = 1$. It follows that 
\[ \int w^{\otimes m}(t) \chi(q_{a,b,c-u}(t)) d\mu(t) \geq \mu(S) \gg \eta^m L^{-B},\] contradicting \eqref{to-contradict} in view of the assumption that $L > (QB)^{C_2 B}$.
\end{proof}

Proposition \ref{main-quad-gap} is deduced from the preliminary version, Proposition \ref{quad-gap-fixed}, by a kind of amplification argument. It is driven by the following combinatorial lemma.

\begin{lemma}\label{cliques}
Let $m$ be a positive integer, and suppose that $s \geq 16m^2$. Then there are sets $I_1, \dots, I_k \subset [s]$, $k \geq s/16$, with $|I_{\ell}| = m$ and such that the sets of pairs $\{(i,j) : i < j, i, j \in I_{\ell}\}$ are disjoint as $\ell$ ranges over $\ell = 1,\dots, k$.
\end{lemma}
\begin{proof}
Let $p$, $m \leq p < 2m$, be a prime. Consider the projective plane of order $p^2 + p + 1$ over $\F_p$. This contains $p^2 + p + 1$ lines, each with $p+1 > m$ elements. Identifying $[p^2 + p + 1] \subset \N$ with the projective plane in some arbitrary way, we may take $I_1,\dots, I_{p^2+p+1} \subset [p^2 + p + 1]$ to be subsets of these lines, each of size $m$. For these sets $I_{\ell}$, the required disjointness statement is simply the fact that two points in projective space determine a unique line. Provided that $t(p^2 + p + 1) \leq s$, we may embed $t$ disjoint copies of this construction inside $[s]$, and thereby take $k = t(p^2 + p + 1)$. Using the crude bound $p^2 + p + 1 \leq 2p^2 < 8m^2$ and $\lfloor s/8m^2 \rfloor \geq s/16m^2$, the result follows.
\end{proof}
\emph{Remark.} Erd\H{o}s and R\'enyi \cite{erdos-renyi} attribute the use of projective planes in this context to Thoralf Skolem. One may equivalently think of this lemma as a result about embedding cliques $K_m$ into the complete graph $K_s$ in an edge-disjoint fashion, and the lemma states that in the regime $s \sim C m^2$ this may be done quite efficiently, so as to use up a positive proportion of the edges. Similar constructions would allow one to do the same for, say, $s \sim Cm^3$, but it seems to me to be unclear\footnote{Added in proof: Stefan Glock has drawn my attention to the paper \cite{kuzjurin}, which clarifies this issue.} what the situation is when (for instance) $s \sim m^{5/2}$, or even when $s = m^2/10$. Most of the extensive literature on questions of this type considers the case $m$ fixed and $s \rightarrow \infty$.

\begin{proof}[Proof of Proposition \ref{main-quad-gap}]
Condition on the choice of diagonal terms $a_{ii}$; it suffices to prove the required bound, uniformly in each fixed choice of these terms. Let $I_1,\dots, I_k \subset [s]$ be as in Lemma \ref{cliques} above, and let $\Sigma(\ell)$ be the event that the set
\begin{equation}\label{quad-set-restrict} \{ q_{a,b}(\frac{x_1}{L_1},\dots, \frac{x_s}{L_s}) ) : 0 \leq x_i < L_i, x_j = 0 \; \mbox{for $j \notin I_{\ell}$}\}\end{equation} is $L^{-B}$-dense in $[\frac{1}{2}, \frac{3}{2}]$ for all $b$ with $|b_i| \leq Q$, $|c| \leq \frac{1}{4}$ and $b_i^2 - 4a_{ii} c < 0$ for all $i$. Since the set \eqref{quad-set-restrict} is contained in the set \eqref{quad-set-rp}, we have the containment $\Sigma(\ell) \subset \Sigma$.

Note moreover that the event $\Sigma(\ell)$ only depends on the variables $(a_{ij})_{i < j}$ with $i, j \in I(\ell)$. By construction, these sets of variables are disjoint as $\ell$ varies and therefore the events $\Sigma(\ell)$, $\ell = 1,\dots, k$, are independent. 

By Proposition \ref{quad-gap-fixed}, $\P(\neg \Sigma(\ell)) \leq L^{-B}$. It follows that 
\[ \P(\neg \Sigma) \leq \P \big( \bigwedge_{\ell = 1}^k \neg \Sigma(\ell) \big) \leq L^{-Bk}  \leq L^{-Bs/16}.\]
Averaging over the choices of $a_{11}, \dots, a_{ss}$ (that is, undoing the conditioning on diagonal terms), this at last concludes the proof of Proposition \ref{main-quad-gap}.
\end{proof}

We have now proven all three of the ingredients stated in Section \ref{second-step-sec}, and hence by the arguments of that section Proposition \ref{second-step} is true. This completes the proof that there are (with high probability) no progressions of length $N^{1/r}$ in the red points of our colouring, and hence finishes the proof of Theorem \ref{mainthm-1}.

\part*{Appendix}

\appendix

\section{Lattice and geometry of numbers estimates}

In this appendix we use the convention (also used in the main paper) that if $x \in \Z^D$ then $|x|$ means $\Vert x \Vert_{\infty}$.The following lemma is probably standard, but we do not know a reference.

\begin{lemma}\label{hermite}
Suppose that $\Lambda, \Lambda'$ are two $m$-dimensional lattices and that $\Lambda' \leq \Lambda$. Let $(e_i)_{i = 1}^m$ be an integral basis for $\Lambda$. Then there is an integral basis $(e'_i)_{i = 1}^m$ for $\Lambda'$ such that the following is true: if $x \in \Lambda'$ and $x = \sum x_i e_i = \sum x'_i e'_i$, then $\max_i |x'_i| \leq 2^m \max_i |x_i|$.
\end{lemma}
\begin{proof}
By the existence of Hermitian normal form, $\Lambda'$ has a basis $(e'_i)_{i = 1}^m$ in which
\[ e'_i = d_i e_i + \sum_{j > i} b_{i,j} e_j,\] with $d_1,\dots, d_m \geq 1$ integers and $0 \leq b_{i,j} < d_j$.
It follows that 
\[ x_i = x'_i d_i + \sum_{j < i} b_{j,i} x'_j.\]
Suppose that  $|x_i| \leq M$ for all $i$. Then an easy induction confirms that $|x'_i| \leq 2^{i-1} M$, and the result follows. (For example, 
\[ |x'_2| =  \big| \frac{x_2}{d_2} - \frac{b_{1,2} x'_1}{d_2} \big| \leq |x_2| + 2^0 M \leq 2^1 M.)\]
\end{proof}

\begin{lemma}\label{quant-lattice}
Let $Q \geq 1$ be a parameter. Let $V \leq \Q^D$ be a vector subspace, spanned over $\Q$ by linearly independent vectors $v_1,\dots, v_m \in \Z^D$ with $|v_i| \leq Q$ for all $i$. Then there is an integral basis $w_1,\dots, w_m$ for $V \cap \Z^D$ such that every element $x \in V \cap \Z^D$ with $|x| \leq Q$ is a (unique) $\Z$-linear combination $x = \sum_{i = 1}^m n_i w_i$ with $|n_i| \leq m!(2Q)^m$.
\end{lemma}
\begin{proof}
Write $\Lambda' := V \cap \Z^D$. Consider the $m$-by-$D$ matrix whose $(i,j)$-entry is the $j$th coordinate $v^{(i)}_j$. This has full rank $m$, and so it has a nonsingular $m$-by-$m$ minor. Relabelling, we may suppose that this is $(v^{(i)}_j)_{1 \leq i, j \leq m}$. Suppose now that $x = \sum_{i=1}^m \lambda_i v^{(i)}  \in V \cap \Z^D$, where the $\lambda_i$ lie in $\Q$. Then the $\lambda_i$ may be recovered by applying the inverse of $(v^{(i)}_j)_{1 \leq i, j \leq m}$ to $(x_1,\dots, x_m)$. By the formula for the inverse in terms of the adjugate, this inverse has entries in the set $\{ \frac{a}{q} : a \in \Z, |a| \leq (m-1)! Q^{m-1}\}$ where $q = \det( (v^{(i)}_j))_{1 \leq i ,j \leq m}$, and so $q\lambda_i$ is an integer of size at most $m! Q^m$.

Take $e_i := \frac{1}{q} v^{(i)}$, and let $\Lambda$ be the lattice generated by the $e_i$. We have shown that $\Lambda' \leq \Lambda$, and moreover that if $x \in \Lambda'$ and $|x| \leq Q$ then $x = \sum_{i} x_i e_i$ with $|x_i|  \leq m! Q^m$. Applying Lemma \ref{hermite}, the result follows.
\end{proof}

\begin{lemma}\label{lattice-subspace}
Let $Q \geq 1$. Let $n$ be sufficiently large and let $V \leq \R^n$ be a subspace of dimension $m$.  Then, uniformly in $V$, $\# \{ x \in \Z^n : |x| \leq Q, x \in V \} \leq 20^n m^{n/2} Q^m$.
\end{lemma}
\begin{proof}
Let $S := \{ x \in \Z^n : |x| \leq Q, x \in V \}$. Pick an orthonormal basis $v_1,\dots, v_n$ for $\R^n$ with $v_1,\dots, v_m$ being a basis for $V$. For each $x \in \R^n$, consider the rotated box
\[ R(x) := \{ x + \sum_{i = 1}^m c_i v_i + Q\sum_{i = m+1}^n c_i v_i : x \in S, |c_i| < \frac{1}{2}m^{-1/2}\; \mbox{for all $i$}\}.\] We claim that for distinct $x, x' \in S$, $R(x)$ and $R(x')$ are disjoint. Indeed if not we would have (since $x, x' \in V$) $x + \sum_{i=1}^m c_i v_i = x' + \sum_{i = 1}^m c'_i v_i$ and hence by orthogonality $\Vert x - x' \Vert_2^2 = \sum_{i = 1}^m |c_i - c'_i|^2 < 1$, which is a contradiction since $x - x' \in \Z^D$.

Now the volume of $R(x)$ is $m^{-n/2} Q^{n-m}$, and if $y \in R(x)$ then
\[\Vert y \Vert_2  \leq \Vert x \Vert_2 + \Vert \sum_{i = 1}^m c_i v_i + Q\sum_{i = m+1}^n c_i v_i \Vert_2  \leq 2n^{1/2} Q.\]
Using a crude upper bound of $(100/n)^{n/2}$ for the volume of the unit ball in $\R^n$, the volume of this set is at most $20^n Q^n$.  The result follows. 
\end{proof}

\section{Smooth bump functions}

In this appendix we give constructions of the various cutoff functions used in the main body of the paper. We begin with cutoffs with compact support in physical space. These are all variants of the classical Fej\'er kernel construction, sometimes with an extra convolution to create more smoothing.

\begin{lemma}\label{w-tent}
Let $\eta > 0$. Then there is a continuously differentiable function $w : \R \rightarrow [0,\infty)$ with the following properties.
\begin{enumerate}
\item $w$ is supported on $[0,1]$, $w = 1$ on $[\eta, 1 - \eta]$, and $0 \leq w \leq 1$ everywhere;
\item $\Vert w' \Vert_{\infty} \ll \eta^{-1}$;
\item $\Vert \hat{w} \Vert_1 \ll \eta^{-1}$.
\end{enumerate}
\end{lemma}
\begin{proof}
This is a standard kind of ``tent'' function.  Take 
\[ w := \frac{4}{\eta^2} 1_{[\eta/2, 1 - \eta/2]} \ast 1_{[-\eta/4, \eta/4]} \ast 1_{[-\eta/4, \eta/4]}.\]
It is straightforward to see that this has the relevant support properties (1). For (2), on the intervals where $w$ is not constant it is of the form $\psi(\frac{x+a}{2\eta})$, where $\psi$ is the triple convolution $1_{[-1/2,1/2]} \ast 1_{[-1/2,1/2]} \ast 1_{[-1/2,1/2]}$. It is well-known that such a triple convolution is continuously differentiable; one mode of argument is via the Fourier transform, noting that $|\hat{\psi}(\xi)| \ll |\xi|^{-3}$, so one gets a convergent integral by differentiating $\psi(x) = \int \hat{\psi}(\xi) e(\xi \cdot x) d\xi$ under the integral. Alternatively, one can work entirely in physical space.

(3) By performing the integrals explicitly,
\[ \hat{w}(\xi) = \frac{4}{\eta^2} \hat{1}_{[\eta/2, 1 - \eta/2]}(\xi) \hat{1}_{[-\eta/4, \eta/4]}(\xi)^2 \ll \eta^{-2} \min(\eta, |\xi|^{-1})^2.\]
Now consider the contributions from $|\xi| \leq \eta^{-1}$ and $|\xi| \geq \eta^{-1}$ separately.
\end{proof}

\begin{lemma}\label{fejer} Let $X \geq 1$. There is a function $w : \Z \rightarrow [0,\infty)$ such that
\begin{enumerate}
\item $w$ is supported on $[-X/5, X/5]$;
\item $\hat{w} : \T \rightarrow \C$ is real and non-negative;
\item $\sum_n w(n) \geq X$;
\item $|\hat{w}(\beta)| \leq 2^5 X^{-1} \Vert \beta \Vert_{\T}^{-2}$ for all $\beta \in \T$.
\end{enumerate}
\end{lemma}
\begin{proof}
This is a standard Fej\'er kernel construction. Take 
\[ w := \frac{25}{X}1_{[-X/10, X/10]} \ast 1_{[-X/10, X/10]}(n).\] Then $w$ is immediately seen to satisfy (1), (2) and (3). For (4), we evaluate the Fourier transform explicitly as 
\[ \hat{w}(\beta) = \frac{25}{X}|\sum_{|n| \leq X/10} e(-\beta n)|^2.\] The sum here is a geometric progression; summing it, we obtain
\[ |\sum_{|n| \leq X/10} e(-\beta n)| \leq \frac{2}{|1 - e(\beta n)|} = \frac{1}{|\sin \pi \beta|} \leq \Vert \beta \Vert_{\T}^{-1}.\] The result follows.
\end{proof}

\begin{lemma}\label{chi-15-exist} Let $\delta \in (0,1)$. Then there is $\chi := \chi_{\delta} : \R \rightarrow [0,\infty)$ satisfying the following properties, where the implied constants are absolute and do not depend on $\delta$.
\begin{enumerate}
\item $\chi(x) \geq 1$ for $|x| \leq \delta/2$;
\item $\chi(x) = 0$ for $|x| > \delta$;
\item $\int \chi \ll \delta$;
\item $\Vert \hat{\chi} \Vert_1 \ll 1$;
\item $\int_{|\xi| > \delta^{-2}} |\hat{\chi}(\xi)| \ll \delta^{2}$.
\end{enumerate}
\end{lemma}
\begin{proof}
It suffices to construct a function $\psi : \R \rightarrow [0,\infty)$ satisfying
\begin{enumerate}
\item $\psi(x) \geq 1$ for $|x| \leq 1/2$;
\item $\psi(x) = 0$ for $|x| > 1$;
\item $\int \psi \ll 1$;
\item $\Vert \hat{\psi} \Vert_1 \ll 1$;
\item $\int_{|\xi| > X} |\hat{\psi}(\xi)| \ll X^{-2}$ for $X \geq 1$.
\end{enumerate}
Then one may take $\chi_{\delta} := \psi(\delta^{-1} x)$, and the five properties of $\chi$ in the lemma follow from the five properties of $\psi$ just stated using $\hat{\chi}_{\delta}(\xi) = \delta \hat{\psi}(\xi \delta)$, and taking $X = \delta^{-1}$. A function with these properties is
\[ \psi(x) := 16 \cdot 1_{[-3/4, 3/4]} \ast 1_{[-1/8, 1/8]} \ast 1_{[-1/8, 1/8]}.\]
Properties (1), (2), (3) are easily checked. For (4) and (5), one may compute the Fourier transform explicitly and thereby obtain the bound $\hat{\psi}(\xi) \ll \min(1, |\xi|^{-3})$, from which (4) and (5) both follow straight away.
\end{proof}

Now we turn to some cutoff functions with compact support in frequency space.

\begin{lemma}\label{band-limited-r}
There is a function $\psi : \R \rightarrow \R$ satisfying the following:
\begin{enumerate}
\item $\psi \geq 0$ everywhere, and $\psi(x) \geq 1$ for $|x| \leq 1$
\item $\hat{\psi}(y) = 0$ for $|y| \geq 1$;
\item $\int \psi \leq 5$.
\end{enumerate}
\end{lemma}
\begin{proof}
Take $\psi(x) :=  \frac{\sin^2 x}{x^2 \sin^2 1}$.
Then (1) is immediate. For (2) and (3), observe the Fourier transform $\int^{\infty}_{-\infty} \frac{\sin^2 x}{x^2} e(-\xi x)dx = \pi (1 - \pi |\xi|)_+$, and then finally use $\pi < 5 \sin^2 1$.
\end{proof}

\begin{lemma}\label{cutoff-for-dio}
There is a smooth cutoff $\chi : \T^D \rightarrow [0,\infty)$ satisfying
\begin{enumerate}
\item $\chi(x) \geq 1$ for $\Vert x \Vert_{\T^D} \leq X^{-1/D}$;
\item $\int \chi \leq 5^D X^{-1}$;
\item $\hat{\chi}(\xi) = 0$ for $|\xi| \geq X^{1/D}$; 
\end{enumerate}
\end{lemma}
\begin{proof}
It suffices to prove the following $1$-dimensional result: given $\eps > 0$, there is $\phi = \phi_{\eps}: \T \rightarrow [0,\infty)$ such that 
\begin{enumerate}
\item $\phi(x) \geq 1$ for $\Vert x \Vert_{\T} \leq \eps$;
\item $\int \phi \leq 5 \eps$;
\item $\hat{\phi}(\xi) = 0$ for $|\xi| \geq 1/\eps$, $\xi \in \Z$.
\end{enumerate}
Indeed, one may then take $\eps = X^{-1/D}$ and $\chi(x_1,\dots, x_D) = \prod_{i=1}^D \phi(x_i)$ to satisfy the desiderata of the lemma.

It remains to construct $\phi$. With a slight abuse of notation, we construct $\phi$ as a $1$-periodic real function rather than a function on $\T$ (they are, of course, basically the same thing). To do this, take the function $\psi$ constructed in Lemma \ref{band-limited-r}, and set.
\[ \phi(x) := \sum_{n \in \Z} \psi(\eps^{-1}(x + n)).\]  Then $\phi$ is a smooth function on $\T$ taking non-negative values, and it satisfies (1) above (just by taking the term $n = 0$ in the sum).
By unfolding the sum we have
\[ \int_{\T} \phi(x) = \int_{t \in \R} \psi(\eps^{-1} t) dt \leq 5 \eps, \] by change of variables and Lemma \ref{band-limited-r} (3).
Turning to the Fourier transform, if $\xi \in \Z$ then 
\begin{align*} \hat{\phi}(\xi) & = \int^1_0 \sum_{n \in \Z} \psi(\eps^{-1}(x + n))e(-\xi x) \\ & = \int^1_0 \sum_{n \in \Z} \psi(\eps^{-1}(x + n))e(-\xi (x + n)) \\ & = \int_{\R} \psi (\eps^{-1} t) e(- \xi t) dt  = \eps \hat{\psi}(\xi \eps). \end{align*}
By Lemma \ref{band-limited-r} (2), this vanishes when $|\xi| \geq 1/\eps$.

\end{proof}

Finally, for use in Section \ref{first-step-sec} we construct a minorant function with compactly supported, nonnegative Fourier transform. For the construction we use a trick shown to me in a different context by Joni Ter\"av\"ainen.

\begin{lemma}\label{tera}
Let $D$ be sufficiently large and suppose that $\rho \leq D^{-4}$.  There is a function $\chi : \T^D \rightarrow \R$ such that
\begin{enumerate}
\item $\chi(x) \leq 0$ unless $x \in \pi(B_{\rho/10}(0))$, where $B_{\eps}(0) \subset \R^D$ is the Euclidean ball of radius $\eps$ and $\pi : \R^D \rightarrow \T^D$ is the natural projection;
\item $\hat{\chi}(\xi) \geq 0$ for all $\xi \in \Z^D$;
\item $\hat{\chi}$ is supported on $|\xi| \leq \rho^{-3}$;
\item $\int \chi  = 1$.
\item $\int |\chi| \leq 3$.
\end{enumerate}
\end{lemma}
\begin{proof}
Rather than working with Euclidean balls and the $\ell^2$-norm, it is easier to work with the distance $\Vert \cdot \Vert_{\T^D}$ directly. Note that $D^{1/2}\Vert \pi(t) \Vert_{\T^D} \geq \Vert t \Vert_2$, so it suffices to replace (1) by the stronger condition (1') that $\chi(x) \leq 0$ outside of the box $\Vert x \Vert_{\T^D} \leq \rho^{7/6} < D^{-1/2}\rho/10$.

Let $k = \lfloor \rho^{-3} \rfloor$. Consider
\begin{align*} \psi(x) & := (2D + \sum_{i=1}^D(e(x_i) + e(-x_i)) )^k - 4^k (D - \rho^{7/3})^k \\ & 
=  4^k (\cos^2(\pi x_1) + \dots + \cos^2(\pi x_D))^{k} - 4^k(D - \rho^{7/3})^{k} .\end{align*}
Since $\cos^2 (\pi t) \leq 1 - t^2$ for $|t| \leq \frac{1}{2}$, if $\Vert x \Vert_{\T^D} > \rho^{7/6}$ then we have
\begin{equation}\label{x1} 0 \leq \cos^2 (\pi x_1) + \cdots + \cos^2 (\pi x_D) \leq D - \rho^{7/3}\end{equation} and so $\psi(x) \leq 0$. It is clear by expanding out the definition that $\hat{\psi}(\xi)$ is supported on $|\xi| \leq k$ (in fact on $\Vert \xi \Vert_1 \leq k$), and also that $\hat{\psi}(\xi) \geq 0$ except possibly at $\xi = 0$.  

To get a lower bound for $\int \psi$, we use the inequality $\cos^2 \pi t \geq 1 - \pi^2t^2$, which is valid for $|t| \leq \frac{1}{2}$, to conclude that if $\Vert x_i \Vert_{\T}  \leq 1/4\sqrt{k}$ for all $i$ then
\[(\cos^2 (\pi x_1) + \dots + \cos^2(\pi x_D))^k \geq D^k (1  - \frac{1}{k})^k \geq \frac{1}{3} D^k.  \] Therefore
\begin{equation}\label{lower-cos-2} \int_{\T^D} (\cos^2(\pi x_1) + \dots + \cos^2(\pi x_D))^k \geq \frac{1}{3}(2 k^{-1/2})^D D^k > 2 k^{-D} D^k.\end{equation}

By contrast, using $k = \lfloor \rho^{-3}\rfloor$, $\rho \leq D^{-4}$ and the fact that $D$ is large, we see that  \begin{equation}\label{substantial} (D - \rho^{7/3})^{k} \leq D^k e^{-\rho^{7/3} k/D} <  k^{-D} D^k .\end{equation}
Therefore, comparing with \eqref{lower-cos-2}, we see that $\int \psi >k^{-D}(4D)^k$.

Now define $\chi := (\int \psi)^{-1}\psi$. Then, from what has been said above, (1'), (2), (3) and (4) all hold. 

It remains to establish (5). For this write $\psi = \psi_+ - \psi_-$ in positive and negative parts, and note that $\psi_- \leq 4^k (D - \rho^{7/3})^k$ pointwise. By \eqref{substantial} it follows that $\int \psi_- < k^{-D} (4D)^k < \int \psi$.  Since $|\psi| = \psi + 2 \psi_-$, it follows that $\int |\psi| \leq 3 \int \psi$, and (5) follows immediately.

\end{proof}

\end{document}